\documentclass{amsart}
\usepackage[dvips]{epsfig}
\usepackage{graphicx}
\usepackage{latexsym}
\usepackage{amsmath}
\usepackage{amsthm}
\usepackage{amssymb}
\usepackage{ulem}
\RequirePackage[colorlinks=true,citecolor=blue,urlcolor=blue]{hyperref}

\usepackage{eucal}
\usepackage{float}
\usepackage{xcolor}
\usepackage{multirow} % AO, 2 septembre = j'ajoute un package, attention !

 \usepackage[applemac]{inputenc}

\newtheorem{thm}{Theorem}
\newtheorem{assumption}[thm]{Assumption}
\newtheorem{lem}[thm]{Lemma}

\newtheorem{defi}[thm]{Definition}

\newtheorem{prop}[thm]{Proposition}
\newtheorem{rk}[thm]{Remark}

\newcommand{\field}[1]{\mathbb{#1}}

\newcommand{\E}{\field{E}}

\newcommand{\PP}{\field{P}}

\newcommand{\R}{\field{R}}

\newcommand{\vip}{\vskip.2cm}

\begin{document}

\title[Birth and death models in a large population limit]{Nonparametric adaptive inference of birth and death models in a large population limit}

\author{Alexandre Boumezoued, Marc Hoffmann and Paulien Jeunesse}

\address{A. Boumezoued, Milliman R\&D, 14 Avenue de la Grande Arm\'ee, 75017 Paris, France.}

\email{alexandre.boumezoued@milliman.com}

\address{Marc Hoffmann, Universit\'e Paris-Dauphine \& PSL, CNRS, CEREMADE, 75016 Paris, France}
%Ceremade, Universit\'e Paris-Dauphine PSL, 75016 Paris, France. (to whom correspondence should be addressed.)}

\email{hoffmann@ceremade.dauphine.fr}

\address{Paulien Jeunesse, Universit\'e Paris-Dauphine \& PSL, CNRS, CEREMADE, 75016 Paris, France}
%Ceremade, Universit\'e Paris-Dauphine PSL, 75016 Paris, France.}

\email{jeunesse@ceremade.dauphine.fr}

\begin{abstract} 
 Motivated by improving mortality tables from human demography databases, we investigate statistical
inference of a stochastic age-evolving density of a population alimented by time inhomogeneous mortality
and fertility. Asymptotics are taken as the size of the population grows within a limited time horizon: the
observation gets closer to the solution of the Von Foerster Mc Kendrick
equation, and the difficulty lies in controlling simultaneously the stochastic approximation to the limiting
PDE in a suitable sense together with an appropriate parametrisation of the anisotropic solution. In this
setting, we prove new concentration inequalities that enable us to implement the Goldenshluger-Lepski
algorithm and derive oracle inequalities. We obtain minimax optimality and adaptation over a wide range of anisotropic H\" older smoothness classes.
\end{abstract}

\maketitle



\tableofcontents

\section{Introduction}
\subsection{Setting} \label{sec: setting}
Suppose one wishes to recover a probability density $g$ over the nonnegative real line $\R_+=[0,\infty)$ from a $N$-sample $a_1,\ldots, a_N$, where the $a_i$ are not necessarily independent.
% denotes a $N$-sample of a real-valued distribution $\xi(da) = g(a)da$ with density $g$, encoded by its empirical distribution 
If $Z^N = N^{-1}\sum_{i = 1}^N\delta_{a_i}$ denotes the empirical distribution of the $N$-sample, designing a good statistical estimator of $g$ requires a fine quantitative control of the fluctuations in the convergence 
\begin{equation} \label{eq weak basic}
\int_{\R_+}\psi(a)Z^N(da) \rightarrow \int_0^\infty \psi(a)g(a)da
\end{equation}
(at least in probability) as $N$ grows, for a large enough class of test functions $\psi$.
Moreover, the performance of such a procedure depends on the smoothness properties of the function $g$, typically quantified by a smoothness parameter, like a (possibly fractional) number of derivatives in any reasonable sense and is usually unknown by the practitioner. For suitable $\psi$ (possibly data-dependent), optimal estimators can be found provided good concentration inequalities are available for \eqref{eq weak basic}, following the broad guiding principle of Lepski's method \cite{LEPSKIBASIC, GOLDENSHLUGERLEPSKI1, GOLDENSHLUGERLEPSKI2} or other adaptive methods like model selection or wavelets, see for instance the comprehensive textbooks of Gin\'e and Nickl \cite{GINENICKL} or H\" ardle {\it et al.} \cite{HKPT} or Tsybakov \cite{TSYBAKOV}. In this paper, we generalise the classical situation described above by adding a time variable. We investigate statistical inference of a time-evolving particle system governed by stochastic dynamics: for every $t \in [0,T]$, we observe the state of a population of (approximately) $N$ particles, encoded by its empirical measure $Z^N = \big(Z_t^N(da)\big)_{0 \leq t \leq T}$. Informally, $Z^N$ is solution to a certain stochastic differential equation (SDE)
$$\mathcal H_{b,\mu}^N\big(Z^N\big) = 0,$$
constructed in \eqref{eq micro} below; $\mathcal H_{b,\mu}^N$ is parametrised by two functions $b$ and $\mu$ and $Z_t^N(da)$ represents the state of a population structured in age $a \in \R_+$, alimented by a time-inhomogeneous fertility rate $b(t,a)$ and decimated by a mortality rate $\mu(t,a)$.  Moreover, we are given an initial empirical age distribution $Z_0^N$ at time $t=0$. Under appropriate regularity conditions
% on $b$ and $\mu$ and if $Z_0^N$ is close enough to an initial limiting distribution $g_0$, 
we have a convergence $\mathcal H_{b,\mu}^N \rightarrow \mathcal H_{b,\mu}$ in a large population limit $N \rightarrow \infty$, where
$\mathcal H_{b,\mu}(g)= 0$ is an inhomogeneous version of the classical McKendrick Von Foerster renewal equation \cite{MCKENDRICK, VONFOERSTER}, given by
\begin{align} \label{McKendrick}
\left\{
\begin{array}{ll}
\partial_t g(t,a)+\partial_ag(t,a) + \mu(t,a)g(t,a) = 0 \\ \\
\displaystyle g(0,a) = g_0(a),\;\; g(t,0) = \int_0^\infty b(t,a)g(t,a)da. \\ 
%g(t,0) = \int_{\mathbb R_+} b(t,a)g(t,a)da.
\end{array}
\right.
\end{align}
%, where 
%\begin{equation} \label{eq limit weak basic dynamic}
%\mathcal H_{b,\mu}(g)=0,
%\end{equation}
and 
that reveals the interplay between the limiting solution $g$ and the model parameter $b$ and $\mu$.  In particular, we have a convergence 
\begin{equation} \label{eq weak basic dynamic}
\int_{0}^T \int_{\R_+} \psi(t,a) Z_t^N(da)dt \rightarrow \int_0^T \int_0^\infty \psi(t,a)g(t,a)da\,dt
\end{equation}
(at least in probability) as $N$ grows for a rich enough class of functions $\psi$, and this situation generalises \eqref{eq weak basic} in a time-dependent framework.\\

Informally, our statistical problem takes the following form: estimate $g$ or the parameters of the model $b,\mu$ from data $Z^N$ in the limit $N \rightarrow \infty$. In this setting, it is crucial to understand: (i) the quantitative properties of the convergence \eqref{eq weak basic dynamic} and in particular, how concentration inequalities can be obtained (with a view towards an adaptive estimation scheme in the idea of Lepski's principle) and (ii) what is the structure of the equation $\mathcal H_{b,\mu}(g)= 0$  in terms of identification and interplay between the parameters $b,\mu, g$ and their smoothness properties. In particular, the anisotropic smoothness of $g$ viewed as a graph-manifold can benefit from the structure $\mathcal H_{b,\mu}(g)=0$ and lead to better approximation properties in certain directions along the characteristics of the transport.

\subsection{Motivation}

%\begin{center}
%{\color{blue} {\tt [links with Alexandre, to be completed!!]}}
%\end{center}
Of primary interest for us 
%this study is inspired by 
is human demography through the recent efforts and contributions for improving mortality estimates, see \cite{CAIRNS2016, BOUMEZOUED2016, BHJ2018APP} among others and the references therein. In particular, the recent development of large human datasets like the Human Mortality Database (HMD) and Human Fertility Database (HFD) \cite{HMD, HFD} -- in open access -- allows one to process fertility and mortality data simultaneously, and subsequently addresses demographical issues such as the anomalies of cohort effects that have long fascinated demographers and actuaries \cite{RICHARDS2008, CAIRNS1}. In this rejuvenated context, it becomes reasonable to study the estimation of population density or mortality rate in the enriched dynamical framework provided by birth-death particle systems that converge to the classical McKendrick Von Foerster equation in a large population limit, and revisit classical studies like {\it e.g.} \cite{KEIDING1990, NIELSEN1} for statistical estimation of the death rate; see the detailed literature review in next section. In this setting, we consider the idealised model where we can observe the (renormalised) evolution of the state of the population $Z^N_t$ continuously for $t\in [0,T]$, where $t=0$ is the starting date for the observation of the population and $t=T$ a terminal time horizon, fixed once for all. We are interested in identifying or estimating the parameters of the model. Of major importance is the inhomogeneous death rate $\mu(t,a)$. In our framework, we cannot recover the birth rate since we are not given any genealogical input: mathematically, this simply expresses the lack of injectivity of the mapping $b \mapsto g$. Still, our observation enables us to identify the functions $(t,a) \mapsto g(t,a)$ and $(t,a) \mapsto \mu(t,a)$ in the limit $N \rightarrow \infty$ and establish a thorough nonparametric estimation program, in the methodology of adaptive minimax estimation.

%assume we observe the evolution of the 

\subsection{Link with literature on death rate inference}

The main difficulty in establishing a consistent theory to estimate mortality rates comes from two key points:
%\begin{enumerate}
%\item[-] 
(i) incorporate the fact that the death rate depends on both age and  time (non-homogeneous setting)
%\item[-] 
and (ii) use as observables the outcome of a stochastic population dynamics (birth-death process).
%\end{enumerate}
In the literature, we argue that each point is treated separately. The inference of a time-dependent death rate also related to a time-dependent covariate (possibly age), which relates to the first point has been addressed from a nonparametric perspective by e.g. \cite{BERAN1981, DABROWSKA1987, KEIDING1990, MCKEAGUE1990, NIELSEN1995, BRUNEL2008, COMTE2011} and the references therein. 
From \cite{KEIDING1990}, {\small \it 
"One way of understanding the difficulties in establishing an Aalen theory in the
Lexis diagram is that although the diagram is two-dimensional, all movements are
in the same direction (slope 1) and in the fully non-parametric model the diagram
disintegrates into a continuum of life lines of slope 1 with freely varying intensities
across lines. The cumulation trick from Aalen's estimator (generalizing ordinary
empirical distribution functions and Kaplan \& Meier's (1958) nonparametric
empirical distribution function from censored data) does not help us here." 
}
On the other side, the inference of an age-dependent death rate in an homogeneous birth-death model (or similar) - oiuyr second point - has been addressed in \cite{TRAN2008, DHKR1, HOFFMANN2016} among others. To the best of our knowledge, no statistical method deals with the usual problem faced by demographers related to the inference of a time and age-dependent death rate table based on the observation of population dynamics. Note that in this paper, the observation of the population is assumed to be continuous over time, whereas in practice the information on population exposure is extracted from census (point observation); these practical considerations are discussed in a companion paper, see \cite{BHJ2018APP}.

\subsection{Results and organisation of the paper}

In a first part of the paper, Section \ref{sec stability}, we construct the SDE that describes the state of the population $Z^N$ by means of a birth-death process characterised via a stochastic differential equation -- given in \eqref{eq micro} -- driven by a random Poisson measure. We recall its convergence in a large population limit to the solution of the McKendrick Von Foerster equation $g$ based on classical results of \cite{TRAN1, TRANMELEARD}. Our next step consists in quantifying the stability of the convergence $Z^N \rightarrow g$. To that end and anticipating the subsequent statistical analysis, we introduce two pseudo-distances:
$$\mathcal W_{w_2}^N(\mathcal F)_t = \sup_{f \in \mathcal F}\int_{\R_+} w_2(t-a) f(t,a)\big(Z^N_t(da)-g(t,a)da\big)$$
and its integrated version
%Under Assumptions \ref{H basic}, \ref{H stability} and \ref{H initial limit}, we define
$$\mathcal W_{w_1, w_2}^N(\mathcal F)_t = \sup_{f \in \mathcal F} \int_0^t w_1(s)\int_{\R_+} w_2(s-a) f(s,a)\big(Z^N_s(da)-g(s,a)da\big)ds,$$
where $w_1$ and $w_2$ are two bounded weight functions and $\mathcal F$ a rich enough class of function with complexity measured in terms of entropy conditions. Note that formally $\mathcal W_{w_2}^N(\mathcal F)_t$ is a degenerate version of $\mathcal W_{w_1, w_2}^N(\mathcal F)_t$. Taking $w_1=w_2=1$ is reminiscent of the 1-Wasserstein distance if $\mathcal F$ consists of $1$-Lipschitz functions. However, for the statistical analysis, we must be able to handle approximating kernels that do not have bounded Lipschitz norms, hence the presence of the weights $w_1$ and $w_2$ that can accomodate such kernels. 
The main result of this section, Theorem \ref{thm concentration optimal} states that under appropriate regularity conditions on $b$ and $\mu$, if $|w_2|_{1,\infty}^{-1}\mathcal W_{w_2}^N(\mathcal F)_0$ is of (small) order $r_N$, so are
$$|w_2|_{1,\infty}^{-1}\mathcal W_{w_2}^N(\mathcal F)_T\;\;\;\text{and}\;\;\;\big(|w_1|_{1,\infty}|w_2|_{1,\infty}\big)^{-1}\mathcal W_{\,w_1,w_2}^N(\mathcal F)_T.$$
The rate of decay $r_N$  possibly inflates by an order $N^{-1/2}$ and the result holds in terms of exponential decay of the fluctuation probabilities. The functional control $|\cdot|_{1,\infty} = (\|\cdot\|_{L^1}\|\cdot\|_{L^\infty})^{1/2}$ interpolates between $L^1$ and $L^\infty$-norms, and is sufficient to handle the behaviour of statistical kernels in an optimal way, since it can therefore be compared to the usual $L^2$-norm that appears in variance terms. The concentration of $\mathcal W_{w_2}^N(\mathcal F)_T$ expresses a kind of stability of the particle system from $t=0$ to $t=T$, while the more intricate control of $\mathcal W_{\,w_1,w_2}^N(\mathcal F)_T$ is crucial to control variance terms in bi-variate kernel estimators for the nonparametric estimation of $g(t,a)$ and $\mu(t,a)$. The proof relies on a combination of martingales techniques in the spirit of Tran \cite{TRAN1}, a central reference for the paper, combined with classical tools from concentration of processes indexed by functions under entropy controls, following for instance Ledoux-Talagrand \cite{LEDOUXTALAGRAND}.\\

In a second part, Section \ref{sec stat oracle}, we construct nonparametric estimators of $g(t,a)$ and $\mu(t,a)$ by means of kernel approximation: we consider estimators of the form
$$\widehat g_h^N(t,a) = K_h \star Z_t^N(a)$$
for $g(t,a)$, where $\star$ denotes convolution and $K_h = h^{-1}K(h^{-1}\cdot)$, with $|K|_1=1$, is a kernel normalised in $L^1$ with bandwidth $h>0$. It is noteworthy that for estimating the population density $g(t,a)$ at time $t$, the information $Z_t^N$ is sufficient and we do not need the data $(Z_s^N, s\neq t)$. The situation is very different for estimating $\mu(t,a)$ the main parameter of interest. We constuct a quotient estimator, inspired from a Nadaraya-Watson type procedure, and use
\begin{equation} \label{eq def info esti mu}
\widehat \mu_{h_1,h_2,h_3}^N(t,a) = \frac{(H_{h_1} \otimes K_{h_2}\circ \varphi )\star\Gamma^N(du,ds)}{\widehat g_{h_3}^N(t,a)}
\end{equation}
where $\Gamma^N(du,ds)$ is the point process of the death occurences in the population lifetime that can be extracted from $Z^N$ and that converges to $\pi = \mu g$, see \eqref{def death jumps} in Section \ref{sec constrution estimators} for the details. In \eqref{eq def info esti mu}, we consider a bivariate kernel $H\otimes K$ with bandwidth $(h_1,h_2)$ and $\varphi(t,a) = (t,t-a)$ is a certain change of coordinates that enables one to benefit from the smoothness along the characteristics of the transport. The choice of the bandwidths $h_1,h_2,h_3$ is chosen according to the data $Z^N$ itself, in the spirit of Lepski's principle \cite{GOLDENSHLUGERLEPSKI1, GOLDENSHLUGERLEPSKI2}. In Theorems \ref{thm oracle g} and \ref{thm oracle mu}, we derive oracle inequalities that control the pointwise risk of $\widehat g_h^N(t,a)$ and $\widehat \mu_{h_1,h_2,h_3}^N(t,a)$ in terms of optimal balance between the error propagation of Theorem \ref{thm concentration optimal} and the linear approximation kernels.\\ 

Section \ref{sec: minimax adaptive} is devoted the adaptive estimation of $g$ and $\mu$ for the pointwise risk under smoothness constraints. In a first part, we study the smoothness of $g$ when $b$ and $\mu$
%$b \in \mathcal H^{\alpha,\beta}$ and $\mu \in \mathcal H^{\gamma,\delta}$ 
belong to anisotropic H\"older spaces (and for simplicity, we assume that the initial condition $g_0$ is sufficiently smooth). Thanks to the relatively explicit form of the solution of the McKendrick Von Foester equation,  we establish in Proposition \ref{H smoothness basic} that when parametrised via $\varphi$, the function $\widetilde g$ in the representation $g = \widetilde g \circ \varphi$ has explicitly quantifiable improved smoothness over $g$, suggesting to consider the approximation kernel $H_{h_1} \otimes K_{h_2}\circ \varphi$ for estimating $\pi$ via the quotient estimator  \eqref{eq def info esti mu} that implicitly uses the representation of $\mu = \pi/g$. We establish in Theorem \ref{thm LB} minimax lower bounds for estimating $g(t,a)$ and $\mu(t,a)$ and prove in Theorems \ref{adapt esti g} and \ref{th adapt minimax mu} that these bounds are optimal in some cases, thanks to the oracle inequalities established Theorems \ref{thm oracle g} and \ref{thm oracle mu}. In particular, we achieve minimax adaptation over anisotropic H\"older smoothness constraints, up to poly-logarithmic terms.\\

The techniques developed in this paper have at least two possible lines of extensions for considering more general models than \eqref{McKendrick}: (i) first, when we replace the constant transport by an arbitrary aging function solution to $dX_t = v(X_t)dt$ if $X_t$ denotes the age evolution of an individual, and (ii) if we allow for interacting particle system in the following sense: we replace $\mu(t,a)$ by a population dependent mortality rate $\widetilde \mu(t,a) + \int_{\R_+} U(a,a')Z^N_s(da')$, as already studied for instance by Tran \cite{TRAN2008} for some baseline mortality rate $\widetilde \mu(t,a)$ affected in a mean-field sense by a kernel $U(a,a')$. Under appropriate regularity assumptions, the limiting model takes the form

\begin{align*} 
\left\{
\begin{array}{ll}
\partial_t g(t,a)+\partial_a\big(v(a)g(t,a)\big) + \big(\widetilde \mu(t,a)+\int_{\R_+}U(a,a')g(t,a')da'\big)g(t,a) = 0 \\ \\
\displaystyle g(0,a) = g_0(a),\;\; g(t,0) = \int_0^\infty b(t,a)g(t,a)da. \\ 
%g(t,0) = \int_{\mathbb R_+} b(t,a)g(t,a)da.
\end{array}
\right.
\end{align*}

We intend to describe the extension to this situation in a forthcoming work. Sections \ref{sec proofs concentration} is devoted to the proof of the main concentration result of Theorem \ref{thm concentration optimal} and auxiliary stability results of Section \ref{sec stability}. In Section \ref{sec proof stat}, we give the proofs of the statistical results of  \ref{sec stat oracle} and \ref{sec: minimax adaptive}. The Appendix Section \ref{sec appendix} contains some useful technical and auxiliary results. 

%

%\begin{itemize}
%\item[-] Under a suitable approximation $Z_0^N \approx g_0 \leadsto$ {\color{blue} identification} of $g_0$. 
%\item[-] Need to understand how $Z_0^N \approx g_0$ propagates to $Z_t^N \approx g(t,\cdot)$ for $t \in [0,T]$.
%\item[-] Under ``suitable  propagation", we can identify $g$ from $Z^N$.
%% via a kernel estimator $\widehat g_N$ 
%\item[-] Likewise, we can identify $\mu$ from $Z^N$. 
%\item[-] We cannot identify $b$ from $Z_N$ for lack of injectivity of $b \mapsto g$.
%\end{itemize}

%\begin{center}
%{\color{blue} {\tt [Review or the results \& techniques, to be completed!!]}}
%\end{center}
%
%
%\subsection{Organisation of the paper}
%
%\begin{center}
%{\color{blue} {\tt [standard, to be completed!!]}}
%\end{center}

\section{The microscopic model and its large population limit} \label{sec stability}

\subsection{Notation}

\subsubsection*{The function spaces} We fix once for all a terminal time $T>0$ and $\mathcal D = [0,T] \times \R_+$.
We  work with the set of (measurable) functions 
%\begin{align*}
$$
\mathcal L_{\mathcal D}^\infty   = \big\{f: \mathcal D \rightarrow \R,\; \sup_{t,a}|f(t,a)| < \infty\},
$$
%\end{align*}
implicitly continuated on $\R \times \R$ by setting 
$f(t,a)=0$ for $(t,a) \notin \mathcal D$ 
and also introduce
\begin{align*}
\mathcal L_{\mathcal D}^{{\small\mathrm{time}}}  = \big\{f:[0,T] \rightarrow \R,\; \sup_{t}|f(t)| < \infty\},\;\;
\mathcal L_{\mathcal D}^{{\small\mathrm{age}}}  = \big\{f:\R_+ \rightarrow \R,\; \sup_{a}|f(a)| < \infty\}, 
\end{align*}
with natural embeddings 
$\mathcal L_{\mathcal D}^{{\small\mathrm{time}}} \subset \mathcal L_{\mathcal D}^\infty $ and also $\mathcal L_{\mathcal D}^{{\small\mathrm{age}}} \subset \mathcal L_{\mathcal D}^\infty $ for appropriate arguments. For
%We sometimes write $f_t(a)=f(t,a)$ when no confusion is possible 
$p=1,2$, we set 
%$|f|_p = \big(\int_\mathcal D |f(t,a)|^pdtda\big)^{1/p}$, $|f|_\infty = \sup_{(t,a) \in \mathcal D}|f(t,a)|$ and introduce the interpolation quantity
\begin{equation} \label{eq: inter norme}
|f|_p = \big(\int_\mathcal D |f(t,a)|^pdtda\big)^{1/p},\;\;|f|_\infty = \sup_{(t,a) \in \mathcal D}|f(t,a)|,\;\;|f|_{1,\infty} = \big(|f|_1|f|_\infty\big)^{1/2}.
\end{equation}
We obviously have $|f|_2 \leq |f|_{1,\infty}$, but also the following interesting stability property under dilation: for every $\tau>0$, 
$$|f|_2|\tau^{1/2} f(\tau \cdot)|_{1,\infty} = |\tau^{1/2}f(\tau \cdot)|_2|f|_{1,\infty}.$$
For $0 \leq s \leq 1$, we denote by $\mathcal C^s_{\mathcal D}$ the set of $s$-H\"older continuous functions $f$ on $\mathcal D$ that satisfy
\begin{equation} \label{eq def holder global}
%\mathcal C^s_{\mathcal D} = \big\{f \in\mathcal L^\infty_\mathcal D, 
|f(t,a)-f(t',a')| \leq c(|t-t'|^s+|a-a'|^s)
%\;\;\text{for every}\;\;(t,a),(t',a') \in \mathcal D
%\;\text{for some}\;c>0\big\}
\end{equation}
for every $(t,a),(t',a') \in \mathcal D$ and some $c>0$.
%for the set of $s$-H\"older continuous functions on $\mathcal D$.

\subsubsection*{The random measures}

$\mathcal M_F$ denotes the set of finite point measures on $\R_+=[0,\infty)$ and ${\mathcal M_F}_+$ the set of positive finite measures on $\R_+$. Any $Z \in \mathcal M_F$ admits the representation 
$Z = \sum_{i = 1}^n \delta_{a_i}$ for some ordered set $\{a_1,\ldots, a_n\} \subset \R_+$.
For a real-valued function $f$ defined on 
$\R_+$, we write
$$\langle Z,f \rangle = \int_{\R_+} f(a)Z(da) = \sum_{i = 1}^{n}f(a_i).$$
In particular $n = \langle Z,{\bf 1}\rangle$. For $Z=\sum_{i = 1}^{n} \delta_{a_i}  \in \mathcal M_F$, abusing notation slightly, we define the evaluation maps $a_i(Z)=a_i$ and for $t \geq 0$, the shift $\tau_tZ=\sum_{i = 1}^{n}\delta_{a_i+t}$.

% consists of functions $f \in \mathcal L_{\mathcal D}^\infty $ such that
%$|f(t,a)-f(t',a')| \leq C(|t-t'|^s+|a-a'|^s)$ for some $C>0$.

%for $p=1,2$.
% and $\mathcal L_{\mathcal D}^{{\small\mathrm{time}}} \subset\mathcal L_{\mathcal D}^{{\small\mathrm{\,age}}}$.

\subsection{Construction of the model} \label{sec construction model}
The basic assumptions on the model are the following:

\begin{assumption} \label{H basic}
%We have $\sup_{t,a}(b(t,a)+\mu(t,a))<\infty$. 
We have 
\begin{enumerate}
\item[(i)] $b \in \mathcal L_{\mathcal D}^\infty $ and $\mu \in \mathcal L_{\mathcal D}^\infty $,
% and
%Moreover, for some $\epsilon >0$
%$$\sup_{N \geq 1}\E[\langle Z_0^N,{\bf 1}\rangle^{1+\epsilon}] <\infty\;\;\text{and}\;\;Z_0^N(da) \rightarrow \xi(da)\;\;\text{narrowly as}\;\;N \rightarrow \infty,$$ 
%where $\xi(da)$ is a deterministic finite positive measure on $\R_+$.
%\end{assumption}
%\begin{assumption} \label{H stability} 
%For some (deterministic) $$We have 
%We have 
%$$\mathrm{Supp}(NZ_0^N) \subset [0,A],\;\;
\item[(ii)] $NZ_0^N \in \mathcal M_F$ is random and satisfies\footnote{where $a_N \lesssim b_N$ means $\sup_{N \geq 1} a_N b_N^{-1} <\infty$.} $\sup_N\; \langle Z_0^N,{\bf 1}\rangle \lesssim 1$ almost-surely; moreover  
%and for every $\lambda \geq 0$, we have $\sup_{N \geq 1}\E\big[\exp(\lambda\langle Z_0^N,{\bf 1}\rangle)\big] \leq e^{q\lambda}$ for some $q>0$ {\color{red} {\tt [Improve!!!]}}
%\text{and}\;\;
$Z_0^N \rightarrow \xi_0$
%\;\;\text{as}\;N \rightarrow \infty$$
narrowly, for some deterministic $\xi_0 \in \mathcal M_+$,
%\end{assumption} 
%Here, $\mathcal M_+$ denote the set positive finite measures on $\R_+=[0,\infty)$, equipped with the topology of {\color{blue} [specify!]}
%\begin{assumption} \label{H initial limit}
%We have 
\item[(iii)] $\xi_0(da)=g_0(a)da$ for some $g_0 \in \mathcal L_{\mathcal D}^{\mathrm{age}}$ such that $\int_0^\infty g_0(a)da < \infty$.
%:\R_+\rightarrow \R_+$ such that $\int_0^\infty g_0(a)da=1$.
\end{enumerate}
\end{assumption}

For $t \in [0,T]$, consider the equation\\
\begin{align} 
Z_t^N =\; & \tau_t Z_0^N 
+N^{-1}\int_0^t  \int_{\mathbb N \times \mathbb R_+}  \delta_{t-s}(da){\bf 1}_{\big\{0 \leq \vartheta \leq b(s,a_i(Z_{s^-}^N)), i \leq \langle NZ_{s^-}^N, {\bf 1}\rangle \big\}} \mathcal Q_1(ds, di, d\vartheta) \nonumber \\
&- N^{-1}\int_0^t  \int_{\mathbb N \times \mathbb R_+} \delta_{a_i(Z_{s^-}^N)+t-s}(da){\bf 1}_{\big\{0 \leq \vartheta \leq \mu(s,a_i(Z_{s^-}^N)), i \leq \langle NZ_{s^-}^N, {\bf 1}\rangle \big\}}\mathcal Q_2(ds,di,d\vartheta), \label{eq micro}
 \\ \nonumber
 \end{align}
\noindent where $\mathcal Q_i$, $i=1,2$ are independent Poisson random measures on $\R_+\times \mathbb N \setminus \{0\} \times \R_+$ with intensity measure $ds \big(\sum_{k \geq 1}\delta_k(di)\big) d\vartheta$. In this setting, the distribution $Z^N_0$ describes the renormalised state of the population at time $t=0$ and $N\langle Z^N_0, {\bf 1} \rangle$ its size.\\

% where  $\frak{n}$ denotes the counting measure on positive integers.
%$\mathbb N = \{1,2,\ldots\}$. {\color{blue} [describe dynamics.]} 

Under Assumption \ref{H basic} (i), we have existence and strong uniqueness of a solution to \eqref{eq micro} in $\mathbb D([0,T], \mathcal M_+)$, the Skorokhod space of c\`adl\`ag processes with values in $\mathcal M_+$.
%   {\color{blue} [specify topology on $\mathcal M_+$, give precise reference here]}
%, see Theorem \ref{} in \cite{} for instance. 
%Consider next the following stability and convergence of the initial condition:
%\begin{assumption} \label{H stability} 
%%For some (deterministic) $$We have 
%For every $\lambda \geq 0$, we have $\sup_{N \geq 1}\E[\exp(\lambda\langle Z_0^N,{\bf 1}\rangle)] \leq e^{q\lambda}$ for some $q>0$.
%and
%$Z_0^N \rightarrow \xi_0 \in \mathcal M_+$ as $N \rightarrow \infty$ narrowly, for a deterministic $\xi_0$.
%\end{assumption}
Under Assumption \ref{H basic} (i) and (ii)\footnote{Actually, the condition of the almost-sure bound $\sup_N\; \langle Z_0^N,{\bf 1}\rangle \lesssim 1$ can be relaxed to the significant weaker moment condition $\sup_{N \geq 1}\E[\langle Z_0^N,{\bf 1}\rangle^{1+\epsilon}] <\infty$ for some $\epsilon >0$.}, we even have the narrow convergence of $Z^N$ in $\mathbb D([0,T], \mathcal M_+)$ to a deterministic limit  $\xi \in \mathcal C([0,T], \mathcal M_+)$, see {\it e.g.} \cite{TRAN1, FOURNIERMELEARD}.
% {\color{blue} [give precise reference here]}.
% as follows for instance 
%from Theorem \ref{} in \cite{}. 
%{\color{blue} [give precise reference here]}.\\
%Consider next the following smothness hypothesis on the initial limiting distribution: 
  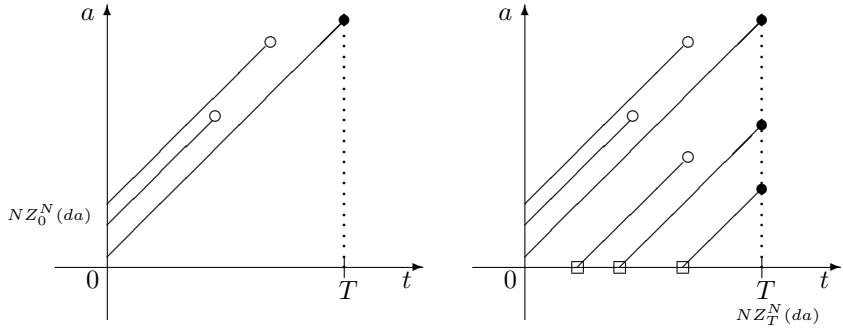
\begin{figure}[h]
\centering
\setlength{\unitlength}{0.7cm}
\begin{picture}(6,6)(-1,-1)
% axe x
   \put(0,-1){\vector(0,1){6}}
   \multiput(4.5,-0.2)(1,0){1}{\line(0,1){0.4}} % traits de graduation
     \put(4.4,-0.6){$T$} % valeurs graduées
     \put(5.6,-0.4){$t$}
       \put(-1.9,0.9){{\tiny $N Z_0^N(da)$}}
     %\put(1,3){\shortstack{$\mathcal D_U$}}
     %\put(3,1){\shortstack{$\mathcal D_L$}}
     %\put(4,-1){\shortstack{$ Z_T(da)$}}
     %\put(-1.7,2.2){\shortstack{$\phi(a)da$}}
% axe y
   \put(-1,0){\vector(1,0){7}}
   \put(-0.4,-0.4){$0$} % valeurs graduées
   \put(-0.5,4.7){$a$}
   %\put(0,0.2){\circle*{0.2}}
   \put(4.5,4.7){\circle*{0.2}}
   \put(0,0.8){\line(1,1){2}}
   \put(0,1.2){\line(1,1){3}}
   \put(2.05,2.87){\circle{0.2}}
    \put(3.1,4.28){\circle{0.2}}
% quadrillage
   \multiput(4.5,0)(1,0){1}{
      \multiput(0,0)(0,0.2){25}{\circle*{0.05}}}
    %\put(-0.3,2.2){\shortstack{$\left\lbrace \begin{array}{l} \\ \\ \\ \\ \\ \\ \\ \\ \end{array}\right.$}}
% demi-droite
   
   \put(0,0.2){\line(1,1){4.55}}

\end{picture}
\hspace{1cm}
\setlength{\unitlength}{0.7cm}
\begin{picture}(6,6)(-1,-1)
% axe x
   \put(0,-1){\vector(0,1){6}}
   \multiput(4.5,-0.2)(1,0){1}{\line(0,1){0.4}} % traits de graduation
     \put(4.4,-0.6){$T$} % valeurs graduées
     \put(5.6,-0.4){$t$}
   %  \put(-1,0.6){{\tiny $Z_0(da)$}}
     %\put(1,3){\shortstack{$\mathcal D_U$}}
     %\put(3,1){\shortstack{$\mathcal D_L$}}
     \put(4,-1){\shortstack{{\tiny $ NZ_T^N(da)$}}}
     %\put(-1.7,2.2){\shortstack{$\phi(a)da$}}
% axe y
   \put(-1,0){\vector(1,0){7}}
   \put(-0.4,-0.4){0} % valeurs graduées
   \put(-0.5,4.7){${\small a}$}
   %\put(0,0.2){\circle*{0.2}}
   \put(4.5,4.7){\circle*{0.2}}
   \put(0,0.8){\line(1,1){2}}
   \put(0,1.2){\line(1,1){3}}
   \put(2.05,2.87){\circle{0.2}}
    \put(3.1,2.1){\circle{0.2}}
   \put(3.1,4.28){\circle{0.2}}
    \put(0.9,-0.1){\framebox(0.2,0.2)}
     \put(1.7,-0.1){\framebox(0.2,0.2)}
      \put(2.9,-0.1){\framebox(0.2,0.2)}
% quadrillage
   \multiput(4.5,0)(1,0){1}{
      \multiput(0,0)(0,0.2){25}{\circle*{0.05}}}
    %\put(-0.3,2.2){\shortstack{$\left\lbrace \begin{array}{l} \\ \\ \\ \\ \\ \\ \\ \\ \end{array}\right.$}}
% demi-droite
   
   \put(0,0.2){\line(1,1){4.55}}
  \put(1,0){\line(1,1){2}}
  \put(1.8,0){\line(1,1){2.65}}
  \put(4.5,2.7){\circle*{0.2}}
  \put(3,0){\line(1,1){1.5}}
  \put(4.5,1.5){\circle*{0.2}}
\end{picture}
\vspace{-1mm}
\caption{Sample path of $NZ_0^N(da)$ and its evolution without births (left), sample path of $(NZ_t^N(da))_{0 \leq t \leq T}$ (right).}
\end{figure}
Under Assumption \ref{H basic} (iii), the limit $\xi=(\xi_t(da))_{0 \leq t \leq T}$ is smooth in the following sense: we have that $\xi_t(da)=g(t,a)da$, where $g$ is a weak solution to 
%\begin{align} \label{McKendrick}
%\left\{
%\begin{array}{ll}
%\partial_t g(t,a)+\partial_ag(t,a) + \mu(t,a)g(t,a) = 0 \\ \\
%\displaystyle g(0,a) = g_0(a),\;\; g(t,0) = \int_0^\infty b(t,a)g(t,a)da. \\ 
%%g(t,0) = \int_{\mathbb R_+} b(t,a)g(t,a)da.
%\end{array}
%\right.
%\end{align}
%see for instance 
%{\color{blue} [give precise reference here]}
%Theorem \ref{} in \cite{}.
%The limit $g$ 
%is the solution 
%of an inhomogeneous version of 
the McKendrick Von Foerster equation \eqref{McKendrick} defined in Section \ref{sec: setting} above (see \cite{MCKENDRICK, VONFOERSTER} and the comprehensive textbook of Perthame \cite{PERTHAME}).
%{\color{blue} [give precise reference of such PDE in applied analysis here.]} 
With the notation of Section \ref{sec: setting}, the equation $\mathcal H^N_{b,\mu}(Z^N)=0$ is given by \eqref{eq micro} while $\mathcal H_{b,\mu}(g)=0$ is given by \eqref{McKendrick}.
% for $f=(g_0,\mu,b)$ or $f=g$.

\subsection{Stability of the model} \label{sec: stability}

\subsubsection*{Preliminaries} The stability of $Z^N_t(da)$ relative to its limit $g(t,a)$ will be expressed in terms of weighted quantities of the form
$$\mathcal W_{w_2}^N(\mathcal F)_t = \sup_{f \in \mathcal F}\int_{\R_+} w_2(t-a) f_t(a)\big(Z^N_t(da)-g(t,a)da\big)$$
and also
%Under Assumptions \ref{H basic}, \ref{H stability} and \ref{H initial limit}, we define
$$\mathcal W_{w_1, w_2}^N(\mathcal F)_t = \sup_{f \in \mathcal F} \int_0^t w_1(s)\int_{\R_+} w_2(s-a) f_s(a)\big(Z^N_s(da)-g(s,a)da\big)ds,$$
%where $w_1 \in \mathcal L_{\mathcal D}^{{\small\mathrm{\,time}}}$ and $w_2(-\cdot) \in \mathcal L_{\mathcal D}^{{\small\mathrm{\,age}}}$ 
where $w_i$, $1=1,2$ are two bounded weight functions (possibly taking negative values). For notational simplicity, we write $f_t(a)=f(t,a)$ for $f \in \mathcal L_{\mathcal D}^\infty $ when no confusion is possible. Implicitly, we assume that $\mathcal F$ is well-behaved in the sense that $\mathcal W_{w_2}^N(\mathcal F)_t$ and $\mathcal W_{w_1, w_2}^N(\mathcal F)_t$ are measurable, as random variables on the ambient probability space over which $Z^N$ is defined.

%If $\mathcal F$ consists of $1$-Lipschitz functions and $w_2=1$,  the distance $\mathcal W_{1}^N(\mathcal F)_t$ 
%%(respectively $\mathcal W_{w_1, w_2}^N(\mathcal F)_t$) 
%is reminiscent of the $1$-Wasserstein distance between $Z^N_t(da)$ and $g(t,a)da$ provided $b$ and $\mu$ are smooth enough. Note also that we formally deduce $\mathcal W_{w_2}^N(\mathcal F)_t$ from $\mathcal W_{w_1, w_2}^N(\mathcal F)_t$ by letting $w_1 = \delta_t$.
%

\subsubsection*{The structure of $\mathcal F$}

%\begin{rk}
%If $\mathcal F$ consists of bounded $1$-Lipschitz functions and $w_2=1$,  the distance $\mathcal W_{1}^N(\mathcal F)_t$ 
%%(respectively $\mathcal W_{w_1, w_2}^N(\mathcal F)_t$) 
%is reminiscent of the $1$-Wasserstein distance between $Z^N_t(da)$ and $g(t,a)da$ provided $b$ and $\mu$ are smooth enough. 
%%(respectively between $Z^N_s(da)ds$ and $g(s,a)dads$).
%\end{rk}

We describe the minimal structure we need to put on $\mathcal F$ so that the subsequent concentration properties hold for $\mathcal W_{w_2}^N(\mathcal F)_t$ and $\mathcal W_{w_1, w_2}^N(\mathcal F)_t$. In particular, we must be able to control the complexity of $\mathcal F$ measured in terms of entropy. Let $\mathsf s_t$, $\mathsf t_t$ and $\mathsf u_t$ be the operators on $\mathcal L_{\mathcal D}^\infty $ defined by
$$\mathsf s_t(f) = \big((s,a) \mapsto f(t,a+t)\big),\;\;\mathsf t_t(f) = \big((s,a) \mapsto f(t,t-s)\big),\;\;\mathsf u_t(f) = \big((s,a) \mapsto f(t,t+a-s)\big).$$
%The minimal structure on $\mathcal F$ we need is the following

\begin{assumption} \label{minimal F}
We have $0, c_0, c_0 b, c_0 \mu\in \mathcal F$ for some constant $c_0 >0$. Moreover, for every $t \in [0,T]$, the class $\mathcal F$ is stable under the following operations: 
\begin{equation} \label{eq def stabilite operateurs}
f \mapsto -f,\; (f,g) \mapsto f g,\;
f  \mapsto \mathsf s_t(f),\; f \mapsto  \mathsf t_t(f),\; f \mapsto \mathsf u_t(f).
\end{equation}
\end{assumption}

Let $\mathrm{diam}_{|\cdot|_\infty}(\mathcal F) = \sup_{f,g\in\mathcal F}|f-g|_\infty$ and write $\mathcal N(\mathcal F, |\cdot|_\infty, \epsilon)$ for the minimal number of $\epsilon$-balls for the $|\cdot|_\infty$-metric that are necessary to cover $\mathcal F$.

\begin{prop} \label{prop entropie minimale}
Let $\mathcal F$ be the minimal set satisfying Assumption \ref{minimal F} for some $c_0 > 0$ such that $c_1 = c_0 \max(|b|_\infty,|\mu|_\infty) < 1$. If moreover $b,\mu \in \mathcal C^s_{\mathcal D}$ for some $s>0$ ($ \mathcal C^s$  is the set of H\"older continuous functions defined in \eqref{eq def holder global}), then 
\begin{equation} \label{eq def entropie}
\mathrm e(\mathcal F) = \int_0^1 \log\big(1+\mathcal N(\mathcal F, |\cdot|_\infty, \epsilon)\big)d\epsilon <\infty.
\end{equation}
\end{prop}

\subsubsection*{Concentration properties}

\begin{defi}[mild concentration] A sequence of nonnegative random variables $(X^N)_{N \geq 1}$ has a mild concentration property of order $0 \leq r_N\rightarrow 0$ if for large enough $N$, we have
$$\PP\big(X^N \geq (1+u)r_N\big) \leq \frac{1}{e^u-1}\;\;\text{for every}\;\;u \geq 0.$$
\end{defi}

\begin{assumption} \label{concentration initiale}
The sequence 
$$|w_2|_{1,\infty}^{-1}\max_{h=1,w_2}\mathcal W_{h}^N(\mathcal F)_0$$ has a mild concentration property of order $r_N$
for some $0 \leq r_N \rightarrow 0$.
% and some $\mathcal F$ that satisfies the .
\end{assumption}

%\begin{rk} \label{rk moment estimate}
%Note that Assumption \ref{concentration initiale} implies the moment estimate 
%\begin{equation} \label{eq: moment estimate}
%\E\big[\max_{h=1,w_2}\mathcal W_{h}^N(\mathcal F)_0^p\big] \lesssim |w_2|_{1,\infty}^p r_N^{p}\;\;\text{for every}\;\; p>0, 
%\end{equation}
%since
%\begin{align}
%\E\big[\max_{h=1,w_2}\mathcal W_{h}^N(\mathcal F)_{0}^p\big] & = p\int_0^\infty \kappa^{p-1}\PP(\max_{h=1,w_2}\mathcal W_{h}^N(\mathcal F)_{0} \geq \kappa)d\kappa \nonumber \\
%& =|w_2|_{1,\infty}^p r_N^{p}\int_{-1}^\infty (1+u)^{p-1}\PP\big(|w_2|_{1,\infty}^{-1}\max_{h=1,w_2}\mathcal W_{h}^N(\mathcal F)_{0} \geq (1+u)r_N\big)du \nonumber\\
%& \leq|w_2|_{1,\infty}^p r_N^{p} \int_{-1}^\infty (1+u)^{p-1}\min\big((e^{u}-1)^{-1},1\big) du \lesssim  |w_2|_{1,\infty}^p r_N^{p}. \nonumber
%\end{align}
%In particular, if $Z_0^N(da)$ consists of a $N$-drawn of independent random variables with common distribution $g_0(a)da$ with $\int_0^\infty a^pg_0(a)da <\infty$, we have Assumption \ref{concentration initiale} if $g_0$ has subgaussion tails or the weaker \eqref{eq: moment estimate} if $\int_0^\infty a^pg_0(a)da < \infty$ with $r_N$ of order $N^{-1/2}$. 
%\end{rk}

\begin{thm} \label{thm concentration optimal} Work under Assumptions \ref{H basic}, \ref{minimal F} and \ref{concentration initiale}.
Assume moreover $\mathrm{diam}_{|\cdot|_\infty}(\mathcal F) \leq 1$ and 
$$\mathrm e(\mathcal F) = \int_0^1 \log\big(1+\mathcal N(\mathcal F, |\cdot|_\infty, \epsilon)\big)d\epsilon <\infty.$$ 
If $w_2$ has compact support with length support bounded in $N$ by some $\mathfrak u >0$ and satisfies an estimate of the form 
\begin{equation} \label{estimate LinftyL1}
|w_2|_\infty \lesssim \max(N^{1/2},r_N^{-1})|w_2|_1,
\end{equation} then 
$$\big(|w_1|_{1,\infty}|w_2|_{1,\infty}\big)^{-1}\mathcal W_{\,w_1,w_2}^N(\mathcal F)_T\;\;\;\text{and}\;\;\;|w_2|_{1,\infty}^{-1}\mathcal W_{w_2}^N(\mathcal F)_T$$
share both a mild concentration property of order $C\max(r_N,N^{-1/2})$, for an explicitly computable  $C = C(\mathfrak u, \mathrm e(\mathcal F),T,|b|_\infty, |\mu|_\infty, g_0, |w_1|_1, |w_2|_1) >0$ continuous in its arguments. In particular, if $|w_i|_1$, $i=1,2$ is uniformly bounded in $N$, then $C$ can be chosen independently of $N$.
\end{thm}

Several remarks are in order: {\bf 1)} If the initial condition $Z_0^N$ is close to its limit $g_0$ in $\mathcal W_{w_2}(\mathcal F)_0$-norm of order $r_N$, Theorem \ref{thm concentration optimal} states that the error inflates in $\mathcal W_{w_2}(\mathcal F)_t$-norm by a factor no worse than $N^{-1/2}$ for $t \in [0,T]$. In particular, whenever $r_N \lesssim N^{-1/2}$, the error propagation is stable. {\bf 2)} The order of magnitude of the error propagation is $\max(N^{-1/2},r_N)$, as one could expect. As for the order in terms of $w_1$ or $w_2$, the ideal order would be the integrated squared-error norm $|w_i|_{2}$ as a variance term in a central limit theorem for instance. Here, we obtain the slightly worse interpolation quantity $|w_i|_{1,\infty}$ which is always bigger than $|w_i|_2$. However, for statistical purposes, when $w_i$ is replaced by a kernel $w_i = h_N^{-1}K(h_N^{-1}\cdot)$ for some kernel $K$ such that $|K|_1=1$, the order is sharp, since in that case 
$$|w_i|_{1,\infty} \approx h_N^{-1/2} \approx |w_i|_{2}$$
and moreover $|w_i|_1$ is uniformly bounded in $N$.  
The fact that we have here the correct order for dilating kernels is crucial for nonparametric estimation and is the main purpose (and difficulty) of Theorem \ref{thm concentration optimal}. This seems to be a standard situation for nonparametric estimation in structured populations, where such effects are also met, see \cite{DHKR1, HOFFMANN2016, BHO}.  {\bf 3)} If $w_2$ is not compactly supported or if \eqref{estimate LinftyL1} does not hold, we still have 
that 
$$\big(|w_1|_{1,\infty}|w_2|_{\infty}\big)^{-1}\mathcal W_{\,w_1,w_2}^N(\mathcal F)_T\;\;\;\text{and}\;\;\;|w_2|_{\infty}^{-1}\mathcal W_{w_2}^N(\mathcal F)_T$$
share both a mild concentration property of order $C\max(r_N,N^{-1/2})$,
as explicitly obtained in the proof. However, such a result is not sufficient for nonparametric estimation: picking $w_2 =  h_N^{-1}K(h_N^{-1}\cdot)$ yields $|w_2|_\infty \approx h_N^{-1} $ which is dramatically worse than the expected $h_N^{-1/2}$ in kernel estimation.
%The constant $C$ is then possibly inflated by a factor that only depends on $\max(|w_1|_1,|w_2|_1)$ and has no consequence for the subsequent statistical results.
% of Theorems \ref{}.
 {\bf 4)} The constant $C$ also depends on the length of the support of $w_2$, but that may be considered as fixed once for all for later statistical purposes. {\bf 5)} Assumption \ref{concentration initiale} implies the moment estimate 
\begin{equation} \label{eq: moment estimate}
\E\big[\max_{h=1,w_2}\mathcal W_{h}^N(\mathcal F)_0^p\big] \lesssim |w_2|_{1,\infty}^p r_N^{p}\;\;\text{for every}\;\; p>0. 
\end{equation}
%since
%\begin{align}
%\E\big[\max_{h=1,w_2}\mathcal W_{h}^N(\mathcal F)_{0}^p\big] & = p\int_0^\infty \kappa^{p-1}\PP(\max_{h=1,w_2}\mathcal W_{h}^N(\mathcal F)_{0} \geq \kappa)d\kappa \nonumber \\
%& =|w_2|_{1,\infty}^p r_N^{p}\int_{-1}^\infty (1+u)^{p-1}\PP\big(|w_2|_{1,\infty}^{-1}\max_{h=1,w_2}\mathcal W_{h}^N(\mathcal F)_{0} \geq (1+u)r_N\big)du \nonumber\\
%& \leq|w_2|_{1,\infty}^p r_N^{p} \int_{-1}^\infty (1+u)^{p-1}\min\big((e^{u}-1)^{-1},1\big) du \lesssim  |w_2|_{1,\infty}^p r_N^{p}. \nonumber
%\end{align}
{\bf 6)} We finally give a reasonable and sufficient condition for Assumption \ref{concentration initiale} to hold. 
%In particular, if $NZ_0^N(da)$ consists of a $N$-drawn of independent random variables with common distribution $g_0(a)da$ 
%with $\int_0^\infty a^pg_0(a)da <\infty$, 
%we have Assumption \ref{concentration initiale} under appropriate conditions:
\begin{prop} \label{prop: KleinRio}
If $\mathcal F$ is uniformly bounded (in particular if $\mathcal F$ is the minimal set of Proposition \ref{prop entropie minimale}) and if 
if $NZ_0^N(da)$ consists of a $N$-drawn of independent random variables with common distribution $g_0(a)da$ (with the normalisation assumption $\int_{\R_+}g_0(a)da=1$), we have Assumption  \ref{concentration initiale}.
% $w_2$ is  $g_0$ has subgaussian tails or the weaker \eqref{eq: moment estimate} if $\int_0^\infty a^pg_0(a)da < \infty$ with $r_N$ of order $N^{-1/2}$. 
\end{prop}
The proof is based on a concentration inequality of Klein and Rio \cite{KLEINRIO} and is developed in a statistical setting in Comte {\it et al.} \cite{COMTEDEDECKERTAUPIN} and
%. The proof of Proposition \ref{KleinRio} 
is delayed until Appendix \ref{proofKleinRio}\\

We end this section by giving a global stability result for the propagation of the error $Z^N_t(da)-g(t,a)da$, given a preliminary control on $Z^N_0(da)-g(0,a)da$, which relies on the techniques developed in Theorem \ref{thm concentration optimal}, but with a weaker moment condition for the initial control of the particle system.  
%, in the spirit of Remark \ref{rk moment estimate} in Section \ref{sec: stability}. 
%Recall that we write 
%$$\mathcal W_{w_2}^N(\mathcal F)_t = \sup_{f \in \mathcal F}\int_0^\infty w_2(t-a) f_t(a)\big(Z^N_t(da)-g(t,a)da\big).$$ 

\begin{prop} \label{prop coherence 1}  Work under Assumptions \ref{H basic} and \ref{minimal F}. 
If 
\begin{equation} \label{eq: controle init propagation}
\E\big[\max_{k=1,w_2}\mathcal W_{k}^N(\mathcal F)_0^p\big] \leq |w_2|_{1,\infty}^p r_N^p
\end{equation}
for some $r_N \geq 0$ and $p \geq 1$, and if $w_2$ is compactly supported and satisfies an estimate of the form $|w_2|_\infty \lesssim \max(N^{1/2},r_N^{-1})|w_2|_1$, then 
\begin{equation} \label{error propagation}
\E\big[\mathcal W_{w_2}^N(\mathcal F)_T^p\big] \lesssim |w_2|_{1,\infty}^p \max(N^{-p/2},r_N^p)
\end{equation}
and
\begin{equation} \label{eq error W2}
\E\big[\mathcal W_{w_1,w_2}^N(\mathcal F)_T^p\big] \lesssim (|w_1|_{1,\infty}|w_2|_{1,\infty})^p\max(N^{-p/2},r_N^p).
\end{equation}
\end{prop}
%The proof is given 
%If Assumption \ref{concentration initiale} is replaced by the weaker moment condition \eqref{eq: moment estimate}, we still obtain
%$$
%\E\big[\mathcal W_{w_2}^N(\mathcal F)_T^p\big] \lesssim |w_2|_{1,\infty}^p r_N^{p}\;\;\text{for every}\;\; p>0.
%$$
%as easily obtained from the proof of Theorem \ref{thm concentration optimal}.

% {\color{blue}[proof of this result needed!]}
%en dire plus ici, + ref si c'est une hypothse classique + preuve??

\section{Nonparametric estimation of $g$ and $\mu$} \label{sec stat oracle}

%{\color{red} Assume in the stat section $r_N \leq N^{-1/2}$!!!!!}

\subsection{Kernel approximation}

%\subsubsection*{Kernel approximation}

\begin{defi}
A kernel $K$ of (integer) order $\ell_0 \geq 1$ is a bounded function with compact support in $\R_+$ such that
%of $\mathcal L_{\mathcal D}^{{\small\mathrm{\,time}}}$ or $\mathcal L_{\mathcal D}^{{\small\mathrm{\,age}}}$ such that
%\begin{equation} \label{kernel order}
$$
\int_{0}^\infty\kappa^{\ell-1} K(\kappa)d\kappa = {\bf 1}_{\{\ell=1\}},\;\;\text{for}\;\;\ell=1,\ldots, \ell_0-1.
%\end{equation}
$$
\end{defi}
For a bandwidth $h >0$, we set $K_h(\kappa) = h^{-1}K(h^{-1}\kappa)$ so that $|K_h|_1 = |K|_1$.
In order to approximate functions of $\mathcal L_{\mathcal D}^\infty $, we use bivariate kernels defined by
$$H \otimes K(t,a) = H(t)K(a)\;\;\text{for}\;\;(t,a) \in \mathcal D,$$
 with 
$H \in \mathcal L_{\mathcal D}^{{\small\mathrm{\,time}}}$, $K \in \mathcal L_{\mathcal D}^{{\small\mathrm{\,age}}}$. 
For a bivariate bandwidth $\boldsymbol{h} = (h_1, h_2)$ with $h_i>0$, let
%$$\mathcal K_{\boldsymbol h}(t,a)= (h_1h_2)^{-1}\mathcal K(h_1^{-1}t,h_2^{-1}a) = K_{h_1}(t)H_{h_2}(a).$$
$$(H \otimes K)_{\boldsymbol h}(t,a)=  H_{h_1}(t)K_{h_2}(a)$$
and define the linear approximation
\begin{equation} \label{kernel approx}
(H \otimes K)_{\boldsymbol h} \star f(t,a)=\int_{0}^T\int_{0}^\infty f(s,u) (H \otimes K)_{\boldsymbol h}(t-s, a-u)dsdu.
\end{equation}
%Note that \eqref{kernel approx} also gives one-dimensional approximations: if $f \in \mathcal L_{\mathcal D}^{{\small\mathrm{\,age}}}$, the above formula reduces to
%$$$$

We may also approximate $f$ in another system of coordinates: if 
%$\varphi: [0,T] \times \R_+ \rightarrow  [0,T] \times \R_+$
$\varphi: \mathcal D \rightarrow \mathcal D$ is invertible, reparametrise $f$ via
$$f(t,a) = \widetilde f \circ \varphi(t,a)$$
and define the $\varphi$-skewed linear approximation 
\begin{align*}
%\begin{equation} \label{kernel approx skewed}
(H \otimes  K)_{\boldsymbol h} \circ \varphi) \star f (t,a) & =\int_{0}^T\int_{0}^\infty f(s,u)\big( (H \otimes K)_{\boldsymbol h}\circ \varphi \big)(s-t,u-a)dsdu 
%\\
%& = \int_{0}^T\int_{\R_+}f(s,u)\mathcal K_{\boldsymbol h}\big(s-t,s-u-(t-a)\big)dsdu 
%\end{equation}
\end{align*}
so that 
%\begin{equation} \label{eq smoothness transfer}
$
\big((H \otimes K)_{\boldsymbol h} \circ \varphi\big) \star f (t,a) = (H \otimes K)_{\boldsymbol h} \star \widetilde f \big(\varphi(t,a)\big).
$
%\end{equation}
The $\varphi$-skewed approximation potentially has better approximation properties for $\widetilde f$ in the viscinity of $\varphi(t,a)$ than $f$ in the viscinity of $(t,a)$, as will become transparent in Section \ref{sec: minimax adaptive} below.
%\end{equation}

\subsection{Construction of estimators of $g$ and $\mu$} \label{sec constrution estimators}

\subsubsection*{Construction of an estimator of $g$} Let $K \in \mathcal L_{\mathcal D}^{{\small\mathrm{\,age}}}$ be a kernel of order $\ell_0 \geq 0$.
For $(t,a) \in \mathcal D$, we consider the family of estimators 
\begin{equation} \label{def est g}
\widehat g_{h}^N(t,a) = K_h \star Z_t^N(a) = \int_{\R_+} K_h(u-a)Z^N_t(du),\;\;h>0.
\end{equation} 

\begin{rk} At first glance, it may seem  slightly suprising to build an estimator of the bivariate function $g(t,a)$ by means of \eqref{def est g} that uses data $Z_t^N$ only and discards the observation $(Z_s^N, s \neq t)$. For instance, one may consider estimators of the form 
%while we have at our diposal the augmented observation scheme $(Z_s^N, s \in [0,T])$. We may alternatively consider
$$\big((H \otimes K)_{\boldsymbol h} \circ \varphi \big) \star Z^N (t,a)=\int_{0}^T\int_{\R_+}\big( (H \otimes K)_{\boldsymbol h}\circ \varphi \big)(s-t,u-a)Z_s^N(du)$$
Formally $\widehat g_{h}^N(t,a) = (H_{h_1=0} \otimes K_h) \star Z^N (t,a)$ without any specific change of coordinates and we will see that such a simple procedure already achieves minimax optimality, see Section \ref{sec: minimax optimality} below.
\end{rk}

\subsubsection*{Construction of the process of death occurences}
We first extract from the data $(Z_t^N(da))_{0 \leq t \leq T}$ the random measure 
$$\Gamma^N(dt,da) = \sum_{i \geq 1}\delta_{(T_i, A_i)}(dt,da)\;\;\text{on}\;\;[0,T] \times \R_+$$ 
associated with the successive times  $T_i$ of the death occurences of the population during the observation period $[0,T]$, together with the corresponding ages $A_i$ of the individuals that die at time $T_i$. 

Remember that the evaluation mappings $a_i(Z_t^N)$ in the representation
$Z_t^N = N^{-1}\sum_{i \geq 1}
%^{\langle NZ_t^N,{\bf 1}\rangle}
\delta_{a_i(Z_t^N)}$
are ordered: 
$$a_1(Z_t^N) < a_2(Z_t^N) < \ldots$$ and that $t \mapsto a_i(Z_t)$ is increasing with slope one unless a birth or a death occurs, in which case we have a non-negative or a negative jump.  It follows that
\begin{equation} \label{def death jumps}
\Gamma^N(dt, da) = \sum_{s >0} {\bf 1}_{\{i^\star = \inf\{i \geq 1, \Delta a_i(Z_s^N) >0\} < \infty\}} \delta_{(s, a_{i^\star}(Z_{s^-}^N))}(dt, da)
\end{equation}
on $[0,T] \times \R_+$,  
where we set $\Delta a_i(Z_s^N) =  a_i(Z_s^N)- a_i(Z_{s^-}^N)$ and with the usual convention $\inf \emptyset = \infty$. This second representation in terms of the jump measure of the processes $a_i(Z_t^N)$ gives an explicit construction of $\Gamma^N(dt, da)$ as a function of $(Z_t^N(da), t \in [0,T])$. 

\subsubsection*{Construction of an estimator of $\mu$}

Let $H \in \mathcal L_{\mathcal D}^{{\small\mathrm{time}}}$ and $K \in \mathcal L_{\mathcal D}^{{\small\mathrm{age}}}$ be two kernels. 
%and $\mathcal K = H \otimes K$. 
For $(t,a) \in \mathcal D$ and $\varphi(t,a) = (t, t-a)$, consider the family 
\begin{equation} \label{eq def est gamma}
\widehat \pi_{\boldsymbol h}^N(t,a) = \int_0^T\int_{\R_+} \big((H \otimes K)_{\boldsymbol h}\circ \varphi\big)(s-t,u-a) \Gamma^N(ds,du),\;\;\boldsymbol h = (h_1,h_2)\;\;\text{with}\;\;h_i >0,
\end{equation}
that estimate the function $\pi = \mu g$. An estimator of $\mu(t,a)$ is obtained by considering the ratio
\begin{equation} \label{eq def est mu}
\widehat \mu_{h, \boldsymbol h}^N(t,a)_{\varpi} = 
%\left\{
%\begin{array}{cl}
\frac{\widehat \pi_{\boldsymbol h}^N(t,a)}{\widehat g_{h}^N(t,a)\vee \varpi}
\end{equation}
for some threshold $\varpi >0$, 
%if $\widehat g_{h}^N(t,a) \neq 0$ and $0$ otherwise, 
%0 & \text{otherwise},
%\end{array}
%\right.
%\end{equation}
%whenever $\widehat g_{h}^N(t,a) \neq 0$ and let $\widehat \mu_{h, \boldsymbol h}^N(t,a) = 0$ otherwise.
and is thus specified by the bandwidths $h>0$, $\boldsymbol h = (h_1,h_2)$ with $h_i >0$ and $\varpi >0$.
%\begin{rk}
%\end{rk}

\subsection{Oracle inequalities} \label{sec oracle inequalities}
\subsubsection*{Estimation of $g$, data-driven bandwidth}
Pick a lattice $\mathcal G_1^N$ included in $[N^{-1/2}, (\log N)^{-1} ]$ and such that $\mathrm{Card}(\mathcal G_1^N) \lesssim N$. The algorithm, based on the Lepski's principle as defined in the Goldenshluger-Lepski's method \cite{GOLDENSHLUGERLEPSKI1, GOLDENSHLUGERLEPSKI2} requires the family of linear estimators 
$$\Big(\widehat g_{h}^N(t,a), h\in \mathcal G_1^N\Big)$$
defined in \eqref{def est g} and selects an appropriate bandwidth $h = \widehat h^N(t,a)$ from the data $(Z_t^N(da))_{0 \leq t \leq T}$. For $(t,a) \in \mathcal D$, writing $\{x\}_+ = \max(x,0)$, define
$$\mathsf A_{h}^N(t,a) = \max_{h' \leq h, h'\in \mathcal G_1^N}\big\{\big(\widehat g_{h}^N(t,a)-\widehat g_{h'}^N(t,a)\big)^2-(\mathsf V_{h}^N+\mathsf V_{h'}^N)\big\}_+,$$
where
\begin{equation} \label{def upper variance}
\mathsf V_{h}^N = \big(4 (\log N)C^\star N^{-1/2}|K_h|_{1,\infty}\big)^2
\end{equation}
and $C^\star$ is a (known) upper bound of the constant $C$ of Theorem \ref{thm concentration optimal}. (Remember that the constant $C$ depends on the parameters of the model via $|b|_\infty, |\mu|_\infty$ and $g_0$.)
%$ {\color{blue}\text{ [bound on the variance of the sum of} \;K_{{\boldsymbol h}}\star \widehat g_{N,{\boldsymbol h'}}\;\;\text{and}\;\; \widehat g_{N,{\boldsymbol h'}}\;\times \sqrt{\log N}]}.
%$
Let 
%also
%$$\mathsf V_{h}^N(t,a) = \max_{h'\in \mathcal G_1^N} \mathsf V_{h,h'}^N(t,a)$$
%%H_N \approx {\color{blue}\kappa} C_{K,T,b_{\max}}h^{-1}N^{-1/2}$$
%and for some threshold $\kappa >0$ 

$$\widehat h^N(t,a) \in \text{argmin}_{h \in \mathcal G_1^N}\big(\mathsf A_{h}^N(t,a)+\mathsf V_{h}^N\big).$$
The data-driven Goldenshluger-Lepski estimator of $g(t,a)$ is defined as 
\begin{equation} \label{def GL est}
\widehat g_{\star}^N(t,a)=\widehat g_{\widehat h^N(t,a)}^N(t,a).
\end{equation}

\subsubsection*{Oracle estimate}
We need some notation. Given a kernel $K_h$, the bias at scale $h$ of $g$ at point $(t,a)$ is defined as
\begin{equation} \label{def bias}
\mathcal B_h^N(g)(t,a) =  \sup_{h' \leq h, h' \in \mathcal G_1^N}\Big|\int_0^\infty K_{h'}(u-a)g(t,u)du-g(t,a)\Big|.
\end{equation}

We are ready to give our first estimation result for  every $(t,a) \in \mathcal D_- =\mathcal D \setminus \{t=a\}$.

\begin{thm}  \label{thm oracle g} Work under Assumptions  \ref{H basic}, \ref{minimal F} and \ref{concentration initiale} with $r_N \leq N^{-1/2}$ and some $\mathcal F$ that satisfies $e(\mathcal F) < \infty$.
%Assume moreover that $b, \mu \in \mathcal C^{2+\epsilon}$ and $\mathcal F \subset \mathcal C^{2+\epsilon}$ for some $\epsilon >0$. 
%with $\alpha_2\geq 1$ and $\beta_2\geq 1$, \ref{H smoothness basic} and \ref{H initial concentration} 
For $(t,a) \in \mathcal D_-$, specify $\widehat g_\star^N(t,a)$ with a bounded and compactly supported kernel $K$. The following oracle inequality holds true
$$\E\big[\big(\widehat g_{\star}^N(t,a)-g(t,a)\big)^2\big] \lesssim \inf_{h \in \mathcal G_1^N}\big(\mathcal B_h^N(g)(t,a)^2+\mathsf V_{h}^N\big) +\delta_N$$
for large enough $N$, with $\delta_N = N^{-1}$ and up to a constant that depends on $C^\star$ and $K$. %with constants depending on $|b|_\infty$, $|\mu|_\infty$, $\mathcal F$, $T$ and $q$ from Assumption \ref{H stability}. 
\end{thm}

Some remarks: {\bf 1)} The fact that we measure the performance of $\widehat g_{\star}^N$ at point $(t,a)$ in pointwise squared-error loss is inessential here. Other integrated norms like $|\cdot|_p$ would work as well, following the general proof of Lepski's principle \cite{LEPSKI, GOLDENSHLUGERLEPSKI1, GOLDENSHLUGERLEPSKI2}. However, if we need a fine control of the bias in terms of smoothness space, this is no longer true and is linked to the anisotropic and spatial inhomogeneous smoothness structure of the solution $g$. This will become transparent in Theorems \ref{adapt esti g} and \ref{eq def sdeath} below. {\bf 2)} In \eqref{def upper variance}, the choice of $C^\star$ has to be set in principle prior to the data analysis and is of course difficult to calibrate. It depends on upper bounds on many quantities like $e(\mathcal F)$ that appear in the constant of Theorem \ref{thm concentration optimal} or supremum of norms of the unknown parameters $b$ and $\mu$. Moreover, the explicit value $C^\star$ obtained by tracking the constants in the computations of Section \ref{sec proofs concentration} is certainly too large. In practice, we need to inject some further prior knowledge and calibrate the threshold by some other method, possibly using data. Such approaches in the context of Lepski's principle have been developed lately in \cite{LACOURMASSARTRIVOIRARD}.
{\bf 3)} The proof relies on Theorem \ref{thm concentration optimal} which requires $e(\mathcal F)$ to be finite. However, this requirement is not heavy, as soon as $b$ and $\mu$ have a minimal global H\"older smoothness, as stems from Proposition \ref{prop entropie minimale}.

%other loss functions, choice of \varpi, choice of constants (refer massart), entropy always finite

\subsubsection*{Estimation of $\mu$, data-driven bandwidth} Analogously to the bandwidth-selection method for estimation of $g$ following Lepski's principle, we pick a discrete set $\mathcal G_2^N \subset [N^{-1/2}, (\log N)^{-1} ]^2$ with 
cardinality $\mathrm{Card} \mathcal G_2^N \lesssim N$. The construction is similar to that of $\widehat g_{\star}^N(t,a)$, given in addition the family of estimators
$$\big(\widehat \pi_{\boldsymbol h}^N(t,a), \boldsymbol h\in \mathcal G_2^N\big)$$
defined in \eqref{eq def est gamma}. For $(t,a) \in \mathcal D$, let
$$\mathsf A_{\boldsymbol h}^N(t,a) = \max_{h'\in \mathcal G_2^N}\big\{\big(\widehat \pi_{\boldsymbol h}^N(t,a)-\widehat \pi_{\boldsymbol h'}^N(t,a)\big)^2-(\mathsf V_{\boldsymbol h}^N+\mathsf V_{\boldsymbol h'}^N)\big\}_+,$$
where
\begin{equation} \label{def upper bivariance}
\mathsf V_{\boldsymbol h}^N = \big(4 (\log N)C^\star N^{-1/2}|H_{h_1}|_{1,\infty} |K_{h_2}|_{1,\infty}\big)^2
\end{equation}
and $C^\star$ is a (known) upper bound of the constant $C$ of Theorem \ref{thm concentration optimal}.
%$\mathsf V_{\boldsymbol h, \boldsymbol h'}^N(t,a)=.$ {\color{blue} [specify!]}
%$ {\color{blue}\text{ [bound on the variance of the sum of} \;K_{{\boldsymbol h}}\star \widehat g_{N,{\boldsymbol h'}}\;\;\text{and}\;\; \widehat g_{N,{\boldsymbol h'}}\;\times \sqrt{\log N}]}.$
%Define also
%$$\mathsf V_{\boldsymbol h}^N(t,a) = \max_{\boldsymbol h'\in \mathcal G_2^N} \mathsf V_{\boldsymbol h,\boldsymbol h'}^N(t,a)$$
%and finally
Let
$$\widehat {\boldsymbol h}^N(t,a) \in \text{argmin}_{\boldsymbol h \in \mathcal G_2^N}\big(\mathsf A_{\boldsymbol h}^N(t,a)+\mathsf V_{\boldsymbol h}^N(t,a)\big).$$

The data-driven Goldenshluger-Lepski estimator of $\mu(t,a)$ is defined as 
\begin{equation} \label{def GL est mu}
\widehat \mu_{\star}^N(t,a)_{\varpi}=
%\left\{
%\begin{array}{cl}
  \mu_{\widehat {h}^N(t,a),\widehat {\boldsymbol h}^N(t,a)}^N(t,a)_{\varpi}.
  % = g_{\star}^N(t,a)^{-1}{\widehat \gamma}_{\widehat {\boldsymbol h}^N(t,a)}^N(t,a)
\end{equation}
%& \text{if}\;\; g_{\star}^N(t,a) \neq 0 \\ \\
%if $\widehat g_{\star}^N(t,a) \neq 0$ and $0$ otherwise.
% \\ 
%0 & \text{otherwise}.
%\end{array}
%\right.
%\end{equation}
 
\subsubsection*{Oracle estimates} In order to estimate $\mu$ in squared-error loss consistently with the quotient estimator \eqref{def GL est mu}, we need a (local) lower bound assumption on $g(t,a)$. Let 
\begin{align*}
\mathcal D_U & = \{(t,a) \in \mathcal D, a >t\}, \\
\mathcal D_L & = \{(t,a) \in \mathcal D, a < t\},
\end{align*}
and $\mathcal D^- = \mathcal D \setminus \{t=a\}$
%, depending on the choice of coordinates.
%denote respectively the upper and lower part of $\mathcal D$ separated by the diagonal $\{t=a\}$, 
so that $\mathcal D^- = \mathcal D_L \cup \mathcal D_U$. 
A sufficient condition is given by the following

\begin{assumption} \label{assumption minoration g}
For every $(t,a) \in \mathcal D^-$ there exists an open set $\mathcal U_{(t,a)}$ such that
\begin{equation} \label{eq condit mino g L}
\inf_{u \in \mathcal U_{(t,a)}}b(t-a,t-a+u)g_0(u)\geq \delta \;\;\text{if}\;\; (t,a) \in \mathcal D_L
\end{equation}
and
\begin{equation} \label{eq condit mino g U}
g_0(t-a) \geq \delta \;\;\text{if}\;\; (t,a) \in \mathcal D_U,
\end{equation}
for some $\delta >  0$.
\end{assumption}
We need some notation. For $\boldsymbol h = (h_1,h_2)$ and $\boldsymbol h' = (h_1',h_2')$ in $\mathcal G_2^N$, we say that $\boldsymbol h \leq \boldsymbol h'$ if
$h_1 \leq h'_1$ and $h_2 \leq h'_2$ hold simultaneously. Given a bivariate kernel $H \otimes K$, the bias at scale $\boldsymbol h$ of $\pi = \mu g$ at point $(t,a)$ in the direction $\varphi$ is defined as
\begin{equation} \label{def bias}
\mathcal B_{\boldsymbol h}^N(\mu g)(t,a) =  \sup_{\boldsymbol h' \leq \boldsymbol h, \boldsymbol h' \in \mathcal G_2^N}\Big|
%\int_0^T \int_0^\infty
\int_{\mathcal D} \big((H \otimes K)_{\boldsymbol h'} \circ \varphi \big)(s-t,u-a)\pi(s,u)duds-\pi(t,a)\Big|.
\end{equation}

%For $\epsilon >0$, we say that $f \in \mathcal L_{\mathcal D}^\infty $ belongs to $\mathcal C^{2+\epsilon}$ 
%denote the subset of $\mathcal L_{\mathcal D}^\infty $ of twice differentiable functions that are moreover $\epsilon$-H\"older continuous, meaning that 
%with second derivatives 
%that are uniformly $\epsilon$-H\"older continuous in the following sense:
%if
%$$|f|_{\mathcal C^{2+\epsilon}} = |f|_\infty + \sup_{h_1,h_2 \neq 0}(h_1^{-\epsilon}|\partial_t^2 f(t+h_1,a)-f(t,a)| + h_2^{-1}|\partial_a^2 f(t,a+h_2)-f(t,a)| < \infty,$$
 %for every $(t,a) \in [0,T] \times \R_+$ with appropriate boundary modifications for $h_1,h_2$.
 %, subject to $t+h_1\in [0,T] $ and $a+h_2 \in \R_+$.
%$$\mathcal C^{2+\epsilon} = \Big\{f \in \mathcal L_{\mathcal D}^\infty , \big|(\partial_t^2+\partial_a^2)\big(f(t+h,a+k)-f(t,a)\big)\big| \leq (|h|+|k|)^\epsilon, \forall h,k \Big\}$$
%We need some notation. 
%Given a kernel $K$, the bias at scale $h$ of $g$ at point $(t,a)$ is defined as
%\begin{equation} \label{def bias}
%\mathcal B_h^N(g)(t,a) =  \sup_{h' \leq h, h' \in \mathcal G_1^N}\big|\int_0^\infty K_{h'}(u-a)g(t,u)du-g(t,a)\big|.
%\end{equation}

\begin{thm} \label{thm oracle mu}
Work under Assumptions  \ref{H basic}, \ref{minimal F}, \ref{concentration initiale} with $r_N \leq N^{-1/2}$ and some $\mathcal F$ that satisfies $e(\mathcal F) < \infty$ together with Assumption \ref{assumption minoration g}.  
%Assume moreover that $b, \mu \in \mathcal C^{2+\epsilon}$ and $\mathcal F \subset \mathcal C^{2+\epsilon}$ for some $\epsilon >0$. 
%with $\alpha_2\geq 1$ and $\beta_2\geq 1$, \ref{H smoothness basic} and \ref{H initial concentration} 
For $(t,a) \in \mathcal D^-$ specify $\widehat \mu_{\star}^N(t,a)_{\varpi}$ with kernels $H, K$. The following oracle inequality holds true 
%$$\E\big[\big(g_{\star}^N(t,a)-g(t,a)\big)^2\big] \leq (1+C_1)\inf_{h \in \mathcal G_1^N}\E\big[\big(\widehat g_{h}^N(t,a)-g(t,a)\big)^2\big]+\delta_N$$
%and
$$\E\big[\big(\mu_{\star}^N(t,a)_{\varpi}-\mu(t,a)\big)^2\big] \lesssim   \inf_{h \in \mathcal G_1^N}\big(\mathcal B_h^N(g)(t,a)^2+\mathsf V_{h}^N\big) + \inf_{\boldsymbol h \in \mathcal G_2^N}\big(\mathcal B_{\boldsymbol h}^N(\mu g)(t,a)^2+\mathsf V_{\boldsymbol h}^N\big) +\delta_N$$
for large enough $N$ and small enough $\varpi >0$, with $\delta_N = N^{-1}$ and up to a constant that depends on $C^\star$ and the kernels $H, K$.
% and $\mathcal K$.
\end{thm}

Some remarks: {\bf 1)} Similar to the case of Theorem \ref{thm oracle g}, other loss functions can be chosen. {\bf 2)} We see that the performance of $\widehat \mu_{\star}^N(t,a)_{\varpi}$ is similar to the worst performance of the estimation of the product $\pi = \mu g$ and the estimation of $g$, as is standard in the study of quotient estimator in the classical Nadaraya-Watson (NW) sense \cite{BIERENS, NADARAYA}. However, the situation is quite different here than what is customary in standard nonparametric regression with NW: the estimation of $g(t,a)$ is actually equivalent to the estimation of a univariate function, while $\pi(t,a)$ is related to a genuinely bi-variate estimation problem that suffers from a dimensional effect. Therefore, there is good hope to obtain here an optimal procedure, as will become transparent under H\" older anisotropic smoothness scales in the subsequent minimax theorems \ref{thm LB} and \ref{th adapt minimax mu} below. {\bf 3)} The same remark about the choice of $C^\star$ (and also the threshold $\varpi$) as in Theorem \ref{thm oracle g} above are valid in the context of the estimation of $\mu(t,a)$.  
%other loss functions, choice of \varpi, choice of constants (refer massart)
\section{Adaptive estimation under anisotropic H\" older smoothness} \label{sec: minimax adaptive}

\subsection{The smoothness of the McKendrick Von Foerster equation} 
%e first define the anisotropic smoothness classes we will work with.
\begin{defi}
Let $\alpha >0$, $x_0\in \R$ and $\mathcal U_{x_0}$ be a neighbourhood of $x_0$.   
%, $\alpha >0$, $\delta>0$ and $\mathcal U \supseteq [x_0-\delta, x_0+\delta]$. 
We say that $f:\mathcal U_{x_0} \rightarrow \R$ belongs to $\mathcal H^{\alpha}(x_0)$
%=\mathcal H^{\alpha}_{\mathcal U_{x_0}}$ if for some $C>0$
% there exists a neighbourhood $\mathcal I$ of $x_0$ such that
if\footnote{The definition depends on $\mathcal U_{x_0}$, further omitted in the notation.} for every $x,y \in \mathcal U_{x_0}$
\begin{equation} \label{def holder}
%\sup_{x \in \mathcal U_{x_0}}|f(x)|+
|f^{(n)}(y)-f^{(n)}(x)| \leq C|y-x|^{\{\alpha\}}
%\forall h\leq \delta,\;\;\inf_{P \in \mathcal P_k}\sup_{|x-x_0|\leq h}|f(x)-P(x-x_0)| \leq Ch^{\alpha}
\end{equation}
%where $k=\lfloor s \rfloor$ (the largest integer small than $k$) and $\mathcal P_k$ is the set of all polynomials with degree $k$.
having $\alpha=n+\{\alpha\}$ for a non-negative integer $n$ and $0 < \{\alpha\} \leq 1$. 
\end{defi}
We obtain a semi-norm by setting 
$|f|_{\mathcal H^\alpha(x_0)} = 
%\sup_{|x-x_0|\leq \delta}|f(x)|
\sup_{x \in \mathcal U_{x_0}}|f(x)|+C_{\mathcal U_{x_0}}(f),$
where $C_{\mathcal U_{x_0}}(f)$ is the smallest constant $C$ for which \eqref{def holder} holds. The extension to multivariate functions is straightforward: 
\begin{defi} \label{def anisotropic space}
The bivariate function $f$ belongs to the anisotropic H\"older class $\mathcal H^{\alpha_1,\alpha_2}(x_0,y_0)$ if
$$|f|_{\mathcal H^{\alpha_1,\alpha_2}(x_0,y_0)} = |f(\cdot,y_0)|_{\mathcal H^{\alpha_1}(x_0)}+ |f(x_0,\cdot)|_{\mathcal H^{\alpha_2}(y_0)} < \infty.$$
% for some neighbourhood $\mathcal I(x_0) \times \mathcal J(y_0)$ of $(x_0,y_0)$, we have
%$x\mapsto f(x,y_0) \in \mathcal H^{\alpha_1}(x_0)$ and $y \mapsto f(x_0,y) \in \mathcal H^{\alpha_2}(y_0)$ hold simultaneously. 
\end{defi}
%For $\mathcal U \subset \R^2$, we set
%$$\mathcal U(t,\cdot) = \{(t',a)\in \mathcal U, t'= t\}\;\;\text{and}\;\;\mathcal U(\cdot, a) = \{(t,a')\in\mathcal U, a'=a\}.$$
%\begin{defi}
%For $\alpha,\beta>0$ and an open set $\mathcal U \subset \mathcal D$, the function $f: \mathcal D \rightarrow \R$ belongs to the anisotropic H\" older class $\mathcal H^{\alpha,\beta}(\mathcal U)$ if for every $(t,a)\in \mathcal U$: we have
%$$
%t' \mapsto f(t',a) \in \mathcal H^{\alpha}\big(\mathcal U(\cdot,a)\big)\;\;\text{and}\;\; a' \mapsto f(t,a') \in \mathcal H^{\beta}\big(\mathcal U(t,\cdot)\big).
%$$
%\end{defi}
%We equip $\mathcal H^{\alpha_1,\alpha_2}(x_0,y_0)$ with the semi-norm  
%\begin{equation} \label{def semi-norm anisotropic}
%$
%|f|_{\mathcal H^{\alpha_1,\alpha_2}(x_0,y_0)} = |f(\cdot,y_0)|_{\mathcal H^{\alpha_1}(x_0)}+ |f(x_0,\cdot)|_{\mathcal H^{\alpha_2}(y_0)}$ and 
We write $f \in \mathcal H$ if for every $(t,a) \in {\mathcal D}$, we have $f \in \mathcal H^{\sigma,\tau}(t,a)$.
\begin{assumption} \label{H smoothness basic}
For some $\alpha, \beta, \gamma, \delta >0, \nu \geq \max(\gamma,\delta)+1$ and for every $(t,a) \in \mathcal D$,
we have 
$$b \in \mathcal H^{\alpha,\beta}(t,a),
\;\;
%$$
%$$
\mu \in \mathcal H^{\gamma, \delta}(t,a),
\;\;
%$$
%$$
g_0 \in \mathcal H^\nu(a).$$
\end{assumption}
%\end{equation}

We give two results about the pointwise smoothness of  the solution of the McKendrick Von Foester equation on $\mathcal D^- = \mathcal D \setminus \{t=a\}$, depending on the choice of coordinates.
%denote respectively the upper and lower part of $\mathcal D$ separated by the diagonal $\{t=a\}$, 
%so that $\mathcal D^- = \mathcal D_L \cup \mathcal D_L$. 
The smoothness of $g$ differs on $\mathcal D_U$ where only mortality affects the population and $\mathcal D_L$, where both mortality and birth come into play. Introduce also the change of coordinates
%\begin{equation} \label{eq def syst}
$\varphi(t,a) = (t,t-a)$
%\end{equation}
that  maps 
\begin{align*}
& \mathcal D_U 
%\stackrel{\varphi}{\rightarrow}\widetilde{\mathcal D}_U = \{(t,a'), 0 < t < T, a' <0\}\\
\rightarrow \varphi(\mathcal D_U) = \widetilde{\mathcal D}_U = \{(t,a') \in \mathcal D, 0 \leq t \leq T, a' <0\} \\
%\;\;\text{and}\;\;
&\mathcal D_L 
%\stackrel{\varphi}{\rightarrow}
\rightarrow \widetilde{\mathcal D}_L = \varphi(\mathcal D_L) = \{(t,a') \in \mathcal D, 0 \leq t \leq T, 0 < a' < t\}
\end{align*} 
onto smoothly. This defines in turn 
$$\widetilde g  : \widetilde{\mathcal D}_U \cup \widetilde{\mathcal D}_L \rightarrow \R_+\;\;\text{via}\;\; 
%$\widetilde g= g \circ \varphi^{-1}$.
% or equivalently 
g(t,a) = \widetilde g \circ \varphi(t,a).$$

%As in every nonparametric smoothing method, the accuracy of an estimator is linked to the pointwise smoothness of the target function $g(t,a)$.

\begin{prop} \label{P first smoothness}
%Let $\eta = \min\{\alpha_1,\alpha_2,\beta_1,\beta_2,\gamma\}$, $\theta = \min\{\beta_1,\beta_2,\gamma\}$, $\kappa = \min\{\beta_2,\gamma\}$. 
Work under Assumptions \ref{H basic}, and \ref{H smoothness basic}. 
\begin{enumerate}
\item[(i)] We have
%\begin{itemize}
%\item[(i)] We have $g \in \mathcal H^{\min\{\alpha,\beta,\gamma,\delta,\nu\},\min\{\alpha,\beta,\gamma,\delta,\nu\}}(t,a)$ for $(t,a) \in \mathcal D_U$
%\item[(ii)] We have $g \in \mathcal H^{\min\{\gamma, \delta, \nu\}, \min\{\delta, \nu\}}(t,a)$ for $(t,a) \in \mathcal D_L$
%\end{itemize}
%$$g{\bf 1}_{\mathcal U}U \in \mathcal H\;\;\text{and}\;\;g{\bf 1}_{\mathcal U}U \in \mathcal H$$
%and
$$g \in \mathcal H^{\min(\alpha,\beta,\gamma+1,\delta),\min(\alpha,\beta,\gamma+1,\delta)}
%(t,a)\;\;\text{for}\;\;(t,a) \in \mathcal D_L$$
%(t,a)\;\;\text{for}\;\;(t,a) 
\;\;\text{on}\;\;\mathcal D_L
%$$
%and
%$$
\;\;\text{and}\;\;
g \in \mathcal H^{\min(\gamma+1, \delta), \max(\gamma\wedge(\delta + 1),\delta)}
%(t,a)
%\;\;\text{for}\;\;(t,a) \in \mathcal D_U.$$
\;\;\text{on}\;\;\mathcal D_U.$$
\item[(ii)] We have the following improvement of the anisotropic smoothness  when the parametrisation is given by $\widetilde g$:
$$\widetilde g \in \mathcal H^{\min(\gamma+1,\delta + 1), \min(\alpha,\beta,\gamma+1,\delta)}
%(t,a)\;\;\text{for}\;\;(t,a) \in \widetilde{\mathcal D}_L$$
\;\;\text{on}\;\;\widetilde{\mathcal D}_L
%$$
\;\;\text{and}\;\;
%and
%$$
\widetilde g \in \mathcal H^{\min(\gamma+1,\delta + 1), \max(\gamma\wedge(\delta + 1),\delta)}
%(t,a)\;\;\text{for}\;\;(t,a) \in \widetilde{\mathcal D}_U.
\;\;\text{on}\;\;\widetilde{\mathcal D}_U.$$
\end{enumerate}
\end{prop}

The proof of Proposition \ref{P first smoothness} is relatively straightforward, given explicit representations of the solution $g$ in terms of $b$, $\mu$ and $g_0$, and is given in Appendix \ref{proof of prop P first smoothness}. 
%Our next result shows that the result of Proposition \ref{P first smoothness} can be improved  in the directions of the characteristics of the transport. %\begin{align*}
%\widetilde g & : \widetilde{\mathcal D}_U \cup \widetilde{\mathcal D}_L \rightarrow \R_+ \\ 
%(t,a) & \mapsto \widetilde g(t,a)= g \circ \psi^{-1}(t,a) 
%\end{align*}
%or equivalently $g(t,a) = \widetilde g \circ \psi(t,a)$.
%$$\widetilde g(t,a)= g \circ \psi^{-1}(t,a)\;\;\text{or equivalently}\;\;g(t,a) = \widetilde g \circ \psi(t,a).$$
%$$g(t,a) = \widetilde g(t,t-a) = \widetilde g \circ \psi(t,a)$$
%or equivalently $\widetilde g(t,a)= g \circ \psi^{-1}(t,a)$.

\begin{figure}[h]
   \centering
   \setlength{\unitlength}{0.8cm}
\begin{picture}(6,6)(-1,-1)
% axe x
   \put(0,-1){\vector(0,1){6}}
   \multiput(4.5,-0.2)(1,0){1}{\line(0,1){0.4}} % traits de graduation
     \put(4.4,-0.6){$T$} % valeurs graduées
     \put(5.6,-0.4){$t$}
     \put(1,3){\shortstack{$\mathcal D_U$}}
     %\put(3,1){\shortstack{$\mathcal D_L$}}
      \put(1,-1){\shortstack{$\widetilde{\mathcal D}_U$}}
        \put(2,1){\shortstack{$\widetilde{\mathcal D}_L = \mathcal D_L$}}
     \put(4.9,2.2){\shortstack{{\small $\xi_T(da)=g(T,a)da$}}}
     \put(-3.8,2.2){\shortstack{$\xi_0(da)=g_0(a)da$}}
% axe y
   \put(-1,0){\vector(1,0){7}}
   \put(-0.4,-0.4){0} % valeurs graduées
   \put(-0.6,4.7){$a$}
% quadrillage
   \multiput(4.5,0)(1,0){1}{
      \multiput(0,0)(0,0.2){25}{\circle*{0.05}}}
    \put(-0.3,2.2){\shortstack{$\left\lbrace \begin{array}{l} \\ \\ \\ \\ \\ \\ \\ \\ \end{array}\right.$}}
% demi-droite
   \put(0,0){\line(1,1){4.8}}
   \color{red}{
   \multiput(0,0)(0.25,0.25){18}{\circle*{0.1}}}
   
   \put(0,0){\line(1,0){4.5}}
   \color{red}{
   \multiput(0,0)(0.25,0){18}{\circle*{0.1}}}

\end{picture}
   \caption{{\small $\widetilde{g}\in \mathcal H^{\min(\gamma+1,\delta + 1), \min(\alpha,\beta,\gamma+1,\delta)}$ on $\widetilde{\mathcal D}_L$ and 
  $\widetilde g \in \mathcal H^{\min(\gamma+1,\delta + 1),\min(\gamma,\delta+1)}$ on $\widetilde{\mathcal D}_U$}.}
\end{figure}
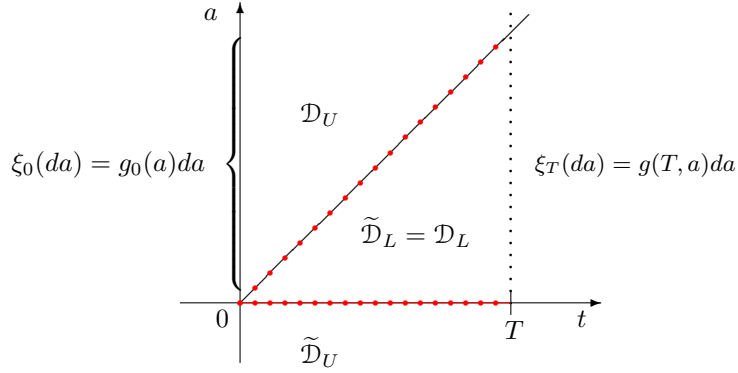

%\begin{prop} \label{P improved smoothness} 
%%MLet $\lambda = \min\{\alpha_1 + 1,\alpha_2,\beta_1,\beta_2, \gamma\}, \nu = \min\{\beta_1,\beta_2+ 1\}$ and $\rho=\min\{\beta_1,\beta_2,\gamma\}$.
%%Work under Assumptions \ref{H basic}, \ref{H stability}, \ref{H initial limit} and \ref{H smoothness basic}. 
%In the same setting of Proposition \ref{P first smoothness}, we have
%$$\widetilde g \in \mathcal H^{\min(\gamma,\delta + 1), \min(\alpha,\beta,\gamma,\delta)}(t,a)\;\;\text{for}\;\;(t,a) \in \widetilde{\mathcal D}_U$$
%%\;\;\text{and}\;\;
%and
%$$
%\widetilde g \in \mathcal H^{\min(\gamma,\delta + 1), \delta}(t,a)\;\;\text{for}\;\;(t,a) \in \widetilde{\mathcal D}_L.
%$$
%%and
%%$$\widetilde g \in \mathcal H^{\kappa, \rho}(t,a)\;\;\text{for}\;\;(t,a) \in \widetilde{\mathcal D}_L.$$
%\end{prop}
%Proposition \ref{P improved smoothness} has some consequences for the accuracy of kernel approximations as shown in the next section.
%

\subsection{Minimax lower bounds}
For $\alpha,\beta>0$ and $L>0$, we set
$$\mathcal H^{\alpha, \beta}_L(t,a) = \big\{f \in \mathcal L_{\mathcal D}^\infty ,\;|f|_\infty+|f|_{\mathcal H^{\alpha,\beta}(t,a)} \leq L\big\},$$
where the semi-norm $|\cdot|_{\mathcal H^{\alpha,\beta}(t,a)}$ is defined after Definition \ref{def anisotropic space}.
% \eqref{def semi-norm anisotropic}. 
%{\color{blue} (need uniformity of the neighbourhood)} 
We also set, for $\epsilon >0$,
$$\mathcal L_{\mathcal D,\epsilon}^\infty = \big\{f \in \mathcal L_{\mathcal D}^\infty , \inf_{(t,a) \in \mathcal D}f(t,a) \geq \epsilon\big\}$$ 

Remember that under Assumption \ref{H basic}, any point $(b,\mu,g_0)$ with $b,\mu,g_0 \in \mathcal L_{\mathcal D}^\infty $ 
defines a unique solution $g$ to the McKendrick Von Foester equation \eqref{McKendrick}. Let
% {\color{blue}
%$$s_{\mathrm{dens}}^-(t,a) = \min(\gamma,\delta+1) {\bf 1}_{\mathcal D_U}(t,a)  + \min(\beta,\delta){\bf 1}_{\mathcal D_L}(t,a)$$
%and}  
% {\color{blue}
$$s_{\mathrm{dens}}^- = \max(\gamma,\delta)
% {\bf 1}_{\mathcal D_U}(t,a)  + \min(\beta,\delta){\bf 1}_{\mathcal D_L}(t,a)$$
\;\;\text{and}\;\;  
s_{\mathrm{death}}^-= \big(\gamma^{-1} + \delta^{-1}\big)^{-1}.$$
Under a non-degeneracy condition of the form $\mu \in \mathcal L_{\mathcal D,\epsilon}^\infty$, we obtain the following minimax lower bound:
\begin{thm} \label{thm LB}
Work under Assumptions \ref{H basic} and \ref{assumption minoration g}. Let $\alpha, \beta, \gamma, \delta >0, \nu \geq \max(\gamma,\delta)+1$ and $L>0$. 
%There exists $C>0$ such that 
For every $(t,a) \in \mathcal D^-$, we have
\begin{equation} \label{eq LB g}
\inf_{F}\sup_{b,\mu,g_0}\E\big[|F-g(t,a)|\big] \gtrsim N^{-s_{\mathrm{dens}}^-/(2s_{\mathrm{dens}}^-+1)}
\end{equation} 
and
\begin{equation} \label{eq LB mu}
\inf_{F}\sup_{b,\mu,g_0}\E\big[|F-\mu(t,a)|\big]  \gtrsim N^{-s_{\mathrm{death}}^-/(2s_{\mathrm{death}}^-+1)},
\end{equation}
where the infimum is taken over all estimators and the supremum over 
$$b \in \mathcal H^{\alpha, \beta}_{L}(t,a),\;\;
\mu \in \mathcal H^{\gamma, \delta}_{L}(t,a) \cap \mathcal L_{\mathcal D, \epsilon}^\infty\;\; 
\text{and}\;\;g_0 \in \mathcal H^\nu_{L}(t,a).$$
%with $\alpha>0, \beta>0, \gamma>1, \delta>0, \nu > \min(\alpha,\beta, \gamma,\delta)$ . 
\end{thm}

Some remarks: {\bf 1)} As for the previous estimation results in Theorems \ref{thm oracle g} and \ref{thm oracle mu}, a glance at the proof shows that the lower bound actually holds for a wider class of loss functions, including loss in probability. We keep up to the statements \eqref{eq LB g} and \eqref{eq LB mu} in expected pointwise absolute value for simplicity. {\bf 2)} If we take $\gamma = \delta$ for simplicity, we see that $s_{\mathrm{dens}}^- = \gamma$ while $s_{\mathrm{death}}^- =\gamma/2$. Therefore, although we are estimating bi-variate functions, the estimation difficulty for $g(t,a)$ is really that of a 1-dimensional function while the estimation of $\mu(t,a)$ remains that of a genuinely bivariate function. Heuristically, there is no information about the population density $g(t,a)$ captured by $(Z^N_s, s \neq t)$ while the estimation of the death rate $\mu(t,a)$ requires dynamical knowledge from the process $\Gamma^N(ds,du)$ for which a truly 2-dimensional information domain around $(t,a)$ is required in order to identify $\mu(t,a)$. 

%\begin{thm}
%Work under Assumptions \ref{H basic}, \ref{H stability}, \ref{H initial limit}. Let $\alpha, \beta, \gamma, \delta >0, \nu \geq \delta$ and $L>0$. 
%Define $s(\gamma, \delta)$ by
%$$s(\gamma, \delta)^{-1} = \gamma^{-1} + \delta^{-1}.$$
%%There exists $C>0$ such that 
%For every $(t,a) \in \mathcal D_U \cup \mathcal D_L$, there exists $C>0$ such that
%%\begin{equation} \label{LB upper}
%$$
%\liminf_{N \rightarrow \infty} \inf_{F}\sup_{b,\mu,g_0}\PP\big(N^{s(\gamma,\delta)/(2s(\gamma,\delta)+1)}|F-\mu(t,a)| \geq C\big) >0,
%$$
%%\end{equation} 
%%For every $(t,a) \in \mathcal D_U$, there exists $C>0$ such that
%%\begin{equation} \label{LB lower}
%%\liminf_{N \rightarrow \infty} \inf_{F}\sup_{b,\mu,g_0}\PP\big(N^{\min(\beta,\delta)/2\min(\beta,\delta)+1)}|F-f(t,a)| \geq C\big) >0.
%%\end{equation}
%%\end{enumerate}
%where the infimum is taken over all estimators and the supremum over 
%$b \in \mathcal H^{\alpha, \beta}_{L}(t,a)$, $\mu \in \mathcal H^{\gamma, \delta}_{L}(t,a)$, $g_0 \in \mathcal H^\nu_{L}(t,a)$.
%\end{thm}

\subsection{Adaptive estimation under anisotropic H\"older smoothness} \label{sec: minimax optimality}

Our next result shows the performance of $g_{\star}^N(t,a)$ defined in \eqref{def GL est} and gives optimal up to inessential logarithmic factors in some cases.
%, and also on $\mathcal D_L$ provided {\color{blue} (check!!) $\alpha \geq \beta$ and $\gamma \geq \delta$}. 
Moreover, $g_{\star}^N(t,a)$ is nearly smoothness adaptive. More precisely, let
\begin{equation} \label{eq def dens}
s_{\mathrm{dens}}^+(t,a) =\max(\gamma\wedge(\delta + 1),\delta) {\bf 1}_{\mathcal D_U}(t,a)  + \min(\alpha,\beta,\gamma+1,\delta){\bf 1}_{\mathcal D_L}(t,a){\color{blue},}
\end{equation}
and note that $s_{\mathrm{dens}}^+(t,a) \leq s_{\mathrm{dens}}^-(t,a)$ always.
%\E\big[\big(g_{\star}^N(t,a)-g(t,a)\big)^2\big]

\begin{thm} \label{adapt esti g}
% {\color{red} (to be rewritten !!!)}
Work under Assumptions  \ref{H basic}, \ref{minimal F}, \ref{concentration initiale} with $r_N \leq N^{-1/2}$ and some $\mathcal F$ that satisfies $e(\mathcal F) < \infty$, and Assumption \ref{assumption minoration g}. 
% Assume moreover that $\mathcal F \supset \mathcal C^2_L(\mathcal D)$ for some $L>0$.
% {\color{red} {\tt [Improve: initial approximation!!!]}}
Specify $\widehat g_{\star}^N(t,a)$ with a compactly supported kernel of order $\ell_0 \geq 0$ and pick 
$$\mathcal G_1^N = (x_1^N < x_2^N < \ldots < x_N^N)$$
a subdivision of $[N^{-1/2}, (\log N)^{-1}]$ with $\max_{1 \leq i \leq N-1} (x_{i+1}^N-x_{i}^N) \lesssim N^{-1}$ so that 
$\mathrm{Card} \,\mathcal G_1^N \lesssim N$.
% and $\mathcal G_1^N = \{2+i\tfrac{(\ell_0-2)_+}{N}, i=1,\ldots, N\}$.
%For large enough $N$, we have
%\begin{enumerate}
%\item[(i)]
For every $(t,a) \in \mathcal D_-$ and large enough $N$, we have
%there exists $C>0$ such that
\begin{equation} \label{UB upper}
\sup_{b,\mu,g_0}\big(\E\big[\big(\widehat g_{\star}^N(t,a)-g(t,a)\big)^2\big]\big)^{1/2} \lesssim \Big(\frac{(\log N)^2}{N}\Big)^{s_{\mathrm{dens}}^+(t,a) \wedge \ell_0/(2s_{\mathrm{dens}}^+(t,a) \wedge \ell_0+1)},
\end{equation} 
%\item[(ii)]
%For every $(t,a) \in \mathcal D_U$ and large enough $N$, we have
%\begin{equation} \label{UB lower}
%\sup_{b,\mu,g_0}\big(\E\big[\big(g_{\star}^N(t,a)-g(t,a)\big)^2\big]\big)^{1/2} \lesssim \big(\log N \max(w_N^2,N^{-1})\big)^{\min(\alpha,\beta,\delta,\gamma)\wedge \ell_0/(2\min(\alpha,\beta,\gamma,\delta)\wedge \ell_0+1)}.
%\end{equation}
%%\end{enumerate}
%In both \eqref{UB upper} and \eqref{UB lower}, 
where the supremum is taken over 
$b \in \mathcal H^{\alpha, \beta}_{L}(t,a)$, $\mu \in \mathcal H^{\gamma, \delta}_{L}(t,a)$, $g_0 \in \mathcal H^\nu_{L}(t,a)$
with $\alpha, \beta, \gamma, \delta >0, \nu  \geq \max(\gamma,\delta)+1$ and $L>0$.
\end{thm}
Some remarks: {\bf 1)} Comparing with the minimax lower bound of Theorem \ref{thm LB}, we see that both upper and lower bounds \eqref{eq LB g} and \eqref{UB upper} agree on $\mathcal D_U$ if $\delta \leq \gamma \leq \delta+1$ and on $\mathcal D_L$ if $\delta-1 \leq \gamma \leq \delta$ (and if $\alpha$ and $
\beta$ are sufficiently large too), provided the order $\ell_0$ of the kernel $K$ is sufficiently large. The rates are tight up to an inessential logarithmic factor. We do not know about the optimality in $g$ beyond this domain, but we see that the difficulty of the estimation of $g(t,a)$ is equivalent to the difficulty of the univariate function $a \mapsto g(t,a)$ for which the time variable $t$ is simply a parameter: it suffices to piece together the estimators $\widehat g_\star^N(t,a)$ for every $t$ in order to estimate the graph $(t,a) \mapsto g(t,a)$. {\bf 2)} While we already know that a logarithmic payment is unavoidable for a smoothness adaptive estimator (see the classical Lepski-Low phenomenon, \cite{LEPSKI, LOW}) we do not know whether the order we find in the log term is correct ({\it i.e.} $(\log N)^2$ versus the classical $\log N$ payment). This stems from Theorem \ref{thm concentration optimal} and the mild concentration property as we define it, where exponential tail are obtained versus subgaussian tails, but this order seems genuinely linked to the Poissonian behaviour of the noise and it is not clear that we can extend our statistical result in order to remove the extra $\log N$ error-term  in \eqref{UB upper}.\\

Similarly, $\mu_{\star}^N(t,a)$ defined in \eqref{def GL est mu} also 
%{\color{blue} (note the counterintuitive no contradiction with $\gamma+1$ in the upper bound!!)} 
shares near optimality in some cases.
%{\color{blue} on $\mathcal D_U$ if $\delta \geq \gamma+1$ and on $\mathcal D_L$ provided $b$ is regular enough}.
 Define 
%$$s_L(\alpha,\beta,\gamma,\delta) = \Big(\frac{1}{\min(\gamma,\delta)}+\frac{1}{\min(\alpha,\beta,\gamma+1,\delta)}\Big)^{-1},
%s_U(\gamma,\delta) = \min(\gamma,\delta)/2,
%$$
$$s_L(\alpha,\beta,\gamma,\delta) = \big(\min(\gamma,\delta)^{-1}+\min(\alpha,\beta,\gamma+1,\delta)^{-1}\big)^{-1},$$
%\end{equation}
%\begin{equation}
%\label{eq def sU}
$$s_U(\gamma,\delta) = \big(\min(\gamma,\delta)^{-1}+\delta^{-1}\big)^{-1},
$$
and
\begin{equation} \label{eq def sdeath}
s_{\mathrm{death}}^+(t,a) = s_U(\gamma, \delta) {\bf 1}_{\mathcal D_U}(t,a)+s_L(\alpha,\beta,\gamma, \delta) {\bf 1}_{\mathcal D_L}(t,a).
\end{equation}
Note that $s_{\mathrm{death}}^+(t,a) \leq s_{\mathrm{death}}^-$ always.
% $\alpha \geq \beta$ and $\gamma \geq \delta$. More precisely
%\E\big[\big(g_{\star}^N(t,a)-g(t,a)\big)^2\big]

\begin{thm} \label{th adapt minimax mu}
Work under Assumptions  \ref{H basic}, \ref{minimal F}, \ref{concentration initiale} with $r_N \leq N^{-1/2}$ and some $\mathcal F$ that satisfies $e(\mathcal F) < \infty$, and Assumption \ref{assumption minoration g}. %Assume moreover that $\mathcal F \supset \mathcal C^2_L(\mathcal D)$ for some $L>0$. 
Specify $\mu_{\star}^N(t,a)$ with  kernels $H,K$ of order $\ell_0 \geq 0$ and pick $\mathcal G_2^N = \mathcal G_1^N \times \mathcal G_1^N$
%a subdivision of $[N^{-1/2}, (\log N)^{-1}]^2$
% with {\color{red}mesh $\max_{1 \leq i \leq N-1} (t_{i+1}^N-t_{i}^N) \lesssim N^{-1}$} 
so that 
$\mathrm{Card} \,\mathcal G_2^N \lesssim N^2$.
%$$s_L(\alpha,\beta,\gamma,\delta)^{-1} = \delta^{-1}+\min(\alpha,\beta,\gamma,\delta)^{-1}.$$
%For large enough $N$, we have
%\begin{enumerate}
%\item[(i)]
For every $(t,a) \in \mathcal D^-$
%_U \cup \mathcal D_L$ 
and large enough $N$, we have
%there exists $C>0$ such that
\begin{equation} \label{UB upper mu}
\sup_{b,\mu,g_0}\big(\E\big[\big(\widehat \mu_{\star}^N(t,a)-\mu(t,a)\big)^2\big]\big)^{1/2} \lesssim \Big(\frac{(\log N)^2}{N}\Big)^{s_{\mathrm{death}}^+(t,a)\wedge \ell_0/(2s_{\mathrm{death}}^+(t,a)\wedge \ell_0+1)},
\end{equation} 
%\item[(ii)]
%and for every $(t,a) \in \mathcal D_L$ and large enough $N$, we have
%\begin{equation} \label{UB lower}
%$$
%\sup_{b,\mu,g_0}\big(\E\big[\big(\mu_{\star}^N(t,a)-\mu(t,a)\big)^2\big]\big)^{1/2} \lesssim \big(\log N \max(w_N^2,N^{-1})\big)^{s_L(\gamma,\delta)\wedge \ell_0/(2s_L(\gamma,\delta)\wedge \ell_0+1)}$$
%\end{equation}
%\end{enumerate}
where the supremum is taken over 
$b \in \mathcal H^{\alpha, \beta}_{L}(t,a)$, $\mu \in \mathcal H^{\gamma, \delta}_{L}(t,a)$, $g_0 \in \mathcal H^\nu_{L}(t,a)$,
with $\alpha, \beta, \gamma, \delta >0, \nu \geq \max(\gamma,\delta)+1$ and $L>0$.
\end{thm}

Some remarks: {\bf 1)} The same remark as {\bf 2)} after the statement or Theorem \ref{adapt esti g} holds here. {\bf 2)} The minimax optimality situation is somewhat clearer for estimating $\mu$: we see that we have near optimality on $\mathcal D_U$ as soon as $\gamma \leq \delta$, while the upper and lower bounds only agree if $\gamma \leq \delta \leq \gamma+1$ on $\mathcal D_L$  (and if $\alpha$ and $
\beta$ are sufficiently large too), provided the order $\ell_0$ of the kernel $K$ is sufficiently large. Thus situation is somewhat similar to the estimation of $g$ on $\mathcal D_U$, see Theorem \ref{adapt esti g} above. {\bf 3)} The rate of estimation is triggered by the smoothness of $\pi = \mu g$ since the estimation of the quotient $g$ will always be better, for 
$$s_{\mathrm{death}}^+(t,a) \leq s_{\mathrm{dens}}^+(t,a)\;\;\text{for every}\;\;(t,a) \in \mathcal D_-$$
always. However, in order to achieve optimality, we need to optimise the approximation property of $\pi$ by looking at the smoothness of $\widetilde \pi = \widetilde \mu \widetilde g$, with $\mu = \widetilde \mu \circ \varphi$. This benefit is obtained thanks to Proposition \ref{P first smoothness} and is given in details in the proof. We would lose by a polynomial order in the rate of convergence given in \eqref{UB upper mu} if we used a kernel of the form $(H \otimes K)_{\boldsymbol h}$ instead of $(H \otimes K)_{\boldsymbol h}\circ \varphi$ for the estimation of the numerator $\pi$ in the representation $\mu = \pi/g$.

\section{Numerical illustration}
We briefly explore the performance of our estimators on simulated data. 
We use the following parameters:
\begin{enumerate}
\item[(i)] The initial condition $g_0$ is taken as the density of Gaussian random variable centred in 40 with variance of 152 ({\it i.e.} a standard deviation of approximately 12 years), conditioned on living between $0$ and $120$.
\item[(ii)] We pick $b(t, a) = {\bf 1}_{\{120 \leq a \leq 40\}}$. Although $b$ is not globally H\"older continuous, we still have (and can prove) similar results for such simple piecewise constant functions.
\item[(iii)] We pick $\mu(t, a) = 4\cdot 10^{-2} \exp(7.4 \cdot 10^{-3}a)\exp( - 5 \cdot 10^{-3}t)$. We pick a relatively high death rate in order to guarantee sufficiently many events of death for the estimation of $\mu(t,a)$ and avoid artefacts.
\end{enumerate}
We consider the domain $\mathcal D = [0, 20] \times [0, 120]$ which means  $T = 20$ and a maximal possible age of $120$.
We estimate $g$ on the grid $\mathcal G^g = \mathcal T^g \times \mathcal A^g$, with $\mathcal T^g = \{k \times 1.005, 0 \leq k < 20\}$ and $\mathcal A^g = \{k \times 0.2002, 0 \leq k < 600\}$. We estimate the functions $\mu$ and $\pi = \mu \cdot g$ on the grids $\mathcal G^\mu = \mathcal T^\mu \times \mathcal A^\mu$, with $\mathcal T^\mu =\mathcal T^g$ and $\mathcal A^\mu = \{k\times 1.0008, 0 \leq k<120\}$.\\

We first estimate $g$ and $\pi = \mu \cdot g$ and obtain consistent results in the regime $N = 4000$.

\begin{figure}[h]
\centering
\includegraphics[scale=0.43]{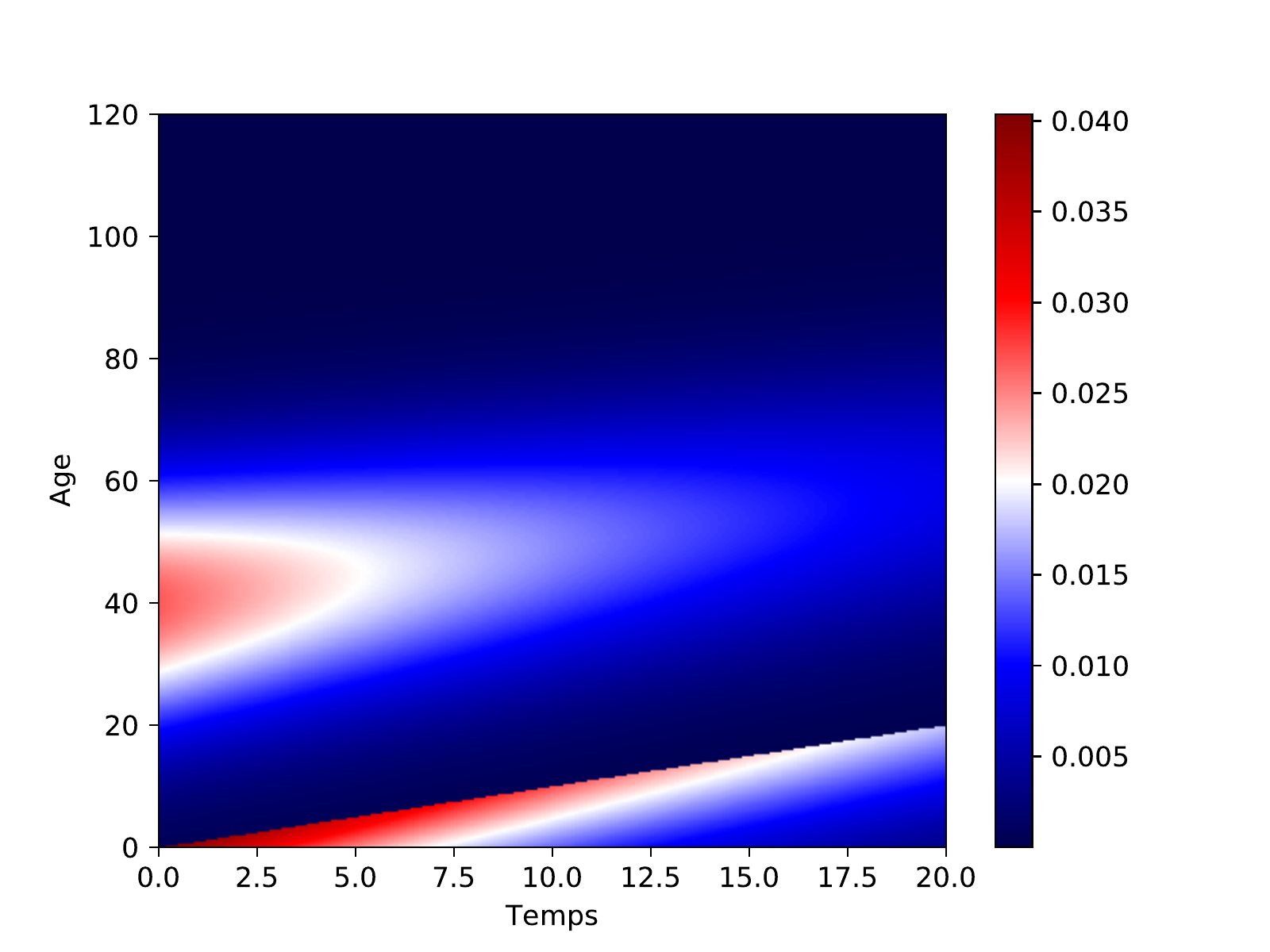}
\includegraphics[scale=0.43]{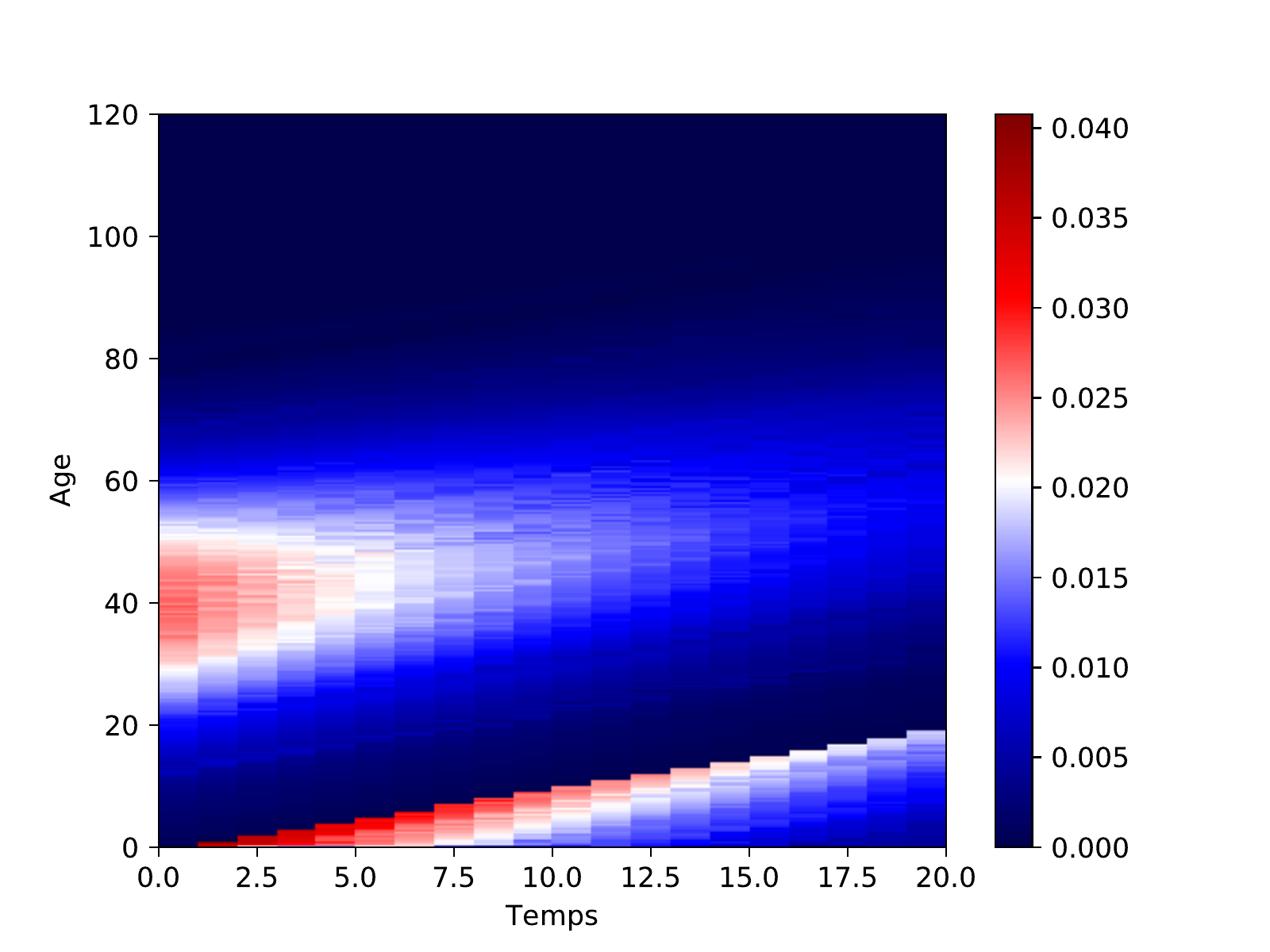}
\caption{\small{Left: true population density $g$. Right: ${\widehat g}^N_{{\widehat h}^N}$ with $N=4000$ over a single simulation of $Z^N$. $X$-axis, $Y$-axis: units in years.}}
\end{figure}

\begin{figure}[h]
\centering
\includegraphics[scale=0.43]{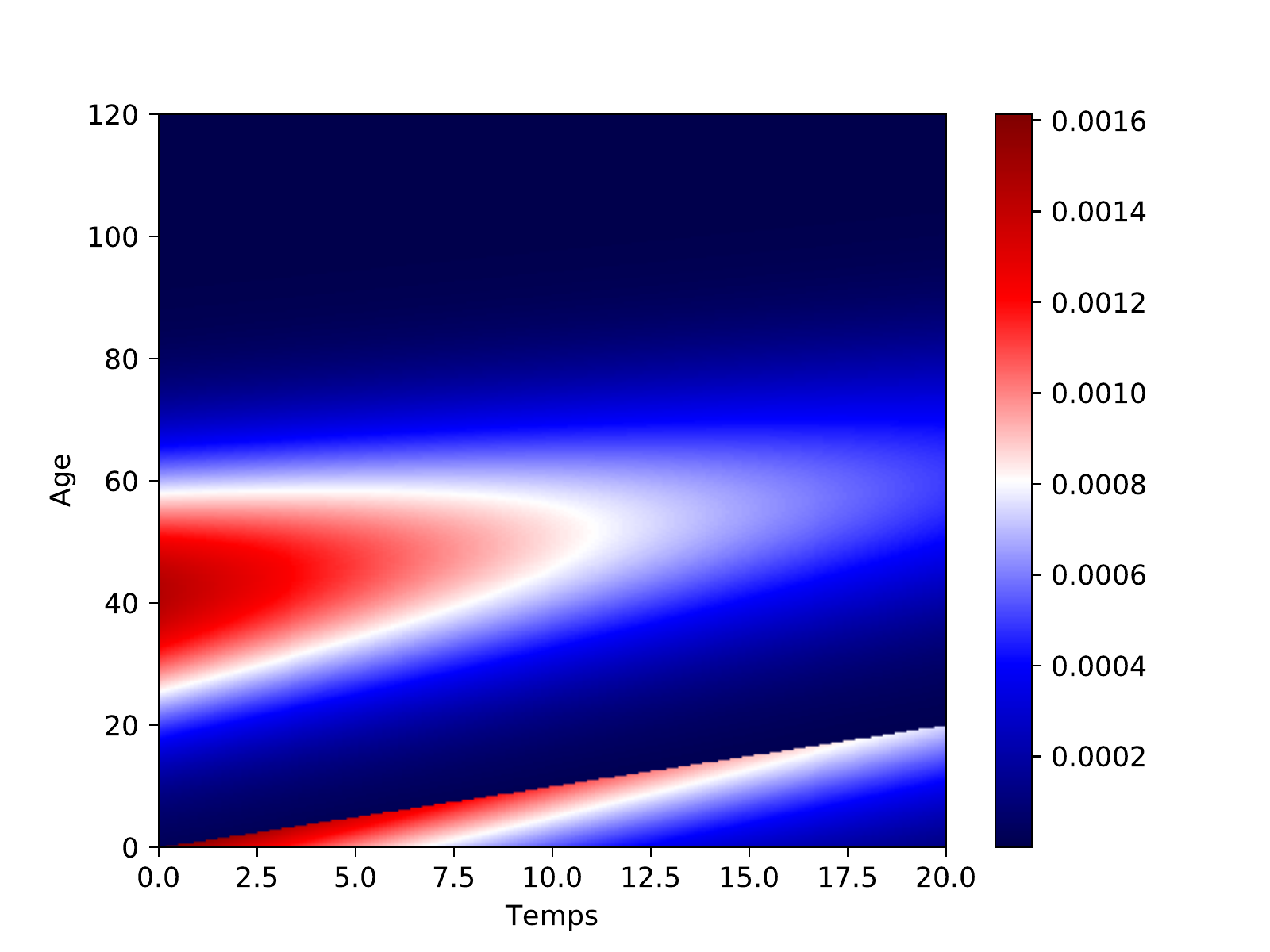}
\includegraphics[scale=0.43]{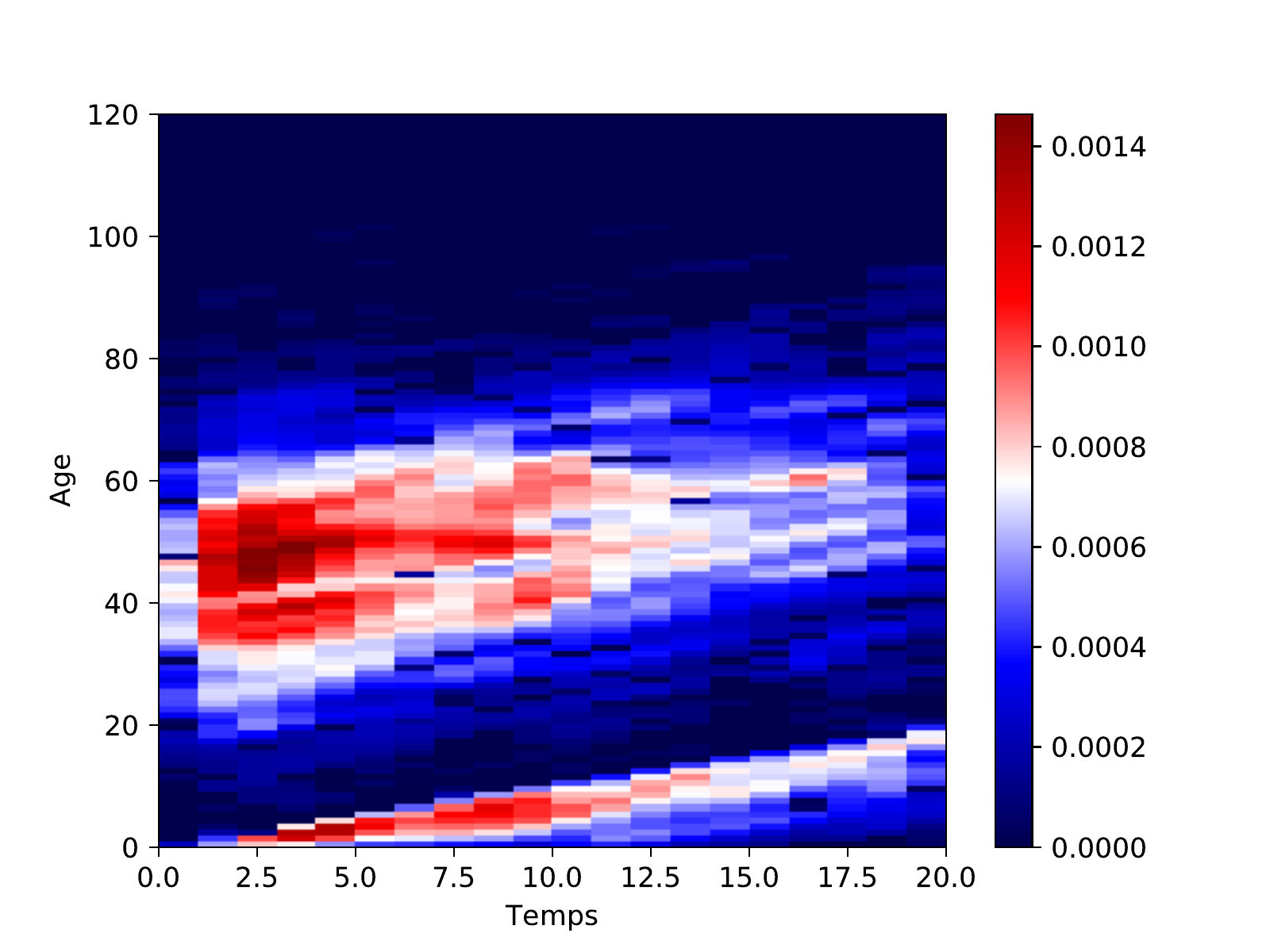}
\caption{\small{Left: true $\pi = \mu \cdot g$. Right: implementation of ${\widehat \pi}^N_{\widehat {\boldsymbol h}^N}$ with $N=4000$. $X$-axis, $Y$-axis: units in years.}}
\end{figure}

%\begin{figure}[h]
%\centering
%\includegraphics[scale=0.43]{TRUEMUGSIM1CHAP2.pdf}
%\includegraphics[scale=0.43]{MUGHATSIM1CHAP2.pdf}
%\caption{\small{Left: true $\pi = \mu \cdot g$. Right: implementation of ${\widehat \pi}^N_{\widehat {\boldsymbol h}^N}$ with $N=4000$. $X$ and $Y$ units in years.}}
%\end{figure}

\begin{figure}[h]
\centering
\includegraphics[scale=0.42]{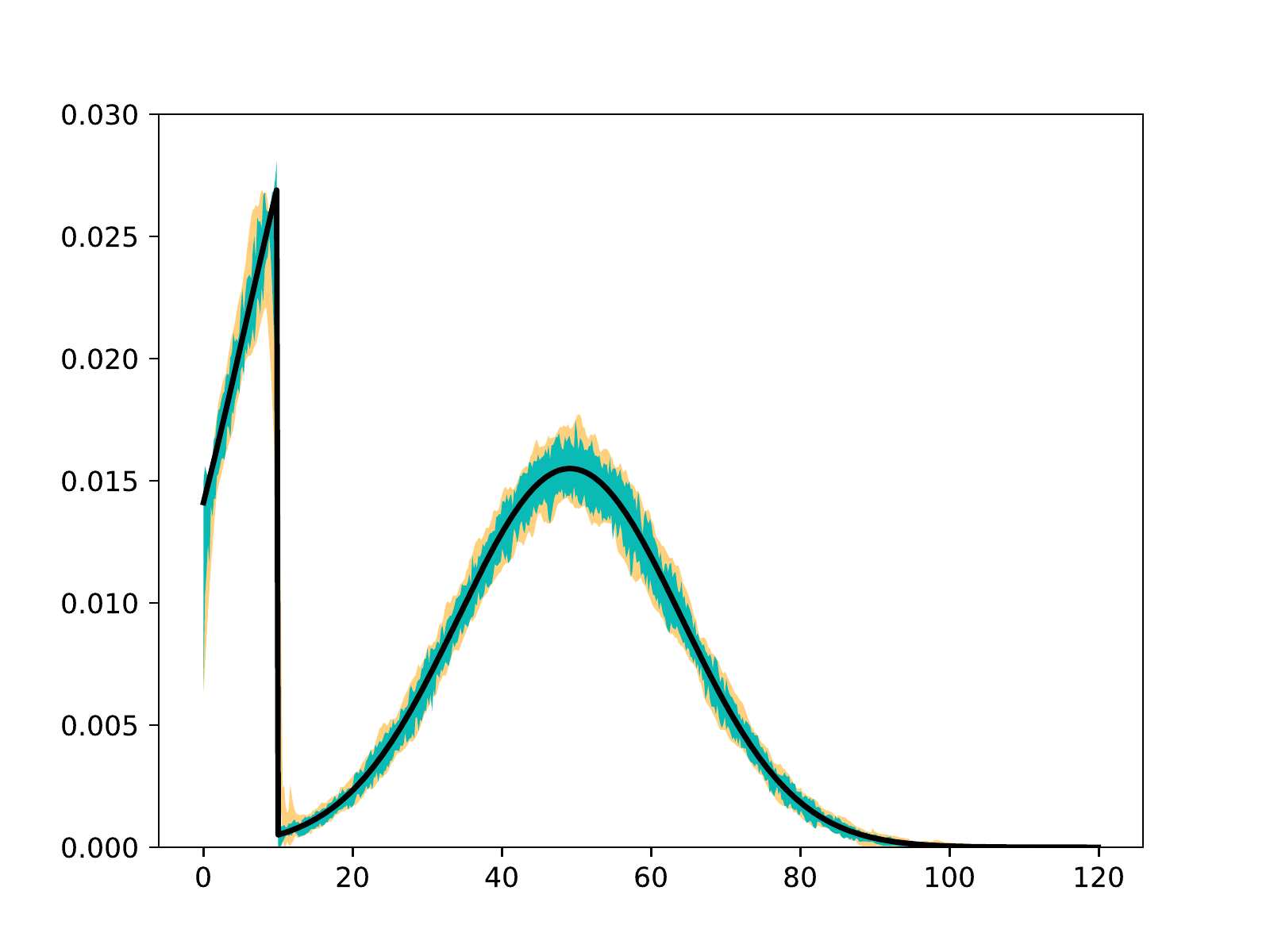}
\includegraphics[scale=0.42]{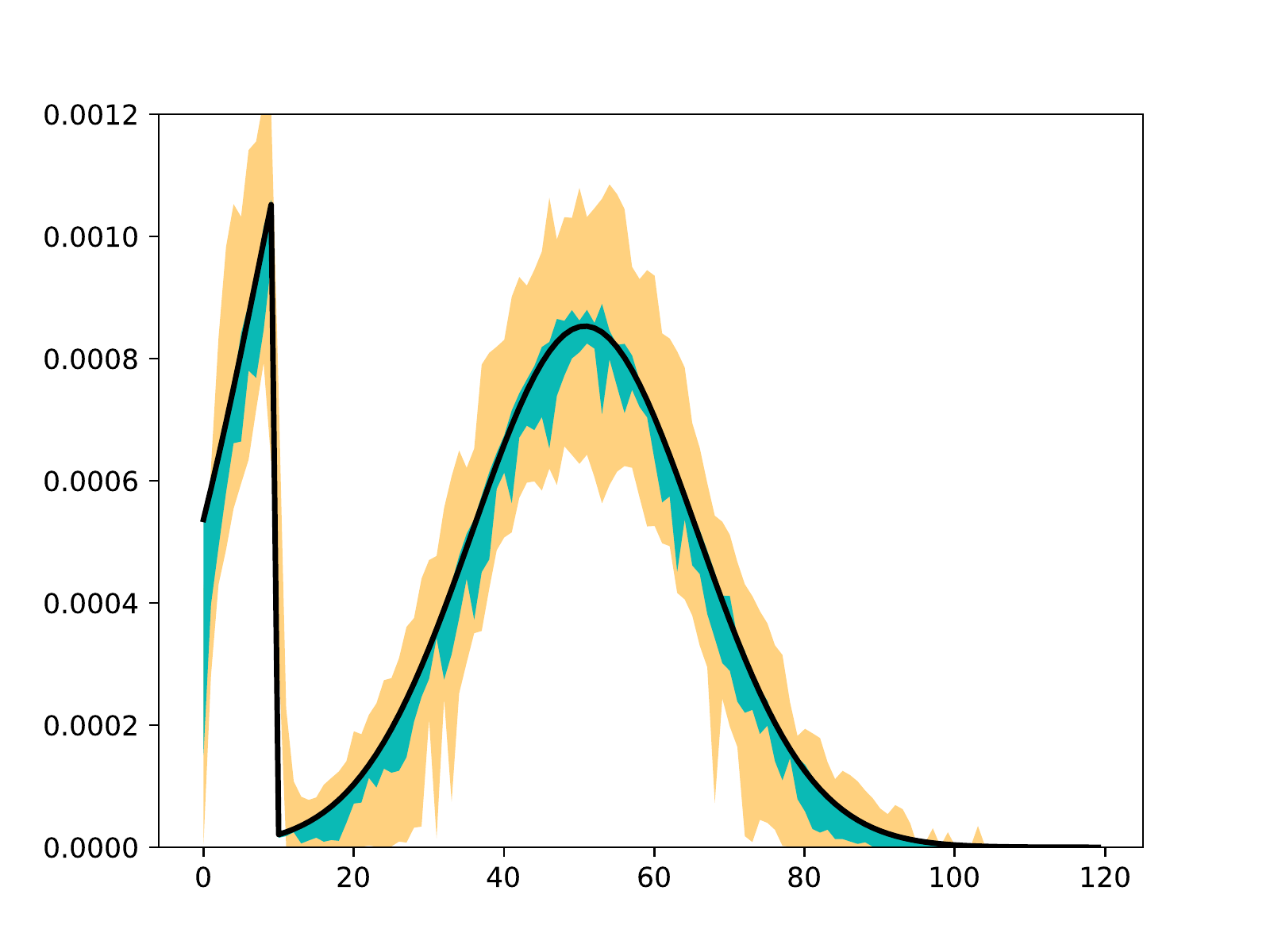}
\caption{\small{For fixed $t = 10$, and $N = 4000$, comparison between the true function (solid black) and a (pointwise) $95\%$ confidence interval based on 50 Monte-Carlo simulations. Oracle estimator in green and our adaptative estimator in yellow. Left: Estimation of $a \mapsto g(10,a)$. Right: Estimation of $a \mapsto \pi(10,a) = \mu(10,a)g(10,a)$. $X$-axis: units in years, Y-axis: rate per unit of time.}}
\end{figure}

\begin{figure}[h]
\centering
\includegraphics[scale=0.43]{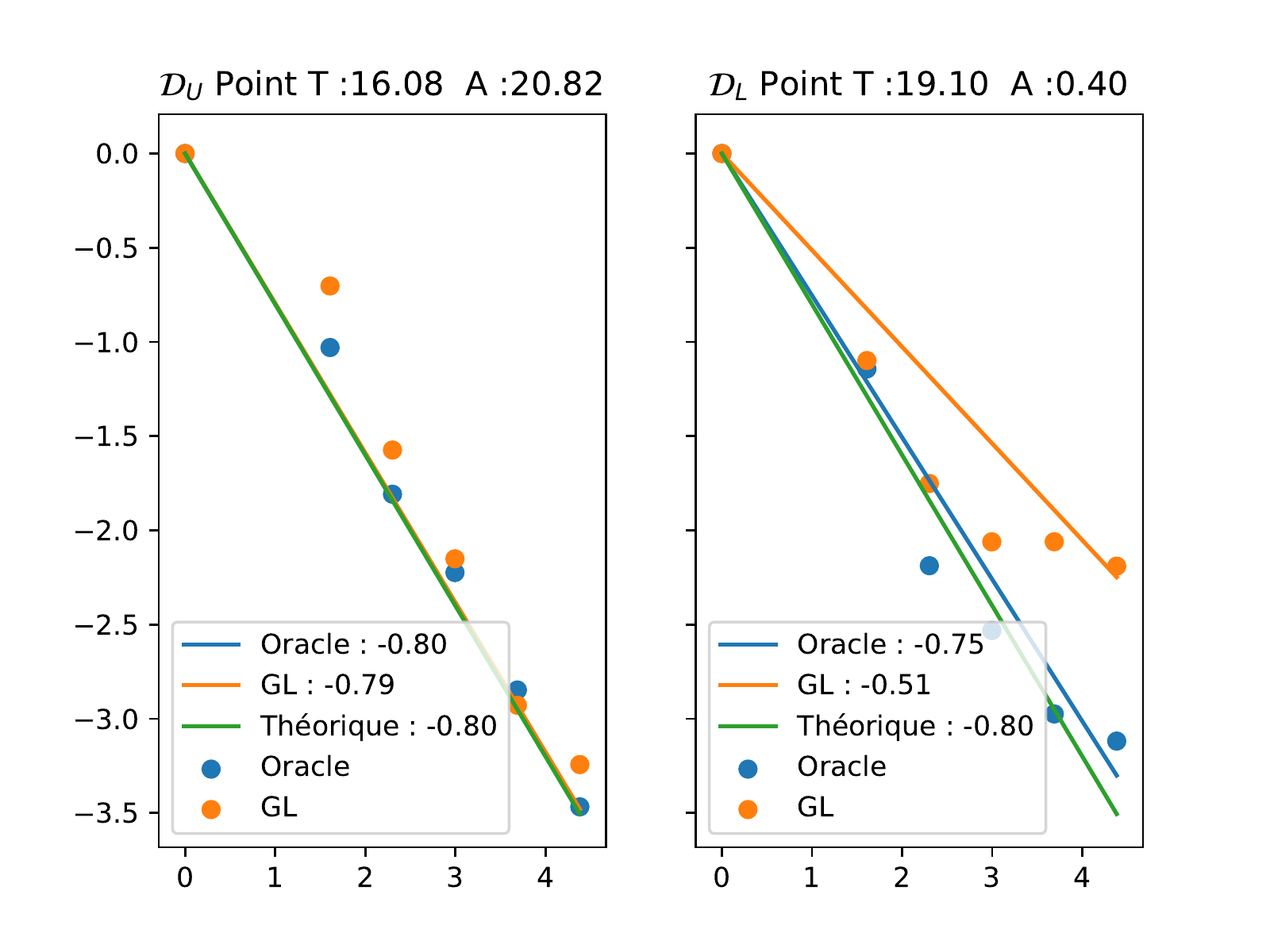}
\includegraphics[scale=0.43]{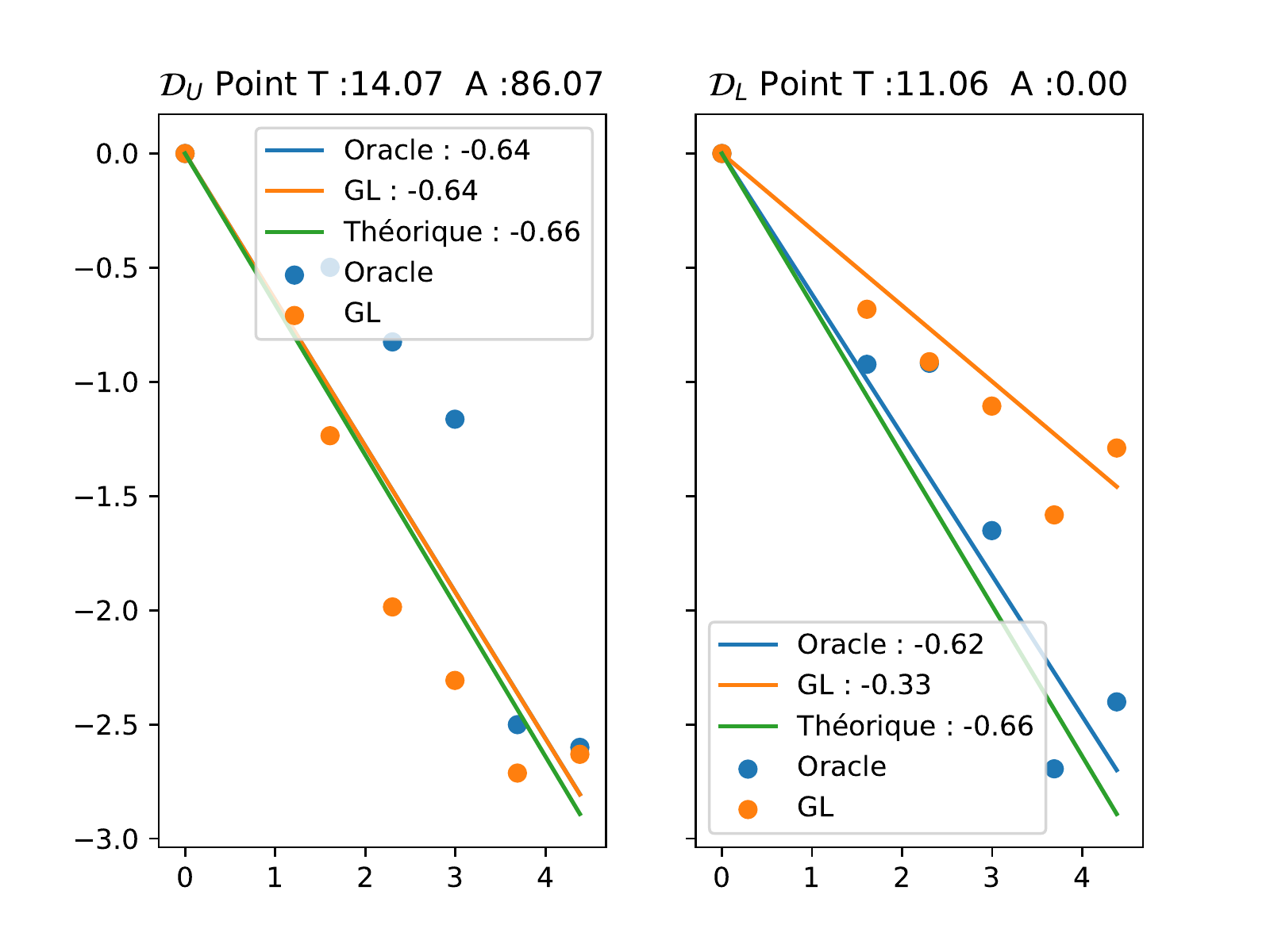}
\caption{\small{$\log_{10}$-integrated error of pointwise estimation against approximatively $\log_{10} N$ based on 50 Monte-carlo simulations. We compare the theoretical line given by the minimax theory (green) with a linear regression given by the oracle estimation (blue) and our estimation method (orange) for different values of $\log_{10} N$ at two given points in $\mathcal D_L$ and $\mathcal D_U$ respectively. Left: estimation of $g(t,a)$ at $(t,a)=(16.08, 20.82)$ and $(t,a) = (19.10, 0.40)$. Right: estimation for $\pi(t,a) = \mu(t,a)g(t,a)$ at $(t,a) = (14.07,86.07)$ and $(t,a)=(11.06, 0.00)$. X-axis: integers $0$ to $5$ correspond to $N=10^2,5\cdot 10^2,10^3,2\cdot 10^3,4\cdot 10^3,8\cdot 10^3$.}}
\end{figure}

We end this section by exploring the estimation of $\mu$ via our quotient estimator.

%\begin{figure}[h]
%\centerings
%%\includegraphics[scale=0.43]{TRUEMUGSIM1CHAP2.pdf}
%%\includegraphics[scale=0.43]{MUGHATSIM1CHAP2.pdf}
%\includegraphics[scale=0.43]{TRUEMUSIM1CHAP2.pdf}
%\includegraphics[scale=0.43]{MUHATSIM1CHAP2.pdf}
%\caption{\small{Left: true $\pi = \mu \cdot g$. Right: implementation of ${\widehat \pi}^N_{\widehat {\boldsymbol h}^N}$ with $N=4000$. $X$ and $Y$ units in years. The estimation is poor, and presumably reflects the effect of the threshold $\varpi$, as suggested by Figure \ref{fig seuil} below.}}
%\end{figure}

\begin{figure}[h] \label{fig seuil}
\centering
\includegraphics[scale=0.42]{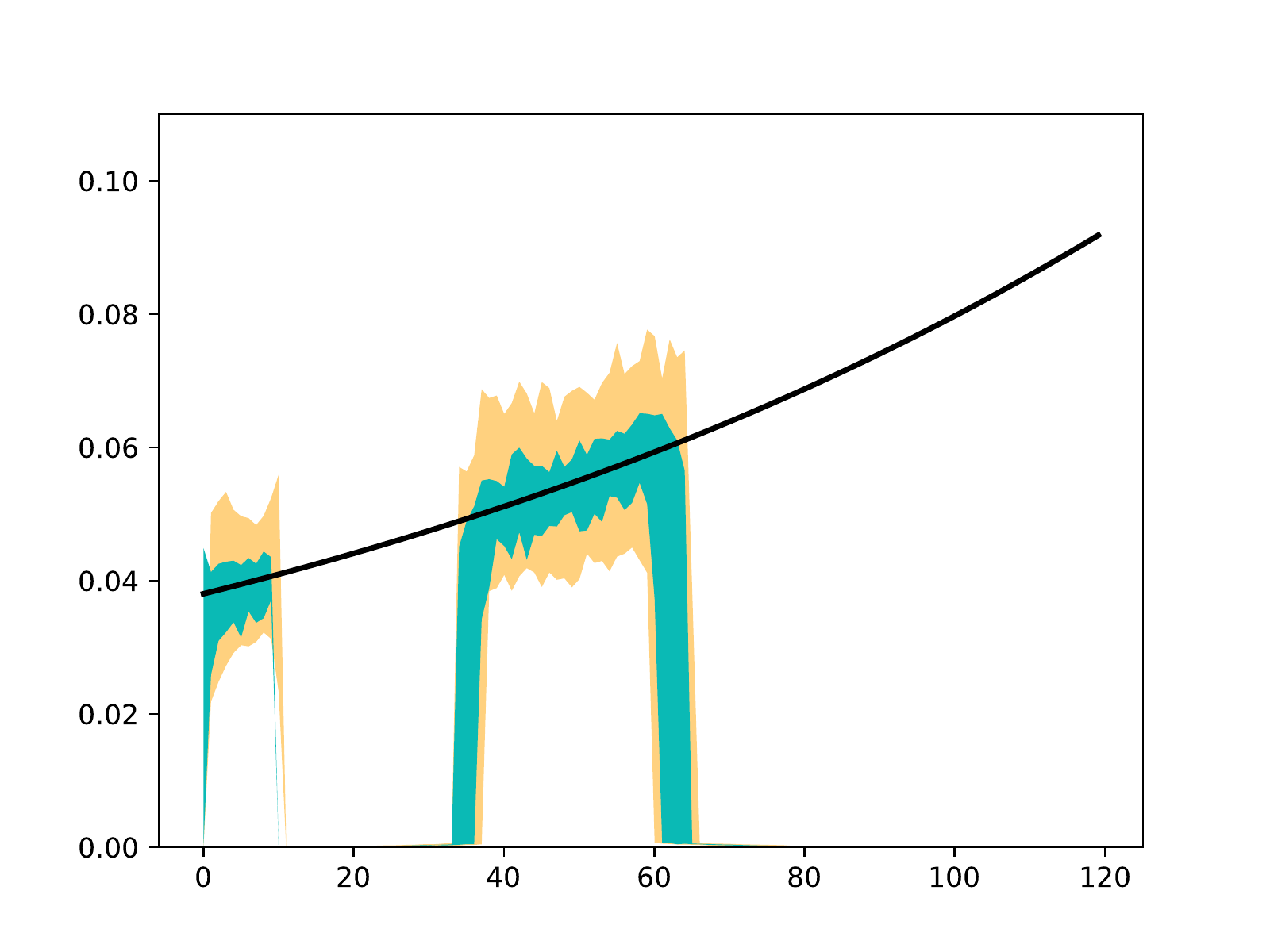}
\includegraphics[scale=0.42]{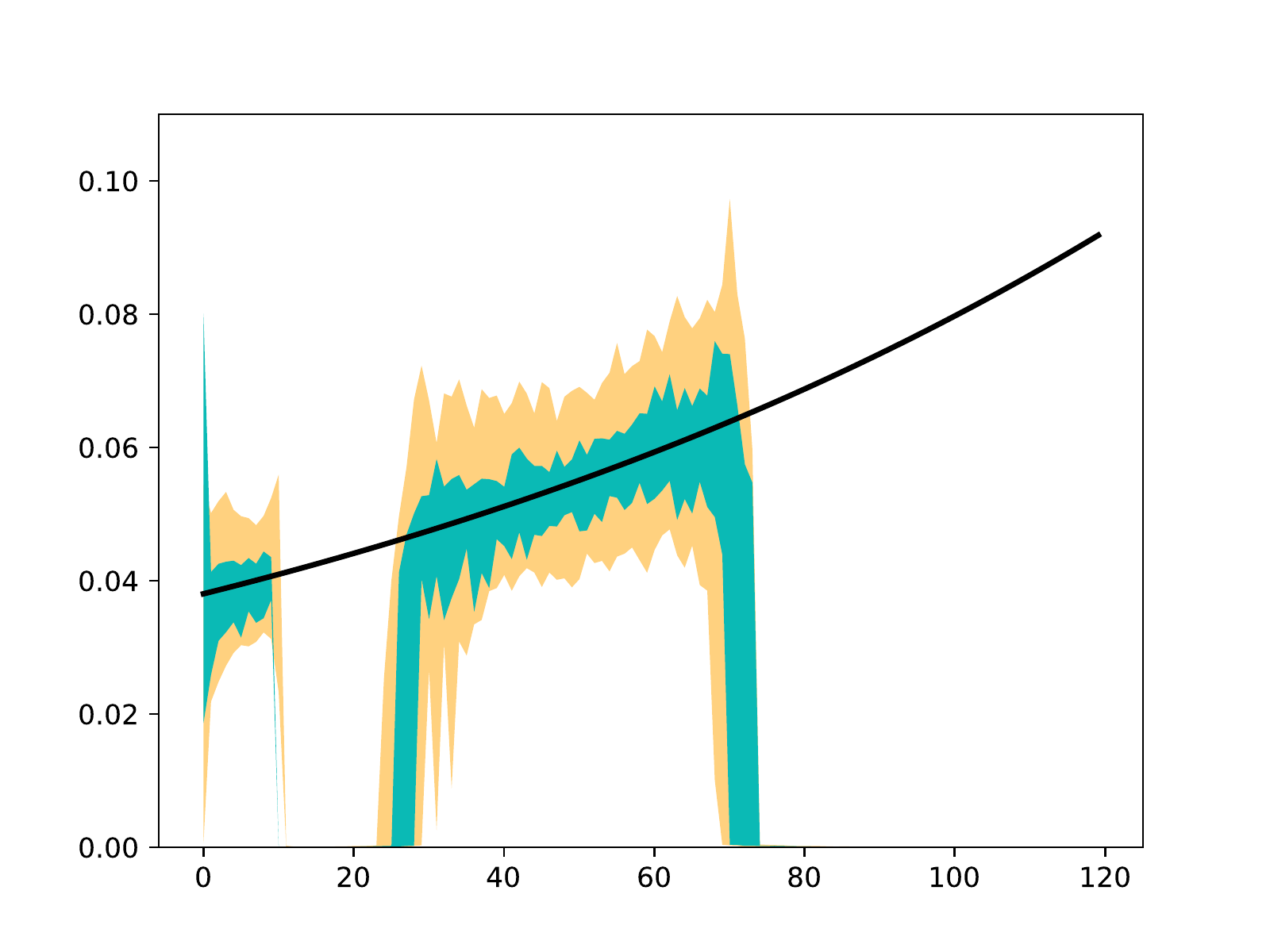}
\caption{\small{For fixed $t = 10$, and $N = 4000$, comparison between the true function (solid black) and a (pointwise) $95\%$ confidence interval based on 50 Monte-Carlo simulations. Oracle estimator in green and our adaptative estimator in yellow. Estimation of $a \mapsto \mu(10,a)$. Left: $\varpi = 10^{-2}$. Right: $\varpi = 5\cdot 10^{-3}$.
 %Estimation of $a \mapsto \pi(10,a) = \mu(10,a)g(10,a)$. 
$X$-axis: units in years, Y-axis: rate per unit of time. In order to improve on these results, one must either lower down $\varpi$ or expect to be in a more favourable regime $N \gg 4000$.}}
\end{figure}

\section{Proof or Theorem \ref{thm concentration optimal}} \label{sec proofs concentration}

This section is devoted to the proof of the concentration properties of the model stated in Theorem \ref{thm concentration optimal}.
%Let 
%$$M_{w_2}^N(f)_t = N^{-1}\int_0^t\int_{\mathbb N \times \R_+}$$
%\subsubsection*{Notation} 
%\subsubsection{Preliminaries}
Recall that $w_1 \in \mathcal L_{\mathcal D}^{{\small\mathrm{\,time}}}$ and $w_2 \in \mathcal L_{\mathcal D}^{{\small\mathrm{\,age}}}$ are two continuous weight functions. We introduce two fundamental processes for which we will establish concentration properties:
% and that
%Under Assumptions \ref{H basic}, \ref{H stability} and \ref{H initial limit}, we define
%$$\mathcal W_{w_1, w_2}^N(\mathcal F)_t = \sup_{f \in \mathcal F} \int_0^t w_1(s)\int_0^\infty w_2(s-a) f_s(a)\big(Z^N_s(da)-g(s,a)da\big)ds,$$
%$$\mathcal W_{w_2}^N(\mathcal F)_t = \sup_{f \in \mathcal F}\int_0^\infty w_2(t-a) f_t(a)\big(Z^N_t(da)-g(t,a)da\big).$$
%where $\mathcal F \subset \mathcal L_{\mathcal D}^\infty $ is such that $f \in \mathcal F$ implies $-f \in \mathcal F$.
% \supseteq \{ f:[0,T]\times \R_+ \rightarrow \R\} that satisfies $\mathcal F = -\mathcal F$. 
%$$\mathcal W_{\delta_t, w_2}(\mathcal F) = \sup_{f \in \mathcal F}\big|\int_0^\infty w_2(t-a) f_t(a)\big(Z^N_s(da)-g(t,a)da\big)\big|.$$
%Set also
$$\mathcal M_{\,w_1,w_2}^N(\mathcal F)_t = \sup_{f \in \mathcal F}\big|\int_0^t w_1(s)M_s^N\big(w_2(s-\cdot)f_s\big)ds\big|,$$
where $M_t^N(f)$ is defined in \eqref{def fausse mg} below
and
$$\mathcal M_{w_2}^N(\mathcal F)_t = \sup_{f \in \mathcal F} \big|M_t^N\big(w_2(t-\cdot)f_t\big)\big|.$$

%\subsubsection*{Dynamics of $Z^N$ over test functions}
%for $t \in [0,T]$.
% For a real-valued function defined on $\mathcal D \subset \R$, we write $|f|_{L^1} = \int_0^T |f(t)|dt$ and

\subsection{A first stability result}

\begin{prop} \label{prop stability}
Work under Assumptions \ref{H basic} and \ref{minimal F}. Then $\mathcal W_{w_1,w_2}^N(\mathcal F)_T$ is bounded above by
%We have
%$$\mathcal W_{w_1,w_2}^N(\mathcal F)_T \leq 
$$|w_1|_1
%_{L^1([0,T])}
\mathcal W_{w_2}^N(\mathcal F)_0+c_0^{-1}\int_0^T|w_1(t)|\big(\mathcal W_{w_2,1}^N(\mathcal F)_t+\mathcal W_{1,w_2}^N(\mathcal F)_t\big)dt+\mathcal M_{\,w_1,w_2}^N(\mathcal F)_T,$$
where $c_0$ is defined in Assumption \ref{minimal F}.
%, as soon as $\mathcal F \subset \{|f|_\infty \leq 1\}$.

 %satisfies {\color{blue} {\tt [def conditions $\mathcal F$ here]}}.
\end{prop} 
\begin{proof}
By \eqref{eq micro}, the action $\langle Z^N_t, f_t\rangle$ of $Z_t^N(da)$ for $f \in \mathcal L_{\mathcal D}^\infty $ can be written as
\begin{align}
\langle Z^N_t, f_t\rangle & = \int_0^\infty f_t(t+a)Z_0^N(da) \nonumber \\
 & +\int_0^t\int_0^\infty \big(b(s,a)f_t(t-s)-\mu(s,a)f_t(a+t-s)\big)Z^N_s(da)ds+M_t^N(f_t),  \label{action}
\end{align}
with
\begin{equation} \label{def fausse mg}
%M_t^N(f_t) = N^{-1}\int_0^t\int_{\mathbb N \times \R_+}{\bf 1}_{\{i \leq \langle Z^N_{s^-},{\bf 1}\rangle\}}\big(f_t(t-s){\bf 1}_{\{b \leq \theta\}}-f_t(a_i(Z^N_{s^-})+t-s){\bf 1}_{\{b \leq \theta \leq b+\mu\}}\big)\widetilde Q(ds,di,d\theta).
M_t^N(f_t) = N^{-1}\int_0^t \int_{\{i \leq n^N_{s^-}\} \times \R_+}\big(f_t(t-s){\bf 1}_{\{\theta \leq b\}}-f_t(a_i(Z^N_{s^-})+t-s){\bf 1}_{\{b \leq \theta \leq b+\mu\}}\big)\widetilde{\mathcal Q}(ds,di,d\theta).
\end{equation}
In the above formula, $n_t^N = N \langle Z_t^N,{\bf 1}\rangle$ is the size of the population at time $t$, the functions $b$ and $\mu$ in the indicators are evaluated at points $(s,a_i(Z^N_{s^-}))$ and
$\widetilde{\mathcal Q}(ds,di,d\theta) = {\mathcal Q}(ds,di,d\vartheta)-ds \big(\sum_{k \geq 1}\delta_k(di)\big) d\vartheta$ is the compensated measure of the Poisson measure
 ${\mathcal Q}$.\\
%a Poisson random measure on $\R_+\times \mathbb N\setminus \{0\} \times \R_+$ with intensity $ds(\sum_{k \geq 1}\delta_k(di))d\vartheta$.\\

Apply now \eqref{action} to the test function $a \mapsto w_2(t-a)f_t(a)$ with $f \in \mathcal F$, substract $g(t,a)da$ in the equation above, noting that $g(t,a)$ solves \eqref{McKendrick},  set   $\eta^N_t(da)=Z^N_t(da)-g(t,a)da$ and obtain
\begin{align*}
& \int_{\R_+} w_2(t-a)f_t(a)\eta^N_t(da)  = \int_{\R_+} w_2(-a)f_t(t+a)\eta^N_0(da) \\
& + \int_0^t\int_{\R_+} \big(w_2(s)f_t(t-s)b(s,a)-w_2(s-a)f_t(a+t-s)\mu(s,a)\big)\eta_s^N(da)ds+M_t^N(w_2(t-\cdot)f_t).
\end{align*}
Multiplying each term by $\omega_1(t)$, integrating from $0$ to $T$ and taking absolute values, we also have
$$
\big|\int_0^Tw_1(t)\int_{\R_+} w_2(t-a)f_t(a)\eta^N_t(da) dt\big| \leq I + II + III+IV,
$$
with
\begin{align*}
I & = \big|\int_0^Tw_1(t)\int_{\R_+} w_2(-a)f_t(t+a)\eta^N_0(da)dt\big|, \\
II & = \big|\int_0^Tw_1(t) \int_0^t\int_{\R_+} w_2(s)f_t(t-s)b(s,a)\eta_s^N(da)ds dt\big|,\\
III & = \big|\int_0^Tw_1(t) \int_0^t\int_{\R_+} w_2(s-a)f_t(a+t-s)\mu(s,a)\eta_s^N(da)ds dt\big|,\\
IV & = \big|\int_0^Tw_1(t)M_t^N(w_2(t-\cdot)f_t)dt\big|.
\end{align*}
By Assumption \ref{minimal F}, we have $f_t(t+a) \in \mathcal F$ therefore $I \leq  \mathcal W_{w_2}^N(\mathcal F)_0$. Using that $c_0f_t(t-s)b(s,a) \in \mathcal F$, we also have
\begin{align*}
\big|\int_0^t\int_{\R_+} w_2(s)f_t(t-s)b(s,a)\eta_s^N(da)ds\big| & \leq c_0^{-1}\sup_{f\in \mathcal F}\big|\int_0^t\int_{\R_+} w_2(s){\bf 1}(s-a)f_t(a)\eta_s^N(da)ds\big| \\
& = c_0^{-1} \mathcal W_{w_2,1}(\mathcal F)_t,
\end{align*}
Therefore  $II \leq c_0^{-1}\int_0^T |w_1(t)|\mathcal W_{w_2,1}(\mathcal F)_t dt$. In the same way,
\begin{align*}
\big|\int_0^t\int_{\R_+} w_2(s-a)f_t(a+t-s)\mu(s,a)\eta_s^N(da)ds\big| & \leq c_0^{-1}\sup_{f \in \mathcal F}\big|\int_0^t\int_{\R_+} {\bf 1}(s)w_2(s-a)f_t(a)\eta_s^N(da)ds\big| \\
& = c_0^{-1} \mathcal W_{1,w_2}(\mathcal F)_t
\end{align*}
and $III \leq c_0^{-1}\int_0^T |w_1(t)| \mathcal W_{1,w_2}(\mathcal F)_t dt$ follows likewise. Finally,
$$|IV| \leq \sup_{f \in \mathcal F} \big|\int_0^Tw_1(t)M_t^N(w_2(t-\cdot)f_t)dt\big| = \mathcal M_{\,w_1,w_2}^N(\mathcal F)_t.$$
Summing up the estimates, we obtain the conclusion noting that 
$$\sup_{f \in \mathcal F}\big|\int_0^Tw_1(t)\int_{\R_+} w_2(t-a)f_t(a)\eta_s^N(da)ds\big| = \mathcal W_{w_1,w_2}^N(\mathcal F)_T$$ since $\mathcal F$ is stable under $f \mapsto -f$ by Assumption \ref{minimal F}.
%%Setting in the equation above and using that $g(t,a)$ solves \eqref{McKendrick}, we also have
%\begin{align*}
%& \big|\int_0^\infty w_2(t-a)f_t(a)\eta_t^N(da)\big|  \leq  \mathcal W_{w_2}^N(\mathcal F)_0 +|f|_{\infty}\int_0^t\big|\int_0^\infty w_2(s)b(s,a)\eta_s^N(da)\big| ds \\
%+ & \int_0^t \big| \int_0^\infty w_2(s-a)f_t(a+t-s)\mu(s,a)\eta_s^N(da)\big|ds +  \big|M_t^N(w_2(t-\cdot)f_t)\big|. 
%%& + \int_0^t\int_0^\infty \big(w_2(s)f_t(t-s)b(s,a)-w_2(s-a)f_t(a+t-s)\mu(a,s)\big)g(s,a)ds.
%\end{align*}
%%Note next that
%%$$\int_0^t \big| \int_0^\inftyw_2(s)b(s,a)\eta_s^N(da)\big|ds \leq |b|_\infty \mathcal W_{w_2, 1}(\mathcal F)_t$$
%Multiplying each term by $w_1(t)$ and integrating from $0$ to $T$, we successively obtain
%%Setting $\eta^N_t(da)=Z^N_t(da)-g(t,a)da$ and substracting the two equations above, we obtain
%$$\int_0^Tw_1(t)\int_0^t w_2(s)\big| \int_0^\infty b(s,a)\eta_s^N(da)\big|dsdt \leq  c_0^{-1} \int_0^T w_1(t) \mathcal W_{w_2,1}(\mathcal F)_t dt $$
%and
%$$\int_0^Tw_1(t)\int_0^t \big| \int_0^\infty w_2(s-a)f_t(a+t-s)\mu(s,a)\eta_s^N(da)\big|dsdt \leq c_0^{-1} \int_0^T w_1(t) \mathcal W_{1,w_2}(\mathcal F)_t dt.$$
%The result follows.
\end{proof}
%Let 
%$$\mathcal M_{w_1,w_2}^N(\mathcal F)_{T}^\star = \max_{h,k=1, w_1, w_2}\mathcal M_{h,k}^N(\mathcal F)_{T}.$$
%\mathcal M_{1,1}^N(\mathcal F)_{T},\mathcal M_{w_1,1}^N(\mathcal F)_{T},\mathcal M_{w_2,1}^N(\mathcal F)_{T},\mathcal M_{1,w_2}^N(\mathcal F)_{T},\mathcal M_{w_1, w_2}^N(\mathcal F)_{T}\big\}$$
%Given a nonnegative function $\omega$, we define $\omega^+(t) = \max\{\omega(t),1\}$.
\begin{prop} \label{global stability}
Work under Assumptions \ref{H basic} and \ref{minimal F} . We have
%\begin{equation} \label{eq gronwall permut}
$$
% \mathcal W_{w_1,w_2}^N(\mathcal F)_T \lesssim   & \; |w_1|_{L^1}(1+|w_1|_{L^1}+|w_2|_{L^1})\big(\mathcal W_{\delta_0,w_2}^N(\mathcal F)_0+ \mathcal W_{\delta_0,1}^N(\mathcal F)_0\big) \\
%&  +(1+|w_1|_{L^1}) \max_{h,k=1, w_1, w_2}\mathcal M_{h,k}^N(\mathcal F)_{T},
%\mathcal W_{w_1,w_2}^N(\mathcal F)_T \lesssim  \big(1+\max_{i = 1,2}|\omega_i|_{L^1([0,T])}^2\big)\big(\max_{k=1,w_2}\mathcal W_{k}^N(\mathcal F)_0
%+\max_{h,k=1, w_1, w_2}\mathcal M_{h,k}^N(\mathcal F)_{T}\big),
\mathcal W_{w_1,w_2}^N(\mathcal F)_T \lesssim |w_1|_1\max_{(k_1,k_2)}|k_1|_{L^1([0,T])}\mathcal W_{k_2}^N(\mathcal F)_0
+\max_{(l_1,\ldots, l_4)}|l_1|_{L^1([0,T])}|l_2|_{L^1([0,T])}\mathcal M_{l_3,l_4}^N(\mathcal F)_{T},
$$
%\end{equation}
where $(k_1,k_2)$ and $(l_1,\ldots, l_4)$ range over permutations of $(1,w_2)$ and $(1,1,w_1,w_2)$ respectively. The symbol $\lesssim $ means inequality up to an explicitly computable constant depending on $T$ and $c_0$ from Assumption \ref{minimal F}. 
%such that  $k_i = 1,w_1,w_2$ and such that $w_1$ and $w_2$ appear once and once only in each configuration.
% , as soon as $\mathcal F \subset \{|f|_\infty \leq 1\}$. 
% that can be explicitly be computed from the computations in the proof.
\end{prop} 
\begin{proof}
Apply first Proposition  \ref{prop stability} with $w_1=1$ and $w_2=1$ to obtain
\begin{align*}
\mathcal W_{1,1}^N(\mathcal F)_T & \leq T \mathcal W_{1}^N(\mathcal F)_0+2c_0^{-1}\int_0^T\mathcal W_{1,1}^N(\mathcal F)_tdt+\mathcal M_{1,1}^N(\mathcal F)_T \\
& \leq \big(T \mathcal W_{1}^N(\mathcal F)_0+\mathcal M_{1,1}^N(\mathcal F)_T\big)e^{2c_0^{-1}T} \\
& = \mathcal G^{(1), N}(\mathcal F)_T,
%(\mathcal W_{\delta_0,1}^N(\mathcal F)_0, \mathcal M_{1,1}^N(\mathcal F)_T)
\end{align*}
say, by Gr\"onwall lemma.  Next, by Proposition  \ref{prop stability} applied to $(w_2, 1)$, we obtain
\begin{align*}
\mathcal W_{w_2,1}^N(\mathcal F)_T & \leq |w_2|_{L^1([0,T])}\mathcal W_{1}^N(\mathcal F)_0+2c_0^{-1}\int_0^T |w_2(t)|\mathcal W_{1,1}^N(\mathcal F)_tdt+\mathcal M_{w_2,1}^N(\mathcal F)_T \\
%& \leq |w_1|_{L^1}\mathcal W_{\delta_0,1}^N(\mathcal F)_0+2C |w_1|_{L^1}\mathcal W_{1,1}^N(\mathcal F)_T+\mathcal M_{w_1,1}^N(\mathcal F)_T \\
& \leq |w_2|_{L^1([0,T])}\big(\mathcal W_{1}^N(\mathcal F)_0+2c_0^{-1}\mathcal G^{(1),N}(\mathcal F)_T\big) +\mathcal M_{w_2,1}^N(\mathcal F)_T \\
& =  \mathcal G^{(2),N}_{w_2}(\mathcal F)_T, 
\end{align*}
say. Apply now Proposition \ref{prop stability} with $(1,w_2)$ so that
\begin{align*}
\mathcal W_{1,w_2}^N(\mathcal F)_T & \leq T\mathcal W_{w_2}^N(\mathcal F)_0+c_0^{-1}\int_0^T\big(\mathcal W_{w_2,1}^N(\mathcal F)_t+\mathcal W_{1,w_2}^N(\mathcal F)_t\big)dt+\mathcal M_{1,w_2}^N(\mathcal F)_T \\
& \leq T\mathcal W_{w_2}^N(\mathcal F)_0+c_0^{-1}T\mathcal G^{(2),N}_{w_2}(\mathcal F)_T+\int_0^T\mathcal W_{1,w_2}^N(\mathcal F)_tdt+\mathcal M_{1,w_2}^N(\mathcal F)_T \\
& \leq \big(T\mathcal W_{w_2}^N(\mathcal F)_0+CT\mathcal G^{(2),N}_{w_2}(\mathcal F)_T+\mathcal M_{1,w_2}^N(\mathcal F)_T\big)e^{c_0^{-1}T} \\
& =\mathcal G^{(3),N}_{w_2}(\mathcal F)_T\end{align*}
say, by the previous estimate and  Gr\"onwall lemma again. By Proposition \ref{prop stability} and the two previous bounds, we infer that $\mathcal W_{w_1,w_2}^N(\mathcal F)_T$ is less than
%\begin{align*}
%\mathcal W_{w_1,w_2}^N(\mathcal F)_T \leq 
\begin{align*}
|w_1|_{L^1([0,T])}\mathcal W_{w_2}^N(\mathcal F)_0+c_0^{-1}|w_1|_{L^1([0,T])}\big(\mathcal G^{(2),N}_{w_2}(\mathcal F)_T+\mathcal G^{(3),N}_{w_2}(\mathcal F)_T\big)+\mathcal M_{\,w_1,w_2}^N(\mathcal F)_T.
\end{align*}
Expanding the estimates $\mathcal G^{(2),N}_{w_2}(\mathcal F)_T$ and $\mathcal G^{(3),N}_{w_2}(\mathcal F)_T$ in terms of their appropriate arguments concludes the proof.
\end{proof}

By Proposition \ref{global stability}, we see that the stability of the system is controlled by the initial approximation $\mathcal W_{w_2}^N(\mathcal F)_0$ (including $w_2=1$) and the propagation of the stochastic term $\mathcal M_{w_1, w_2}^N(\mathcal F)_{T}$. We now turn to that latter term.

\subsection{Stability of the stochastic term}

For $f \in \mathcal L^{\small{\mathrm{\,age}}}_{\mathcal D}$, let 
%\begin{align*}
$$
 \widetilde{\mathcal M}_{\,w_1,w_2}^N(f)_t  = \int_0^t w_1(s)M_s^N\big(w_2(s-\cdot)f\big)ds$$
 and
$$
\widetilde{\mathcal M}_{\,w_2}^N(f)_t  = M_t^N\big(w_2(t-\cdot)f\big).
$$
%\end{align*}
%and
%$$\mathcal M_{\,w_1 = \delta_0,w_2}^N(f)_t =|M_s^N\big(w_2(t-\cdot)f\big)|.$$\\
In particular, since $\mathcal F$ is stable under $f \mapsto -f$, we have 
\begin{equation} \label{eq: controle triv}
\sup_{f \in \mathcal F}\widetilde{\mathcal M}_{\,w_1,w_2}^N(f)_T = \sup_{f \in \mathcal F}\big|\widetilde{\mathcal M}_{\,w_1,w_2}^N(f)_T\big| = \mathcal M_{\,w_2}^N(\mathcal F)_T
\end{equation}
and
$$
%\;\;\text{and}\;\;
\sup_{f \in \mathcal F}\widetilde{\mathcal M}_{\,w_2}^N(f)_T = \sup_{f \in \mathcal F}\big|\widetilde{\mathcal M}_{\,w_2}^N(f)_T\big| = \mathcal M_{\,w_2}^N(\mathcal F)_T.
$$
%\end{equation}
For $\kappa \geq 0$, consider the event
\begin{equation} \label{eq: def A kappa}
\mathcal A^N_\kappa = \big\{\sup_{0 \leq t \leq T}\langle Z_t^N, {\bf 1} \rangle \leq \exp(|b|_\infty T)(1+\kappa)\big\},
\end{equation}
and for $\lambda \geq 0$, set
$$\vartheta_{w_1,w_2}^N(f)_\lambda = 2NT|w_1|_\infty^{-1}\exp(|b|_\infty T)(|b|_\infty+|\mu|_\infty\big)\rho\big(N^{-1}\lambda |w_1 w_2|_\infty|f|_\infty\big),$$
where $\rho(x) = e^x-x-1$.
\begin{prop} \label{prop concentration} Work under Assumptions \ref{H basic}. For large enough $N$, we have
\begin{equation} \label{eq: controle events}
\int_0^\infty \PP\big((\mathcal A^N_\kappa)^c\big)e^\kappa d\kappa \leq \tfrac{1}{2}
\end{equation}
and for $\lambda \geq 0$,
\begin{equation} \label{eq: first chernoff}
\E\big[\exp\big(\lambda \big|\widetilde{\mathcal M}_{\,w_1,w_2}^N(f)_T  -\widetilde{\mathcal M}_{\,w_1,w_2}^N(g)_T \big|\big){\bf 1}_{\mathcal A^N_\kappa}\big] 
\leq  2\exp \big(|w_1|_1(1+\kappa)\vartheta_{w_1,w_2}^N(f-g)_\lambda \big).
\end{equation}
Moreover, \eqref{eq: first chernoff} remains true with $\widetilde{\mathcal M}_{\,w_2}^N(f)_T-\widetilde{\mathcal M}_{\,w_2}^N(g)_T$, replacing formally $w_1$ by $1$ in the right-hand side of the inequality.
%{\color{blue} [needs $\kappa$ in the RHS of $(7)$]}
%\begin{equation} \label{eq: second chernoff}
%\E\big[\exp\big(\lambda \big|\mathcal M_{\,w_1=\delta_0,w_2}^N(f)_T  -\mathcal M_{\,w_1=\delta_0,w_2}^N(g)_T \big|\big){\bf 1}_{\mathcal A^N_\kappa}\big] 
% \leq  2\exp \big(\vartheta_{w_1 = \delta_0,w_2}^N(f-g)_\lambda\big).
%\end{equation}
\end{prop}

\begin{proof}
We first prove \eqref{eq: controle events}, namely
%\begin{equation} \label{eq: controle events}
$$
\int_0^\infty e^\kappa \PP\big(\sup_{0 \leq t \leq T}\langle Z^N_t, {\bf 1} \rangle > \exp(|b|_\infty T)(1+\kappa)\big)d\kappa \leq \tfrac{1}{2}.
$$
%\end{equation}
\noindent {\bf Step 1)} Consider the equation
$$\widetilde Z_t^N =  \tau_t Z_0^N 
+N^{-1}\int_0^t  \int_{\mathbb N \times \mathbb R_+}  \delta_{t-s}(da){\bf 1}_{\{0 \leq \vartheta \leq |b|_\infty, i \leq N\langle \widetilde Z_{s^-}^N, {\bf 1}\rangle\}}\mathcal Q_1(ds, di, d\vartheta)
%&- N^{-1}\int_0^t  \int_{\mathbb N \times \mathbb R_+} \delta_{a_i(Z_{s^-}^N)+t-s}(da){\bf 1}_{\big\{0 \leq \vartheta \leq \mu(s,a_i(Z_{s^-}^N)), i \leq \langle Z_{s^-}^N, {\bf 1}\rangle \big\}}Q_2(ds,di,d\vartheta) \label{eq micro}
$$
defined on the same probability space as $(Z_t(da))_{0 \leq t \leq T}$.
%Clearly, 
%we have $Z_t^N \leq \widetilde Z_t^N$ for every $t$ and 
Applying \eqref{action} with $b=|b|_\infty$, $\mu=0$ and $f_t=1$, we obtain
$$\langle \widetilde Z^N_t, {\bf 1} \rangle  = \langle \widetilde Z^N_0, {\bf 1} \rangle +|b|_\infty \int_0^t \langle \widetilde Z^N_s, {\bf 1} \rangle ds+M_t^N(\bf{1}),$$
and for every $\lambda \geq 0$, by It\^o's formula:
$$\exp\big(\lambda \langle \widetilde Z^N_t, {\bf 1} \rangle \big) = \exp\big(\lambda \langle \widetilde Z^N_0, {\bf 1} \rangle\big)+N|b|_\infty\big(e^{\lambda/N}-1\big)\int_0^t \langle \widetilde Z^N_s, {\bf 1} \rangle\exp\big(\lambda \langle \widetilde Z^N_s, {\bf 1} \rangle \big)ds +\xi_t,$$
 where $(\xi_t)_{0 \leq t \leq T}$ is a local martingale. By localisation, one can prove that $\E[\xi_t]=0$. Writing $f(t,\lambda) = \E[\exp\big(\lambda \langle \widetilde Z^N_t, {\bf 1} \rangle \big)]$, it follows that
\begin{equation} \label{eq: faux transport}
f(t,\lambda) = f(0,\lambda)+N|b|_\infty\big(e^{\lambda/N}-1\big)\int_0^t\partial_\lambda f(s,\lambda)ds.
\end{equation}
The solution of the transport equation \eqref{eq: faux transport} at time $t=T$ with initial condition $ f(0,\lambda) = f_0(\lambda)$ is given by
\begin{align*}
f(T,\lambda) &  = f_0\Big(N \log\frac{e^{\lambda/N-|b|_\infty T}}{1-(1-e^{-|b|_\infty T})e^{\lambda/N}}\Big)  \leq  \exp\Big(qN \log\frac{e^{\lambda/N-|b|_\infty T}}{1-(1-e^{-|b|_\infty T})e^{\lambda/N}}\Big),
%& \leq  \exp\Big(qN \log\frac{e^{\lambda/N-|b|_\infty T}}{1-(1-e^{-|b|_\infty T})e^{\lambda/N}}\Big)
\end{align*}
where the last inequality stems from $f_0(\lambda) = \E[\exp(\lambda \langle\widetilde Z_0^N, {\bf 1}\rangle)] =  \E[\exp(\lambda \langle Z_0^N, {\bf 1} \rangle)] \leq e^{q\lambda }$ for some $q$ by Assumption \ref{H basic} (ii).\\

\noindent {\bf Step 2)} With the notation $r=\exp(-|b|_\infty T)$, the usual Chernoff bound argument yields
\begin{align*}
 \log \PP\big(\langle \widetilde Z^N_T, {\bf 1} \rangle > r^{-1}(1+\kappa)\big) 
% \, \log \PP\big(\mathcal N_t > r^{-1}(1+\kappa)\big) \\
 \leq & \, -\lambda r^{-1}(1+\kappa) + qN \log\tfrac{r e^{\lambda/N}}{1-(1-r)e^{\lambda/N}} \\
 \leq & \,-Nr^{-1}(1+\kappa)  \log \big(\big(1-\tfrac{rq}{\kappa+1}\big)\tfrac{1}{1-r}\big)+
 % Nq \log \big(\big(1-\tfrac{rq}{\kappa+1}\big)\tfrac{r}{1-r}\big)-Nq \log\big(\tfrac{rq}{\kappa+1}\big) \\
 qN \log \tfrac{\kappa+1-rq}{1-r} \\
 \leq & \, \log C_1 -C_2N\kappa
\end{align*}
for the choice $\lambda = N\log \big((1-\frac{rq}{\kappa+1})\frac{1}{1-r}\big)$ and for two constants $C_i = C_i(q,r)>0$ that do not depend on $N$.
Noting that by construction, $\sup_{t \leq T} \langle Z^N_t, {\bf 1} \rangle \leq  \langle \widetilde Z^N_T, {\bf 1} \rangle$, we finally obtain
\begin{align*}
\int_0^\infty e^\kappa \PP\big(\sup_{0 \leq t \leq T}\langle Z^N_t, {\bf 1} \rangle > r^{-1}(1+\kappa)\big)d\kappa 
& \leq \int_0^\infty e^\kappa  \PP\big(\langle \widetilde Z^N_T, {\bf 1} \rangle > r^{-1}(1+\kappa)\big)d\kappa \\
& \leq C_1\int_0^\infty e^{(1-C_2N)\kappa}d\kappa = \frac{C_1}{C_2N-1}\leq \tfrac{1}{2}
\end{align*}
for $N \geq (1+2C_1)/C_2$, and \eqref{eq: controle events} is proved.\\

\noindent {\bf Step 3)} We now turn to \eqref{eq: first chernoff}.  For $t_0 \in [0,T]$ and $f \in \mathcal L_{\mathcal D}^{\small{\mathrm{\,age}}}$,
% {\color{blue} [define notation somewhere, genuinely univariate functions!!]}, 
define
\begin{equation} \label{eq: def B}
B_{t, t_0}^N(f) = N\int_0^{t \wedge t_0} \int_{\R_+} \Big(b(s,a)\rho\big(N^{-1}f(t_0-s)\big)+\mu(s,a)\rho\big(N^{-1}f(a+t_0-s)\big)\Big)Z_s^N(da)ds.
\end{equation}
\begin{lem} \label{lem: martingale exp}
For every $t_0 \in [0,T]$ and $f,g \in \mathcal L_{\mathcal D}^{\small{\mathrm{\,age}}}$, there exists a nonnegative random variable $\Lambda_{t_0,t_0}^N(f-g)$ with $\E[\Lambda_{t_0,t_0}^N(f-g)]=1$ 
%and such that for every locally bounded $f,g$, we have
such that
$$ \E\big[\exp\big(M_{t_0}^N(f)-M_{t_0}^N(g)\big)\big]  
=   \E\big[\Lambda_{t_0, t_0}^N(f-g)\exp B_{t_0, t_0}^N(f-g)\big].$$
%for every $f,g \in \mathcal L_{\mathcal D}^\infty $. 
\end{lem}
\begin{proof}
Fix $t_0 \in [0,T]$ and for $f \in \mathcal L_{\mathcal D}^{\small{\mathrm{\,age}}}$, define the random process
$$
\widetilde M_{t, t_0}^N(f) = N^{-1}\int_0^{t \wedge t_0} \int_{\{i \leq n^N_{s^-}\} \times \R_+}\big(f(t_0-s){\bf 1}_{\{b \leq \theta\}}-f(a_i(Z^N_{s^-})+t_0-s){\bf 1}_{\{b \leq \theta \leq b+\mu\}}\big)\widetilde {\mathcal Q}(ds,di,d\theta),
$$
obtained by keeping $t=t_0$ fixed in the integrand of $M_{t \wedge t_0}^N(f)$ defined  in \eqref{def fausse mg}.
By construction,  $(\widetilde M_{t,t_0}(f))_{0 \leq t \leq T}$ is a martingale. In turn, a simple consequence of It\^o's formula, see {\it e.g.} Tran \cite{TRAN1}
% {\color{blue} [precise ref. needed here + check!]} 
shows that the random process
$$t \mapsto \Lambda_{t, t_0}^N(f) = \exp\big(\widetilde M_{t, t_0}^N(f)-B_{t, t_0}^N(f)\big)$$
is a martingale such that $\E[\Lambda_{t, t_0}^N(f)]=1$. Noting that $M_{t_0}^N(f) = \widetilde M_{t_0, t_0}^N(f)$ at $t=t_0$,
we also have
%\begin{equation} \label{eq: egalite mg}
\begin{align*}
 \E\big[\exp\big(M_{t_0}^N(f)-M_{t_0}^N(g)\big)\big] & =  \E\big[\exp\big(M_{t_0}^N(f-g)\big)\big] \\
 & = \E\big[\exp\big(\widetilde M_{t_0,t_0}^N(f-g)\big)\big] \\
& =   \E\big[\Lambda_{t_0, t_0}^N(f-g)\exp B_{t_0, t_0}^N(f-g)\big].
\end{align*} 
%\end{equation}
%using the linearity of $f \mapsto M_{t_0}^N(f)$.
\end{proof}
Let $\lambda \geq 0$. We substitute $f-g$ by $a \mapsto \lambda w_1(t_0)w_2(t_0-a)(f(a)-g(a))$ and look for an upper bound for
$$B_{t_0, t_0}^N\big(\lambda w_1(t_0)w_2(t_0-\cdot)(f-g)\big).$$\\
% = B_{t_0, t_0}^N\big(\lambda w_1(t_0)w_2(t_0-\cdot)(f_{t_0}-g_{t_0})\big).$$\\
\noindent {\bf Step 4)}
%Conditional on $N\int_0^\infty Z_0^N(da) = k$, the random variable $N\int_0^\infty \widetilde Z_T^N(da)$ is formed of $k$ independent Poisson random variables with parameter $|b|_\infty T$ hence is Poisson distributed with parameter $k|b|_\infty T$. It follows that
%\begin{align*}
%\PP\big(\sup_{0 \leq t \leq T}\int_0^\infty Z_t^N(da) > e^{|b|_\infty T}+\kappa\big) & \leq \sum_{k \geq 1} \PP\big(\mathcal P(k|b|_\infty T) > N(e^{|b|_\infty T}+\kappa)\big)\PP(N\int_0^\infty Z_0^N(da)=k)
%\end{align*}
Observe first that  $\rho(x) = e^x-x-1$ implies
that for any nonnegative function $\psi \in \mathcal L_{\mathcal D}^{\small{\mathrm{\,age}}}$, we have 
\begin{align*}
\rho\big(N^{-1}\lambda \psi(a')(f(a)-g(a))\big) & \leq N^{-1}\psi(a')|f-g|_\infty \int_0^\lambda \big(\exp(\kappa N^{-1}|\psi|_\infty|f-g|_\infty)-1\big)d\kappa \\
& = \frac{\psi(a')}{|\psi|_\infty}\rho\big(N^{-1}\lambda |\psi|_\infty|f-g|_\infty\big).
\end{align*}
Therefore, with $\psi(a')=w_1(t_0)w_2(t_0-a')$ and $a'=t_0-s$, we derive
$$\rho\big(N^{-1}\lambda w_1(t_0)w_2(s)(f(t_0-s)-g(t_0-s))\big) \leq \frac{w_1(t_0)w_2(s)}{|w_1w_2|_\infty}\rho\big(N^{-1}\lambda |w_1w_2|_\infty|f-g|_\infty\big)$$
and 
$$\rho\big(N^{-1}\lambda w_1(t_0)w_2(s-a)(f(a+t_0-s)-g(a+t_0-s))\big) \leq  \frac{w_1(t_0)w_2(s-a)}{|w_1w_2|_\infty}\rho\big(N^{-1}\lambda |w_1w_2|_\infty|f-g|_\infty\big)$$ 
with $a'=a+t_0-s$ follows likewise. Plugging these two estimates in the definition \eqref{eq: def B} of  $B_{t_0, t_0}^N$, we infer on $\mathcal A^N_\kappa = \{\sup_{0 \leq t \leq T}\langle Z^N_t, {\bf 1} \rangle \leq  \exp(|b|_\infty T)(1+\kappa)\}$ the chain of inequalities
\begin{align*}
&  B_{t_0, t_0}^N\big(\lambda w_1(t_0)w_2(t_0-\cdot)(f-g)\big) \\
\leq & \,N(|b|_\infty+|\mu|_\infty)\frac{w_1(t_0)}{|w_1w_2|_\infty} \rho\big(N^{-1}\lambda |w_1w_2|_\infty|f-g|_\infty\big) \int_0^{t_0} \int_{\R_+} \big(w_2(s)+w_2(s-a)\big)Z_s^N(da)ds \\
\leq & \,N(|b|_\infty+|\mu|_\infty)\frac{w_1(t_0)}{|w_1|_\infty} \rho\big(N^{-1}\lambda |w_1w_2|_\infty|f-g|_\infty\big) 2T\sup_{0 \leq t \leq T}\langle Z_t^N,{\bf 1}\rangle \\
\leq & \,N(|b|_\infty+|\mu|_\infty)\frac{w_1(t_0)}{|w_1|_\infty} \rho\big(N^{-1}\lambda |w_1w_2|_\infty|f-g|_\infty\big)\exp(|b|_\infty T)(1+\kappa)2T \\
= & \, w_1(t_0)(1+\kappa)\vartheta_{w_1,w_2}^N(f-g)_\lambda.
\end{align*}
%say. %Writing $\E_{\mathcal A^N_\kappa}[\cdot] = \PP(\mathcal A^N_\kappa)^{-1}\E[\cdot {\bf 1}_{\mathcal A^N_\kappa}]$ for the expectation conditional on $\mathcal A^N_\kappa$, 
We derive 
%, and finally
\begin{align}  \label{eq: controle B}
& \exp\big(\lambda w_1(t_0)M_{t_0}^N(w_2(t_0-\cdot)(f-g))\big){\bf 1}_{\mathcal A^N_\kappa}  \\
\leq & \exp\big(w_1(t_0)(1+\kappa)\vartheta_{w_1,w_2}^N(f-g)_\lambda\big)\Lambda_{t_0,t_0}^N(\lambda w_1(t_0)w_2(t_0-\cdot)(f-g)\big) {\bf 1}_{\mathcal A^N_\kappa}. \nonumber
\end{align}
%on $\mathcal A^N_\kappa$.\\

\noindent {\bf Step 5)} For every integer $n \geq 1$ and $\lambda \geq 0$, $f \in \mathcal L_{\mathcal D}^{\small{\mathrm{\,age}}}$, define
$$\Delta^{N,n}_{w_1,w_2}(f-g)_\lambda = \exp\big(\lambda Tn^{-1}\sum_{i = 1}^n w_1(iTn^{-1})M_{iTn^{-1}}^N(w_2(iTn^{-1}-\cdot)(f-g))\big).$$
Applying  repeatedly \eqref{eq: controle B} with $t_0=iTn^{-1}$  and integrating  with respect to  $\E_{\mathcal A^N_\kappa}[\cdot] = \PP(\mathcal A^N_\kappa)^{-1}\E[\cdot {\bf 1}_{\mathcal A^N_\kappa}]$, we obtain
\begin{align*}
 \E_{\mathcal A^N_\kappa}\big[\Delta^{N,n}_{w_1,w_2}(f-g)_\lambda\big] 
\leq &\; \exp\big(Tn^{-1}\sum_{i = 1}^nw_1(iTn^{-1})(1+\kappa)\vartheta_{w_1,w_2}^N(f-g)_\lambda\big) \times \\
& \times \E_{\mathcal A^N_\kappa}\Big[\prod_{i = 1}^n\Lambda_{iTn^{-1},iTn^{-1}}^N(\lambda Tw_1(t_0)w_2(t_0-\cdot)(f-g)\big)^{1/n}\Big] \\
\leq &\; \exp\big(Tn^{-1}\sum_{i = 1}^nw_1(iTn^{-1})(1+\kappa)\vartheta_{w_1,w_2}^N(f-g)_\lambda\big) \PP\big(\mathcal A^N_\kappa\big)^{-1}, 
\end{align*}
where we used the fact that the geometric mean is controlled by the arithmetic mean:
$$\prod_{i = 1}^n\Lambda_{iTn^{-1},iTn^{-1}}^N(\lambda w_1(t_0)w_2(t_0-\cdot)(f-g)\big)^{1/n}\leq n^{-1}\sum_{i = 1}^{n}\Lambda_{iTn^{-1},iTn^{-1}}^N(\lambda w_1(t_0)w_2(t_0-\cdot)(f-g)\big)$$
and the fact that 
\begin{align*}
& \E_{\mathcal A^N_\kappa}\Big[\Lambda_{iTn^{-1},iTn^{-1}}^N(\lambda w_1(t_0)w_2(t_0-\cdot)(f-g)\big)\Big] \\
 \leq & \; \PP\big(\mathcal A^N_\kappa\big)^{-1} \E\Big[\Lambda_{iTn^{-1},iTn^{-1}}^N(\lambda w_1(t_0)w_2(t_0-\cdot)(f-g)\big)\Big] = \PP\big(\mathcal A^N_\kappa\big)^{-1}
\end{align*}
since $\Lambda_{iTn^{-1},iTn^{-1}}^N(\lambda w_1(t_0)w_2(t_0-\cdot)(f-g)\big)$ has expectation $1$ by Lemma \ref{lem: martingale exp}. Using
$$\liminf_{n \rightarrow \infty}\Delta^{N,n}_{w_1,w_2}(f-g)_\lambda = \exp\big(\lambda \int_0^Tw_1(s)M_{s}^N(w_2(s-\cdot)(f-g))ds\big)$$
by convergence of Riemann sums,
% {\color{blue} [check argument Riemann integrability]}
letting $n \rightarrow \infty$, we obtain by Fatou lemma
\begin{equation} \label{eq: controle pre-final}
\E_{\mathcal A^N_\kappa}\Big[\exp\big(\lambda \int_0^Tw_1(s)M_{s}^N(w_2(s-\cdot)(f-g))ds\big)\Big] \leq \PP\big(\mathcal A^N_\kappa\big)^{-1}\exp\big(|w_1|_1(1+\kappa)\vartheta_{w_1,w_2}^N(f-g)_\lambda\big).
\end{equation}
Noting that Lemma \ref{lem: martingale exp} also holds for $-M_{s}^N(f)$ and applying \eqref{eq: controle pre-final} to $-M_{s}^N(w_2(s-\cdot)(f-g))$, we infer
%\begin{equation} \label{eq: last trick}
$$
\E_{\mathcal A^N_\kappa}\Big[\exp\big(\lambda \big|\int_0^Tw_1(s)M_{s}^N\big(w_2(s-\cdot)(f-g)\big)ds\big|\big)\Big] 
 \leq  2\PP\big(\mathcal A^N_\kappa\big)^{-1}\exp\big(|w_1|_1(1+\kappa)\vartheta_{w_1,w_2}^N(f-g)_\lambda\big),
%\end{equation}
$$
but since $\int_0^Tw_1(s)M_{s}^N\big(w_2(s-\cdot)(f-g)\big)ds = \widetilde{\mathcal M}_{w_1,w_2}^N(f)_T-\widetilde{\mathcal M}_{w_1,w_2}^N(g)_T$, the estimate \eqref{eq: first chernoff} is established.\\
%$
%{\color{blue} [something unclear here!!]}
%By triangle inequality, we have
%$$\E_{\mathcal A^N_\kappa}\Big[\exp\big(\lambda \big|\mathcal M_{w_1,w_2}^N(f)_T-\mathcal M_{w_1,w_2}^N(g)_T\big|\big)\Big] \leq \E_{\mathcal A^N_\kappa}\Big[\exp\big(\lambda \int_0^Tw_1(s)\big|M_{s}^N(w_2(s-\cdot)(f-g))\big|ds\big)\Big] $$
%and thus \eqref{eq: last trick} proves \eqref{eq: first chernoff}.\\

\noindent {\bf Step 6)} It remains to prove \eqref{eq: first chernoff} for $\mathcal M_{w_2}^N(f-g)_T$. We first integrate \eqref{eq: controle B} for $w_1=1$ at $t_0=T$ so that $|w_1|_1=T$ 
%with $M_{T}^N$ and $-M_{T}$ t
and proceed exactly as in Step 5) to obtain
$$
\E_{\mathcal A^N_\kappa}\big[\exp\big(\lambda M_{T}^N(w_2(T-\cdot)(f-g))\big)\big] 
\leq  \PP(\mathcal A^N_\kappa)^{-1}\exp\big(T(1+\kappa)\vartheta_{1,w_2}^N(f-g)_\lambda\big).$$
Applying the same argument for $- M_{T}^N$, we also have
$$
\E_{\mathcal A^N_\kappa}\big[\exp\big(\lambda \big| M_{T}^N(w_2(T-\cdot)(f-g))\big|\big)\big] 
\leq  2\PP(\mathcal A^N_\kappa)^{-1}\exp\big(T(1+\kappa)\vartheta_{1,w_2}^N(f-g)_\lambda\big)$$
which is the desired result.

\end{proof}

Proposition \ref{prop concentration} is the main ingredient to obtain a concentration inequality for the processes $(\widetilde{\mathcal M}_{\,w_2}^N(f)_T)_{f \in \mathcal F}$ and $(\widetilde{\mathcal M}_{\,w_1,w_2}^N(f)_T)_{f \in \mathcal F}$, and in turn, a deviation bound for $\mathcal M_{w_1,w_2}(\mathcal F)_T$ and $\mathcal M_{w_2}(\mathcal F)_T$ thanks to \eqref{eq: controle triv}. The proof is given in Section \ref{sec remaining proofs} below.\\

More precisely, consider the apparently more general situation where we have a real-valued process $\xi(f)_{f \in \mathcal F}$ indexed by some metric set $(\mathcal F, d)$ and a family of events $\mathcal A(\kappa)_{\kappa>0}$ satisfying the following properties: 
\begin{equation} \label{controle 1}
\int_0^\infty \PP\big(\mathcal A(\kappa)^c\big)e^\kappa d\kappa \leq \tfrac{1}{2},
\end{equation}
and
\begin{equation}  \label{controle 2}
\E\big[\exp\big(\lambda|\xi(f)-\xi(g)|\big){\bf 1}_{\mathcal A(\kappa)}\big] \leq 2\exp\big(c_1(1+\kappa)\rho(c_2 d(f,g)\lambda)\big), 
\end{equation}
%{\color{blue} [deliberately skipped $\alpha_4$ in $\rho$ here; to be checked later]}
for every $\lambda \geq 0$ and some $c_1, c_2 >0$.
% $i=1,\ldots, 3$ such that {\color{blue} $c_1 \geq *$, $c_2 \leq *$ and $c_3 \geq *$ [needs to be fixed properly when moving to Proposition 8]}.  
\begin{prop} \label{prop concentration abstract}
Assume that $\xi(f)_{f \in \mathcal F}$ and $\mathcal A(\kappa)_{\kappa>0}$ satisfy \eqref{controle 1} and \eqref{controle 2} and that $\xi(f_0)=0$ for some $f_0 \in \mathcal F$. Then there exists a choice $\varpi = \varpi(c_1,c_2)>0$ such that for every $u \geq 0$:
$$\PP\Big(\sup_{f \in \mathcal F}|\xi(f)| \geq 8\big(u+\int_0^{\mathrm{diam}_{\widetilde d}(\mathcal F)}\log\big(1+\mathcal N(\mathcal F, \widetilde d, \epsilon)\big)d\epsilon\big)\Big) \leq \big(e^{u/\mathrm{diam}_{\widetilde d}(\mathcal F)}-1\big)^{-1},$$
where $\widetilde d = \varpi \,d$, $\mathrm{diam}_{\widetilde d}(\mathcal F) = \sup_{f,g\in \mathcal F}\widetilde d(f,g)$ and $\mathcal N(\mathcal F, \widetilde d, \epsilon)$ is the
% entropy of $\mathcal F$ for the metric $\widetilde d$, that is the 
minimal number of balls of $\widetilde d$-size $\epsilon >0$ that are necessary to cover $\mathcal F$.
\end{prop}

\begin{rk} \label{rk: avant le choix de varpi}
We show in Remark \ref{rk: le choix de varpi} at the end of the proof of Proposition \ref{prop concentration abstract} that if $c_1 \geq 308$, we may pick
$\varpi(c_1,c_2) = k \sqrt{c_1}c_2$, with $k=2\sqrt{77}$.
\end{rk}

The proof of Proposition \ref{prop concentration abstract} relies on standard concentration techniques and goes back to Dudley \cite{DUDLEY}. We use the classical textbook of Ledoux-Talagrand \cite{LEDOUXTALAGRAND} and detail the computations in the Appendix section \ref{sec appendix}. Combining Proposition \ref{prop concentration} and \ref{prop concentration abstract}, we obtain the following

\begin{thm} \label{prop concentration final} Work under Assumptions \ref{H basic}.
% and \ref{minimal F}.
Assume $\mathrm{diam}_{|\cdot|_\infty}\mathcal F \leq 1$ and 
$$\mathrm{e}(\mathcal F) = \int_0^1 \log\big(1+\mathcal N(\mathcal F, |\cdot|_\infty, \epsilon)\big)d\epsilon <\infty.$$ For large enough $N$, there exists an explicit choice of $C = C(\mathrm e(\mathcal F),T,|b|_\infty, |\mu|_\infty) >0$, given in the proof below, such that for every $u \geq 0$:
\begin{equation} \label{eq: first dev finale}
\PP\big(\mathcal M_{\,w_1,w_2}^N(\mathcal F)_T \geq (1+u)CN^{-1/2}|w_1|_{1,\infty}|w_2|_\infty\big) \leq (e^{u \mathrm{e}(\mathcal F)}-1)^{-1}
%\frac{1}{e^{u e(\mathcal F) }-1}
\end{equation}
and
\begin{equation} \label{eq: sec dev finale}
\PP\big(\mathcal M_{\,w_2}^N(\mathcal F)_T \geq (1+u)CN^{-1/2}|w_2|_\infty\big) \leq(e^{u \mathrm{e}(\mathcal F)}-1)^{-1}.
\end{equation}
\end{thm}
\begin{proof}
%[Proof of Theorem \ref{prop concentration final}]
We plan to apply Proposition \ref{prop concentration abstract} 
%in the light or Proposition \ref{prop concentration}.
with $\xi(f) = \widetilde{\mathcal M}_{\,w_1,w_2}^N(f)_T$, having $\xi(f_0)=0$ for $f_0=0$. We take $\mathcal A(\kappa) = \mathcal A_\kappa^N$ defined in \eqref{eq: def A kappa} and notice that \eqref{controle 1} is satisfied by \eqref{eq: controle events}. Also, we have \eqref{controle 2} by \eqref{eq: first chernoff} with
\begin{equation} \label{eq: def c_1}
c_1 = 2NT|w_1|_1|w_1|_\infty^{-1}\exp(|b|_\infty T)(|b|_\infty+|\mu|_\infty\big)\;\;\text{and}\;\;
%\end{equation}
%and 
%\begin{equation} \label{eq: def c_2}
c_2 = |w_1w_2|_\infty N^{-1},
\end{equation}
for the metric $d(f,g) = |f-g|_\infty$. Setting $\widetilde d = \varpi \, d$ with $\varpi$ taken from Proposition  \ref{prop concentration abstract}, we have
$\mathrm{diam}_{\widetilde d}(\mathcal F) = \varpi \,\mathrm{diam}_d(\mathcal F) \leq \varpi$ by assumption and also $\mathcal N(\mathcal F, \widetilde d, \varepsilon) \leq \mathcal N(\mathcal F, d, \varepsilon/\varpi)$. It follows that 
\begin{align*}
\int_0^{\mathrm{diam}_{\widetilde d}(\mathcal F)}\log\big(1+\mathcal N(\mathcal F, \widetilde d, \epsilon)\big)d\epsilon & \leq \int_0^\varpi \log\big(1+\mathcal N(\mathcal F, d, \epsilon/\varpi)\big)d\epsilon = \varpi \,\mathrm{e}(\mathcal F),
%& = \varpi  \int_0^1 \log\big(1+\mathcal N(\mathcal F, d, \epsilon)\big)d\epsilon
%& = \varpi \,\mathcal E(\mathcal F),
\end{align*}
which is finite by assumption. Since $\mathcal M_{\,w_1,w_2}^N(\mathcal F)_T =  \sup_{f \in \mathcal F}|\xi(f)| = \sup_{f \in \mathcal F}\xi(f)$, remember \eqref{eq: controle triv}, we may apply Proposition \ref{prop concentration abstract} and obtain, for every $u \geq 0$,
$$\PP\big(\mathcal M_{\,w_1,w_2}^N(\mathcal F)_T \geq 8(\varpi \, \mathrm{e}(\mathcal F)+u)\big) \leq (e^{u/\varpi}-1)^{-1},$$
or equivalently
%\begin{equation} \label{concentration quasi finale}
$$
\PP\big(\mathcal M_{\,w_1,w_2}^N(\mathcal F)_T \geq 8\varpi \,\mathrm{e}(\mathcal F)(1+u)\big) \leq (e^{u \mathrm{e}(\mathcal F)}-1)^{-1}.
$$
%\end{equation}
By Remark \ref{rk: avant le choix de varpi} (see also Remark \ref{rk: le choix de varpi}), we pick $\varpi = \varpi(c_1,c_2) = k\sqrt{c_1}c_2$ with $k = 2 \sqrt{77}$, assuming $c_1 \geq 308$ which is satisfied for sufficiently large $N$ by \eqref{eq: def c_1}. Using \eqref{eq: def c_1} again,
% and \eqref{eq: def c_2},
 it follows that
\begin{align*}
8 k\sqrt{c_1}c_2\mathrm{e}(\mathcal F)(1+u) 
& =(1+u) 8k\, \mathrm{e}(\mathcal F)e^{\tfrac{1}{2}|b|_\infty T}\sqrt{2T}(|b|_\infty+|\mu|_\infty)^{1/2}N^{-1/2}(|w_1|_1|w_1|_\infty)^{1/2}|w_2|_\infty \\
& = (1+u) CN^{-1/2}|w_1|_{1,\infty}|w_2|_\infty,
\end{align*}
say, 
with
\begin{equation} \label{eq: constante cruciale}
C = C(\mathrm{e}(\mathcal F),T,|b|_\infty, |\mu|_\infty) = 8k\, \mathrm{e}(\mathcal F)e^{\tfrac{1}{2}|b|_\infty T}\sqrt{2T}(|b|_\infty+|\mu|_\infty)^{1/2}
\end{equation}
and \eqref{eq: first dev finale} follows. The proof of  \eqref{eq: sec dev finale} is obtained in the same way and is omitted. 
\end{proof}

\begin{rk} \label{remarque first concentration marting}

(i) Up to inflating the constant $C$ by a multiplicative factor $\max(1, \mathrm{e}(\mathcal F))$, we see that Theorem \ref{prop concentration final} implies a mild concentration property for 
$$(|w_1|_{1,\infty}|w_2|_\infty)^{-1}\mathcal M_{\,w_1,w_2}^N(\mathcal F)_T\;\;\text{and}\;\;|w_2|_\infty^{-1}\mathcal M_{\,w_2}^N(\mathcal F)_T$$ with rate $CN^{-1/2}$. 

(ii) 
%If the order of magnitude of the error propagation is satisfactory in terms of $N$, 
The initial bound $|w_2|_{1,\infty}$ of Assumption \ref{concentration initiale} inflates to $|w_2|_\infty$ in 
%\eqref{error propagation}. 
\eqref{eq: sec dev finale}. This defect actually has dramatic consequences when applied to subsequent statistical estimation: $w_2$ becomes a kernel depending on $N$ that mimicks a Dirac mass which is not stable for the $|\cdot|_\infty$ metric. Improving on this estimates is actually the key difficulty in the proof of Theorem \ref{thm concentration optimal}.
\end{rk}

\subsection{Proof of Theorem \ref{thm concentration optimal}}

The weakness of Theorem \ref{prop concentration final} lies in the use of Proposition \ref{prop concentration}, where the control \eqref{eq: first chernoff} somehow needs to be improved. This improvement however uses the results of Theorem \ref{prop concentration final} that we are going to iterate.\\

\noindent {\bf Step 1)} By Proposition \ref{global stability}, we have
$\mathcal W_{w_1,w_2}^N(\mathcal F)_T \lesssim I+II$, with
$$I=|w_1|_1\max_{(k_1,k_2)}|k_1|_{L^1([0,T])}\mathcal W_{k_2}^N(\mathcal F)_0$$
and
$$II = \max_{(l_1,\ldots, l_4)}|l_1|_{L^1([0,T])}|l_2|_{L^1([0,T])}\mathcal M_{l_3,l_4}^N(\mathcal F)_{T}.$$
Since $|w_1|_1\max_{(k_1,k_2)}|k_1|_{L^1([0,T])}\mathcal W_{k_2}^N(\mathcal F)_0 \lesssim \max_{k=1,w_2}\mathcal W_{k}^N(\mathcal F)_0$ up to a constant that only depends on $T$, $|w_1|$ and $|w_2|_{L^1([0,T])}$, we have by Assumption \ref{concentration initiale} that 
$(|w_1|_{1,\infty}|w_2|_\infty)^{-1}I$
has a mild concentration property (actually, we can even replace $|w_2|_\infty$ by $|w_2|_{1,\infty}$). Next, by Theorem \ref{prop concentration final}, the mild concentration property also holds for
$$(|w_1|_{1,\infty}|w_2|_{\infty})^{-1}\mathcal W_{k,l}^N(\mathcal F)_T,\;\;\text{with}\;\;(k,l) \in \{(1,1), (w_2,1), (1,w_2)\}$$
up to an appropriate change in the constants, and therefore it carries over to
$(|w_1|_{1,\infty}|w_2|_{\infty})^{-1} II$ since $\max_{(l_1,\ldots, l_4)}|l_1|_{L^1([0,T])}|l_2|_{L^1([0,T])}\mathcal M_{l_3,l_4}^N(\mathcal F)_{T} \lesssim \sum_{(k,l)}\mathcal W_{k,l}^N(\mathcal F)_T$ where the summation holds over $\{(1,1), (w_2,1), (1,w_2)\}$. In turn,
$$(|w_1|_{1,\infty}|w_2|_{\infty})^{-1}\mathcal W_{\,w_1,w_2}^N(\mathcal F)_T$$
has a mild concentration property of order $C'\max(r_N,N^{-1/2})$, for some $C'>0$ that depends on $c_0$ of Assumption \ref{minimal F}, $T$, $|w_1|_1$, $|w_2|_{L^1([0,T])}$ and the constant $C(\mathrm{e}(\mathcal F),T,|b|_\infty, |\mu|_\infty)$ of Theorem \ref{prop concentration final} defined in \eqref{eq: constante cruciale}.\\

\noindent {\bf Step 2)} We next carefully revisit Step 4) of the proof of Proposition \ref{prop concentration}. We have

\begin{align*}
&  B_{t_0, t_0}^N\big(\lambda w_1(t_0)w_2(t_0-\cdot)(f-g)\big) \\
\leq & \,N(|b|_\infty+|\mu|_\infty)\frac{w_1(t_0)}{|w_1w_2|_\infty} \rho\big(N^{-1}\lambda |w_1w_2|_\infty|f-g|_\infty\big) \int_0^{T} \int_{\R_+} \big(w_2(s)+w_2(s-a)\big)Z_s^N(da)ds. 
%\leq & \,N(|b|_\infty+|\mu|_\infty)\frac{w_1(t_0)}{|w_1|_\infty} \rho\big(N^{-1}\lambda |w_1w_2|_\infty|f-g|_\infty\big) 2T\sup_{0 \leq t \leq T}\langle Z_t^N(da),{\bf 1}\rangle \\
%\leq & \,N(|b|_\infty+|\mu|_\infty)\frac{w_1(t_0)}{|w_1|_\infty} \rho\big(N^{-1}\lambda |w_1w_2|_\infty|f-g|_\infty\big)\exp(|b|_\infty T)(1+\kappa)2T \\
%= & \, w_1(t_0)(1+\kappa)\vartheta_{w_1,w_2}^N(f-g)_\lambda,
\end{align*}
Adding and substracting the limit $g(t,a)da$, we also have
\begin{align*}
& \int_0^{T} \int_{\R_+} \big(w_2(s)+w_2(s-a)\big)Z_s^N(da)ds \\
\leq &\,  \int_0^{T} \int_0^\infty \big(w_2(s)+w_2(s-a)\big)g(s,a)dads + \mathcal W_{w_2,1}^N(\mathcal F)_{T}+\mathcal W_{1,w_2}^N(\mathcal F)_{T} \\
\leq &\, |w_2|_g 
+ \mathcal W_{w_2,1}^N(\mathcal F)_{T}+\mathcal W_{1,w_2}^N(\mathcal F)_{T},
\end{align*}
where, for $f \in \mathcal F^{\,\mathrm{age}}_b$ we set
\begin{equation} \label{def strange norm}
|f|_g = |f|_{L^1([0,T])}\sup_{0 \leq t \leq T}\int_0^\infty g(t,a)da+T|f|_1|g|_\infty\wedge |f|_\infty|g|_1.
\end{equation}
This bound is tighter than the  estimate $2|w_2|_\infty T\sup_{0 \leq t \leq T}\langle Z_t^N,{\bf 1}\rangle$ that we used in Step 4) of the proof of Proposition \ref{prop concentration}.
Introduce now the family of events
$$
\mathcal B^N_\kappa = \Big\{\mathcal W_{w_2,1}^N(\mathcal F)_{T} \leq 5C'\max(r_N,N^{-1/2})|w_2|_{1,\infty}(1+\kappa)\Big\},\;\;\kappa >0,
$$
and
$$\mathcal C^N_\kappa = \Big\{\mathcal W_{1,w_2}^N(\mathcal F)_{T} \leq 5C'\max(r_N,N^{-1/2})\sqrt{T}|w_2|_{\infty}(1+\kappa)\Big\},\;\;\kappa >0,$$
where $C'$ is the constant of Step 1). On  $\mathcal B^N_\kappa \cap \mathcal C^N_\kappa$, we now have
\begin{align*} 
%\label{new bound B}
& B_{t_0, t_0}^N\big(\lambda w_1(t_0)w_2(t_0-\cdot)(f-g)\big) \\
\leq & \, N(|b|_\infty+|\mu|_\infty)\frac{w_1(t_0)}{|w_1w_2|_\infty} \rho\big(N^{-1}\lambda |w_1w_2|_\infty|f-g|_\infty\big)\big(|w_2|_g+C_{w_2}^N\big)(1+\kappa) \\
= & \, w_1(t_0)(1+\kappa)\widetilde \vartheta_{w_1,w_2}^N(f-g)_\lambda,
\end{align*}
say, with
\begin{equation} \label{def grosse constante}
C_{w_2}^N = 5C'\max(r_N,N^{-1/2})\big(|w_2|_{1,\infty}+|w_2|_{\infty}\sqrt{T}\big)
\end{equation}
and
\begin{align*}
& \widetilde \vartheta_{w_1,w_2}^N(f-g)_\lambda 
 =   N(|b|_\infty+|\mu|_\infty)\frac{|w_2|_g+C_{w_2}^N}{|w_1w_2|_\infty} \rho\big(N^{-1}\lambda |w_1w_2|_\infty|f-g|_\infty\big).
 \end{align*}

We thus have established that \eqref{eq: first chernoff} of Proposition \ref{prop concentration} holds with $\widetilde \vartheta_{w_1,w_2}^N(f-g)_\lambda$ instead of $\vartheta_{w_1,w_2}^N(f-g)_\lambda$ and 
$\mathcal B_\kappa^N \cap \mathcal C_\kappa^N$ instead of  $\mathcal A_\kappa^N$.\\

\noindent {\bf Step 3)} We now prove an analogous bound as
\eqref{eq: controle events} replacing $\mathcal A_\kappa^N$ by $\mathcal B_\kappa^N \cap \mathcal C_\kappa^N$.
% and $\vartheta_{w_1,w_2}^N(f-g)_\lambda$ by $\widetilde \vartheta_{w_1,w_2}^N(f-g)_\lambda$ respectively.
Applying Theorem \ref{prop concentration final} with $(w_1,w_2) = (w_2,1)$ up to an inflation of $C$ by $\max(\mathrm{e}(\mathcal F), 1)$ with the substitution $1+u = 5(1+\kappa)$, we obtain
\begin{align*}
\PP\big((\mathcal B_\kappa^N)^c\big) & = \PP\big( \mathcal W_{w_2,1}^N(\mathcal F)_{T}\geq 5C'\max\{r_N,N^{-1/2}\}|w_2|_{1,\infty}(1+\kappa)\big) \\
& = \PP\big( \mathcal W_{w_2,1}^N(\mathcal F)_{T}\geq (1+u)C'\max(r_N,N^{-1/2})|w_2|_{1,\infty}\big) \\
& \leq (\exp(u)-1)^{-1} = (\exp(4+\kappa+4\kappa)-1)^{-1} \leq e^{-5\kappa}.
\end{align*}
It follows that 
$$\int_0^\infty \PP\big((\mathcal B_\kappa^N)^c\big)e^\kappa d\kappa \leq \int_0^\infty e^{-4\kappa}d\kappa = \tfrac{1}{4}.$$
In the same way, applying Theorem \ref{prop concentration final} with $(w_1,w_2) = (1,w_2)$ and up to an inflating the constant $C$ again, we obtain
\begin{align*}
\PP\big((\mathcal C_\kappa^N)^c\big) & = \PP\big( \mathcal W_{1,w_2}^N(\mathcal F)_{T}\geq 5C'\max(r_N,N^{-1/2})|w_2|_{\infty}\sqrt{T}(1+\kappa)\big) \leq e^{-5\kappa}
%& = \PP\big( \mathcal W_{1,w_2}^N(\mathcal F)_{T}\geq (1+u)C'N^{-1/2}|w_2|_{\infty}(1+\kappa)\big) \leq  e^{-5\kappa}
%& \leq (\exp(u)-1)^{-1} = (\exp(4+\kappa+4\kappa))^{-1} \leq e^{-5\kappa}.
\end{align*}
%by setting $1+u = 5(1+\kappa)$. 
Hence 
$\int_0^\infty \PP\big((\mathcal C_\kappa^N)^c\big)e^\kappa d\kappa \leq \tfrac{1}{4}$ follows likewise
and \eqref{eq: controle events} is proved with $\mathcal B_\kappa^N \cap \mathcal C_\kappa^N$ in place of $\mathcal A_\kappa^N$.\\
%and putting together the two estimates for 

\noindent {\bf Step 4)} We may now reproduce the proof of Theorem \ref{prop concentration final} with our new estimates from Step 2) : the estimate \eqref{eq: def c_1} now becomes
%\begin{equation} \label{new estimate}
$$
c'_1 = N(|b|_\infty+|\mu|_\infty)\frac{|w_1|_1}{|w_1w_2|_\infty}\big(|w_2|_g+
%5C'\max\{r_N,N^{-1/2}\}(|w_2|_{1,\infty}+\sqrt{T}|w_2|_\infty)
C_{w_2}^N\big)\;\;\text{and} \;\;c'_2 = c_2 = N^{-1}|w_1w_2|_\infty,$$
and thanks to Step 3), we may apply in this new setting Proposition \ref{prop concentration abstract} to obtain 
%\eqref{concentration quasi finale}. 
$$
\PP\big(\mathcal M_{\,w_1,w_2}^N(\mathcal F)_T \geq 8\varpi(c_1', c_2') \,\mathrm{e}(\mathcal F)(1+u)\big) \leq (e^{u \mathrm{e}(\mathcal F)}-1)^{-1}.
$$
Again,
we may pick $\varpi = \varpi(c'_1,c'_2) = k\sqrt{c'_1}c'_2$ with $k = 2 \sqrt{77}$, assuming $c_1 \geq 308$ which is true for $N$ is large enough, and 
it follows that
\begin{align*}
& 8 k\sqrt{c_1'}c_2'\mathrm{e}(\mathcal F) \\
 =& \, 8k\, \mathrm{e}(\mathcal F)(|b|_\infty+|\mu|_\infty)^{1/2}N^{-1/2}|w_1|_{1,\infty}|w_2|_\infty^{1/2}\big(|w_2|_g+C^N_{w_2}\big)^{1/2}\\
 \leq & \, C''N^{-1/2}|w_1|_{1,\infty}|w_2|_\infty^{1/2}\big(|w_2|_g+C^N_{w_2}\big)^{1/2}
\end{align*}
say, 
with
%\begin{equation} \label{eq: constante cruciale}
$$
C'' = C''(\mathrm{e}(\mathcal F),T,|b|_\infty, |\mu|_\infty) = 8k\, \max(1,\mathrm{e}(\mathcal F))(|b|_\infty+|\mu|_\infty)^{1/2} \max(5C'\sqrt{T}, 1)^{1/2}.
$$
%\end{equation}
%and \eqref{eq: first dev finale} follows. The proof of  \eqref{eq: sec dev finale} is obtained in the same way and is omitted. 
For $f \in \mathcal L_{\mathcal D}^{\,\mathrm{age}}$, define now
$$[f]_{1,\infty}^{\varepsilon_N} =  |f|_\infty^{1/2}\big(|f|_g+\varepsilon_N(|f|_{1,\infty}+|f|_\infty)\big)^{1/2}.$$ 
We have   
%$|f|_{g,N} = \big(|f|_\infty\big(|f|_g+N^{-1/2}\big(|f|_{1,\infty}+|f|_\infty)\big)^{1/2}$, 
proved that for $\varepsilon_N = \max(r_N,N^{-1/2})$, the sequence
$$(|w_1|_{1,\infty}[w_2]_{1,\infty}^{\varepsilon_N})^{-1}\mathcal M_{\,w_1,w_2}^N(\mathcal F)_T$$
%$$(|w_1|_{1,\infty}|w_2|_{1,\infty}^N)^{-1}\mathcal M_{\,w_1,w_2}^N(\mathcal F)_T$$
has a mild concentration property with rate $C''N^{-1/2}$. Applying the same argument as for Step 1) above, the mild concentration property carries over to 
$$(|w_1|_{1,\infty}[w_2]_{1,\infty}^{\varepsilon_N})^{-1}\mathcal W_{\,w_1,w_2}^N(\mathcal F)_T$$
with rate $C''\max(r_N, N^{-1/2})$, possibly up to inflating the constant $C'' >0$.\\

\noindent {\bf Step 5)} We finally show that 
%\begin{equation} \label{controle strange norm}
$[w_2]_{1,\infty}^{\varepsilon_N} \lesssim |w|_{1,\infty}$
%\big(|w_2|_{\infty}(|w_2|_1+\varepsilon_N|w_2|_\infty)\big)^{1/2} 
%\end{equation}
up to a constant that only depends on $|b|_\infty$, $|\mu|_\infty$, $|g_0|_\infty$ and $T$, under the additional assumption that $w_2$ has compact support and $|w_2|_\infty \lesssim \varepsilon_N^{-1}|w_2|_1$. By definition of $|w_2|_g$ in \eqref{def strange norm}, we have
$$|w_2|_g \lesssim |w_2|_1\big(\sup_{0 \leq t \leq T}\int_0^\infty g(t,a)da+|g|_\infty\big) \lesssim |w_2|_1$$
by the estimates of Lemma \ref{estimate g} in Appendix \ref{app analyse}. Moreover, the compact support of $w_2$ implies $|w_2|_{1,\infty} \leq |w_2|_\infty |\mathrm{supp}(w_2)|^{1/2} \lesssim |w_2|_\infty$. It follows that
$$[w_2]_{1,\infty}^{\varepsilon_N} \lesssim |w_2|^{1/2}\big(|w_2|_1+\varepsilon_N |w_2|_\infty \big)^{1/2} \lesssim |w_2|_{1,\infty}.$$
Let us note that the constant may possibly depend on $|\mathrm{supp}(w_2)|$ which is bounded above by $\mathfrak u$ by assumption.\\
%, that can be set out to be fixed once for all, 
%so the constant in the estimate is really free of $w_2$.\\ 

\noindent {\bf Step 6)} It remains to prove a mild concentration property for $([w_2]_{1,\infty}^{\varepsilon_N})^{-1}\mathcal W_{\,w_2}^N(\mathcal F)_T$ with rate $C''\max(r_N, N^{-1/2})$. The property holds for  
$$([w_2]_{1,\infty}^{\varepsilon_N})^{-1}\mathcal M_{\,w_2}^N(\mathcal F)_T$$ 
with the same proof as for $(|w_1|_{1,\infty}[w_2]_{1,\infty}^{\varepsilon_N})^{-1}\mathcal M_{\,w_1,w_2}^N(\mathcal F)_T$. We omit the details. Next, reproducing the beginning of the proof of Proposition \ref{prop stability} and applying \eqref{action} to the test function $a \mapsto w_2(t-a)f_t(a)$ with $f \in \mathcal F$, we obtain 
$$\mathcal W_{w_2}^N(\mathcal F)_T \leq \mathcal W_{w_2}^N(\mathcal F)_0+c_0^{-1}\big( \mathcal W_{w_2,1}^N(\mathcal F)_T+ \mathcal W_{1,w_2}^N(\mathcal F)_T\big)+\mathcal M_{w_2}^N(\mathcal  F)_T.$$
By Proposition \ref{global stability}, we further have
$$\mathcal W_{w_2,1}^N(\mathcal F)_T \lesssim \mathcal W_{1}^N(\mathcal F)_0
+\max_{h,k=1,w_2}\mathcal M_{h,k}^N(\mathcal F)_{T}$$
and
$$\mathcal W_{1,w_2}^N(\mathcal F)_T \lesssim \max_{k=1,w_2}\mathcal W_{k}^N(\mathcal F)_0
+\max_{h,k=1,w_2}\mathcal M_{h,k}^N(\mathcal F)_{T},$$
up to a constant that only depends on $T$, $c_0$, $|w_1|_1$ and $|w_2|_{L^1([0,T])}$,
therefore $\mathcal W_{w_2}^N(\mathcal F)_T$ is of order
%\begin{equation} \label{eq: big controle}
%$$
%\mathcal W_{w_2}^N(\mathcal F)_T \lesssim  
$$\max_{k=1,w_2}\mathcal W_{k}^N(\mathcal F)_0
+\max_{h,k=1,w_2}\mathcal M_{h,k}^N(\mathcal F)_{T}+\mathcal M_{w_2}^N(\mathcal  F)_T.
% + \mathcal M_{w_2}^N(\mathcal  F)_T.
%\end{equation}
$$
The mild concentration property of $([w_2]_{1,\infty}^{\varepsilon_N})^{-1}\mathcal M_{\,w_2}^N(\mathcal F)_T$ and $(|w_1|_{1,\infty}[w_2]_{1,\infty}^{\varepsilon_N})^{-1}\mathcal W_{\,w_1,w_2}^N(\mathcal F)_T$ enables us to control the last two terms. The first term has the correct order by Assumption \ref{concentration initiale}. The proof of Theorem \ref{thm concentration optimal} is complete.

\subsection{Remaining proofs of Section \ref{sec stability}} \label{sec remaining proofs} 
\subsubsection*{Proof of Proposition  \ref{prop coherence 1}} 
We repeat the argument of Step 6) in the proof of Theorem \ref{thm concentration optimal} above.
By Proposition \ref{global stability}, we have
%$$\mathcal W_{w_2,1}^N(\mathcal F)_T \lesssim \mathcal W_{1}^N(\mathcal F)_0
%+\max_{h,k=1,w_2}\mathcal M_{h,k}^N(\mathcal F)_{T}$$
%and
%$$\mathcal W_{1,w_2}^N(\mathcal F)_T \lesssim \max_{k=1,w_2}\mathcal W_{k}^N(\mathcal F)_0
%+\max_{h,k=1,w_2}\mathcal M_{h,k}^N(\mathcal F)_{T},$$
%up to a constant that only depends on $T$, $c_0$, $|w_1|_1$ and $|w_2|_{L^1([0,T])}$,
%therefore $\mathcal W_{w_2}^N(\mathcal F)_T$ is of order
%%\begin{equation} \label{eq: big controle}
%%$$
%%\mathcal W_{w_2}^N(\mathcal F)_T \lesssim  
$$\mathcal W_{w_2}^N(\mathcal F)_T \lesssim \max_{k=1,w_2}\mathcal W_{k}^N(\mathcal F)_0
+\max_{h,k=1,w_2}\mathcal M_{h,k}^N(\mathcal F)_{T}
% + \mathcal M_{w_2}^N(\mathcal  F)_T.
%\end{equation}
$$
and thus
$$\E\big[\mathcal W_{w_2}^N(\mathcal F)_T^p\big] \lesssim \E\big[\max_{k=1,w_2}\mathcal W_{k}^N(\mathcal F)_0^p\big]
+\E\big[\max_{h,k=1,w_2}\mathcal M_{h,k}^N(\mathcal F)_{T}^p\big]+\E\big[\mathcal M_{w_2}^N(\mathcal  F)_T^p\big],$$
up to a constant that depends on $p$, $T$, $c_0$, $|w_1|_1$ and $|w_2|_{L^1([0,T])}$. The first term is of order  $|w_2|_{1,\infty}^p r_N$ by Assumption. For the two other terms we use
the identity $\E\big[Z^p\big] = p\int_0^\infty x^{p-1}\PP(Z\geq x)dx$  for a nonnegative random variable $Z$ and conclude with the mild concentration property of $([w_2]_{1,\infty}^{\varepsilon_N})^{-1}\mathcal M_{\,w_2}^N(\mathcal F)_T$ and $(|w_1|_{1,\infty}[w_2]_{1,\infty}^{\varepsilon_N})^{-1}\mathcal W_{\,w_1,w_2}^N(\mathcal F)_T$. 

\subsubsection*{Proof of Proposition \ref{prop entropie minimale}}

Let $\mathcal F_0$ denote the minimal set that contains $0, c_0, c_0\mu, c_0b$ and that is stable under the operations defined in \eqref{eq def stabilite operateurs} except for the pointwise product $(f,g)\mapsto f\cdot g$. We also set, for $f \in \mathcal L_{\mathcal D}^\infty $:
$$\mathcal A(f)_{t_1,t_2}^{(k,l)} = \big((s,a) \mapsto f(t_1, t_2+ka-ls)\big)$$
with $t_1,t_2 \in [0,T]$ and $k,l=0,1$.\\

\noindent {\bf Step 1)} We claim that 
\begin{equation} \label{eq characterisation mathcal F_0}
\mathcal F_0 \subseteq \Big\{0, \pm c_0, \pm c_0\mu, \pm c_0b, \pm\, \mathcal L(c_0b)_{t_1,t_2}^{(k,l)}, \pm\, \mathcal A(c_0\mu)_{t_1,t_2}^{(k,l)}, \;\text{for every}\;t_1,t_2 \in [0,T], k,l=0,1\Big\}.
\end{equation} 
Indeed, one can check the following stability properties:
$$\mathsf s_t(\mathcal A(f)_{t_1,t_2}^{(k,l)})(s,a)  = \mathcal A_{t_1,t_2}^{(k,l)}(t,t+a) 
 = f(t_1,t_2+kt+ka-lt) 
 = \mathcal A_{t_1,t_2+kt-lt}^{(k,0)}(s,a),
$$
$$\mathsf t_t(\mathcal A(f)_{t_1,t_2}^{(k,l)})(s,a)  = \mathcal A_{t_1,t_2}^{(k,l)}(t,t-s) 
 = f(t_1,t_2-lt+ka-ks) 
 = \mathcal A_{t_1,t_2+kt-lt}^{(0,k)}(s,a),
$$
$$\mathsf u_t(\mathcal A(f)_{t_1,t_2}^{(k,l)})(s,a)  = \mathcal A_{t_1,t_2}^{(k,l)}(t,t-s+a) 
 = f(t_1,t_2+kt-ks+ka-lt) 
 = \mathcal A_{t_1,t_2+kt-lt}^{(k,k)}(s,a).
$$
This proves \eqref{eq characterisation mathcal F_0}.\\
 
\noindent {\bf Step 2)} We now prove that if $b,\mu \in \mathcal C^s$ for some $0 < s \leq 1$ with H\" older constant $L>0$, then
\begin{equation} \label{eq controle entropie petite classe}
\mathcal N(\mathcal F_0,|\cdot|_\infty,\epsilon) \lesssim \epsilon^{-2/s},
\end{equation} 
up to a constant that only depends on $s$, $T$ and $L$. Indeed, if $f \in \mathcal C^s$ with H\"older constant $L>0$,  we have 
\begin{align*}
\big|\mathcal L(f)_{t_1,t_2}^{(k,l)}-\mathcal L(f)_{t'_1,t'_2}^{(k,l)}\big|_\infty & =\sup_{s,a} |f(t_1, t_2+ka-ls)- f(t_1, t_2+ka-ls)| \\
& \leq L(|t_1-t_1'|^s+|t_2-t_2'|^s),
\end{align*}
therefore, for fixed $(k,l)$ and $f \in \mathcal C^s$, the $\epsilon$-covering number of $\{\mathcal L(f)_{t_1,t_2}^{(k,l)}, t_1,t_2 \in [0,T]\}$ in $|\cdot|_\infty$ is the same as that of $[0,T]^2$ equipped with the metric $d\big((t_1,t_2)-(t_1', t_2')\big) =  L(|t_1-t_1'|^s+|t_2-t_2'|^s)$. Since $\mathcal N([0,T],\epsilon,L|\cdot|^\gamma) = T\,\mathcal N([0,T],(\epsilon/L)^{1/s},|\cdot|) = TL^{1/s}\epsilon^{-1/s}$, we have that  $\mathcal N([0,T]^2, \epsilon, d) \lesssim \epsilon^{-2/s}$ and \eqref{eq controle entropie petite classe} is established.\\

\noindent {\bf Step 3)} We now consider the class $\mathcal F_0^{\mathrm{prod}}$ that contains $\mathcal F_0$ and that is stable under the operation $(f,g) \mapsto fg$. Since $\mathsf s_t(fg) = \mathsf s_t(f)\mathsf s_t(g)$, $\mathsf t_t(fg) = \mathsf t_t(f)\mathsf s_t(g)$, $\mathsf u_t(fg) = \mathsf u_t(f)\mathsf s_t(g)$, the class $\mathcal F_0^{\mathrm{prod}}$ contains the minimal class $\mathcal F$.\\ 

Let $f = \prod_{\ell=1}^m f_\ell \in \mathcal F_0^{\mathrm{prod}}$, with $f_\ell \in \mathcal F_0$. For every $\ell$, we have $|f_\ell |_\infty \leq c_1 < 1$, with $c_1 = c_0\max(b|_\infty,|\mu|_\infty) < 1$ by assumption. Therefore, if $m \geq \log \epsilon/\log c_1 = m(\epsilon)$, we have $|f|_\infty = |f-0|_\infty \leq \epsilon$. Now, let $g_i$ be $\mathcal N(\mathcal F_0, \epsilon m(\epsilon)^{-1}, |\cdot|_\infty)$ functions in $\mathcal F_0$ such that, for every $f\in \mathcal F_0$, there exists an index $i(f)$ such that $|f-g_{i(f)}|_\infty \leq \epsilon m(\epsilon)^{-1}$. If $m \leq \log \epsilon/\log c_0$, we have
$$\big|f-\prod_{\ell = 1}^mg_{i(f_\ell)}\big|_\infty = \big|\prod_{\ell = 1}^mf_\ell-\prod_{\ell = 1}^mg_{i(f_\ell)}\big|_\infty \leq c_1^{m-1}\,m\epsilon m(\epsilon)^{-1} \leq \epsilon.$$
As a result, the family $\big\{0, \prod_{\ell = 1}^kg_\ell, k=1, \ldots, m(\epsilon)\big\}$ is a family of centers of balls of radius at most $\epsilon$ that are sufficient to cover $\mathcal F_0^{\mathrm{prod}}$. It follows that
$$
\mathcal N(\mathcal F_0^{\mathrm{prod}}, \epsilon, |\cdot|_\infty) \leq \mathcal N(\mathcal F_0, m(\epsilon) \epsilon, |\cdot|_\infty)^{m(\epsilon)+1} \lesssim \big(\epsilon m(\epsilon)\big)^{-2m(\epsilon)/s}.
$$
\noindent {\bf Step 4)} We have established $\mathcal F \subseteq \mathcal F_0^{\mathrm{prod}}$ and therefore
\begin{align*}
e(\mathcal F)  = \int_0^1 \log\big(1+\mathcal N(\mathcal F, |\cdot|_\infty, \epsilon)\big)d\epsilon &  \leq \int_0^1\log\big(1+\mathcal N(\mathcal F_0^{\mathrm{prod}}, \epsilon, |\cdot|_\infty)\big) d\epsilon \\
& \lesssim \int_0^1 \log \big(\epsilon m(\epsilon)\big)^{-2m(\epsilon)/s}d\epsilon \lesssim \int_0^1 (\log \epsilon)^2 d\epsilon < \infty.
\end{align*}
The proof of Proposition \ref{prop entropie minimale} is complete.

\section{Proofs of Section \ref{sec stat oracle} and \ref{sec: minimax adaptive}} \label{sec proof stat}

\subsection{Proof of Theorem \ref{thm oracle g}}

Remember that the condition $r_N \leq N^{-1/2}$ is in force in this section.

\subsubsection*{Preliminaries} We first write a standard bias-variance decomposition in squared-error loss, based upon the stability result of Corollary \ref{prop coherence 1}.

\begin{lem} \label{lem decomp biais variance} Let $h \in \mathcal G_1^N$. If $\widehat g_h^N$ is specified with a bounded and compactly supported kernel $K$, we have
$$\E\big[\big(\widehat g_h^N(t,a)-g(t,a)\big)^2\big] \lesssim \mathcal B_h^N(g)(t,a)^2+\mathsf{V}_h^N,$$
where $B_h(g)(t,a)$ and $\mathsf{V}_h^N$ are defined in \eqref{def bias} and \eqref{def upper variance} respectively.
\end{lem}
\begin{proof} Write $\widehat g_{h}^N(t,a)-g(t,a) = I + II$,
with
$$I = \int_{0}^\infty K_h(u-a)g(t,u)du-g(t,a)$$
and
$$II = \int_{\R_+} K_h(u-a)\big(Z_t^N(du)-g(t,u))du.$$
% \leq \mathcal W_{K_h(t-a-\cdot)}^N(\mathcal F)_t.$$
We have $I^2 \leq \mathcal \mathcal B_h(g)(t,a)^2$. For the stochastic term, 
% = \big| K_h \star g(t,a) - g(t,a)\big|.$$
%with $K_h=\{K_h\}_+-\{K_{h}\}_-$ and using that $-\mathcal F=\mathcal F$, 
we have
\begin{equation} \label{eq: sup omega}
| II |\leq  \mathcal W_{K_h(t-a-\cdot)}^N(\mathcal F)_t
\end{equation}
%{\color{blue} {\tt [explain why $\mathcal W_{K_h(t-a-\cdot)}^N(\mathcal F)_t \leq \mathcal W_{\{K_h|(t-a-\cdot)\}_+}^N(\mathcal F)_t$]}}
Moreover
\begin{equation} \label{eq estimate L1-Linfty}
| K_h(t-a-\cdot)|_\infty \leq |K_h(t-a-\cdot)|_\infty = h^{-1}|K|_\infty \lesssim |K_h(t-a-\cdot)|_1 N^{1/2}
\end{equation}
as soon as $h^{-1} \lesssim N^{1/2}$ since $|K_h(t-a-\cdot)|_1 = |K|_1=1$. This condition is true for any $h \in \mathcal G_1^N$ using the fact that $K$ is bounded and compactly supported. We may then apply Corollary  \ref{prop coherence 1} and obtain 
%\begin{equation} \label{eq variance g}
$$
\E[II^2]   \lesssim |K_h(t-a-\cdot)|_{1,\infty}^2 N^{-1} \lesssim \big(C^\star N^{-1/2}|K_h|_{1,\infty}\big)^2 = \mathsf{V}_h^N.
$$
\end{proof}

\subsubsection*{Completion of proof of Theorem \ref{thm oracle g}}
We essentially repeat the main argument of the Goldenshluger-Lepski method (see {\it e.g.} \cite{GOLDENSHLUGERLEPSKI1, GOLDENSHLUGERLEPSKI2} for the pointwise risk) in a setting that we need to adapt to our context.\\
 
\noindent {\bf Step 1)} For any $h \in \mathcal G_1^N$, forcing $\widehat g_h^N(t,a)$ in the risk decomposition and by definition of $\mathsf A_h^N(t,a)$ and $\widehat h^N(t,a)$, we successively have
\begin{align*}
& \E\big[\big(\widehat g_\star^N(t,a)-g(t,a)\big)^2\big]\\
  \lesssim &\; \E\big[\big(\widehat g_\star^N(t,a)-\widehat g_h^N(t,a)\big)^2\big]+ \E\big[\big(\widehat g_h^N(t,a)-g(t,a)\big)^2\big]  \\
 \lesssim & \;\E\big[\big\{\big(\widehat g_{\widehat h^N(t,a)}^N(t,a)-\widehat g_h^N(t,a)\big)^2-\mathsf V_h^N-\mathsf V_{\widehat h^N(t,a)}^N\big\}_+ +\mathsf V_h^N+\mathsf V_{\widehat h(t,a)}^N \big] + \E\big[\big(\widehat g_h^N(t,a)-g(t,a)\big)^2\big]  \\
  \lesssim & \;\E\big[\mathsf A_{\max(\widehat h^N(t,a),h)}^N(t,a) +\mathsf V_h^N+\mathsf V_{\widehat h^N(t,a)}^N \big] + \E\big[\big(\widehat g_h^N(t,a)-g(t,a)\big)^2\big]  \\
  \lesssim & \;\E\big[\mathsf A_{h}^N(t,a)\big]+\mathsf V_h^N +\E\big[\mathsf A_{\widehat h^N(t,a)}^N+\mathsf V_{\widehat h^N(t,a)}^N \big] + \E\big[\big(\widehat g_h^N(t,a)-g(t,a)\big)^2\big] \\
  \lesssim &  \;\E\big[\mathsf A_{h}^N(t,a)\big]+\mathsf V_h^N + \mathcal B_h^N(g)(t,a)^2  
\end{align*}
where we applied Lemma \ref{lem decomp biais variance} to obtain the last line.\\

\noindent {\bf Step 2)} We first estimate $\mathsf A_{h}^N(t,a)$. Write $g_h(t,a)$ for $\int_{\R_+}K_h(u-a)g(t,u)du$. For $h,h'\in \mathcal G_1^N$ with $h' \leq h$, since
\begin{align*}
& \big(\widehat g_h^N(t,a)-\widehat g_{h'}(t,a)\big)^2 \\
\leq  & \;4\big(\widehat g_h^N(t,a)-g_{h}(t,a)\big)^2+4\big(g_h(t,a)-g(t,a)\big)^2+4\big(g_{h'}(t,a)-g(t,a)\big)^2+4\big(\widehat g_{h'}^N(t,a)-g_{h'}(t,a)\big)^2,  
\end{align*}
we have
\begin{align*}
& \big(\widehat g_h^N(t,a)-\widehat g_{h'}(t,a)\big)^2-\mathsf V_h^N-\mathsf V_{h'}^N \\
\leq &\;8\mathcal B_h^N(g)(t,a)^2+\big(4(\widehat g_h^N(t,a)-g_{h}(t,a))^2-\mathsf V_h^N\big)+\big(4(\widehat g_{h'}^N(t,a)-g_{h'}(t,a))^2-\mathsf V_{h'}^N\big).
\end{align*}
using $h' \leq h$ in order to bound $(\widehat g_{h'}^N(t,a)-g_{h'}(t,a))^2$ by the bias at scale $h$. It follows that
\begin{align*}
& \big(\widehat g_h^N(t,a)-\widehat g_{h'}(t,a)\big)^2-\mathsf V_h^N-\mathsf V_{h'}^N \\
\leq &\;8\mathcal B_h^N(g)(t,a)^2+ 4\big(\widehat g_h^N(t,a)-g_{h}(t,a)\big)^2-\mathsf V_h^N+4\big(\widehat g_{h'}^N(t,a)-g_{h'}(t,a)\big)^2-\mathsf V_{h'}^N,
\end{align*}
and taking maximum over $h'\leq h$, we obtain
\begin{align}
& \max_{h' \leq h}\big\{\big(\widehat g_h^N(t,a)-\widehat g_{h'}(t,a)\big)^2-\mathsf V_h^N-\mathsf V_{h'}^N\big\}_+  \label{eq end lepski}\\
\leq & \;8\mathcal B_h^N(g)(t,a)^2+  \big\{4\big(\widehat g_h^N(t,a)-g_{h}(t,a)\big)^2-\mathsf V_h^N\big\}_++\max_{h' \leq h}\big\{4\big(\widehat g_{h'}^N(t,a)-g_{h'}(t,a)\big)^2-\mathsf V_{h'}^N\big\}_+. \nonumber%\label{eq end lepski}
\end{align}

\noindent {\bf Step 3)} We estimate the expectation of the first stochastic term in the right-hand side of \eqref{eq end lepski}. Since $|\widehat g_h^N(t,a)-g_h(t,a)| \leq \mathcal W_{K_h(t-a-\cdot)}^N $, we successively have
\begin{align*}
 \E\big[\big\{4\big(\widehat g_h^N(t,a)-g_{h}(t,a)\big)^2-\mathsf V_h^N\big\}_+\big] &  = \int_0^\infty \PP\big(4\big(\widehat g_h^N(t,a)-g_{h}(t,a)\big)^2-\mathsf V_h^N \geq \kappa\big)d\kappa \\  
 & = \int_0^\infty \PP\big(|\widehat g_h^N(t,a)-g_h(t,a)| \geq \tfrac{1}{2}(\mathsf V_h^N+\kappa)^{1/2}\big)d\kappa \\
  & \leq \int_0^\infty \PP\big(\mathcal W_{K_h(t-a-\cdot)}^N \geq \tfrac{1}{2}(\mathsf V_h^N+\kappa)^{1/2}\big)d\kappa.
\end{align*}
We may apply Theorem \ref{thm concentration optimal} with $w_2=K_h(t-a-\cdot)$ since $K$ is compactly supported 
%and $|w_2|_1 \lesssim N^{1/2}|w_2|_1$ holds, see {\it e.g.}\;
and having \eqref{eq estimate L1-Linfty} of Lemma \ref{lem decomp biais variance} above. By the change of variable 
$$\tfrac{1}{2}(\mathsf V_h^N+\kappa)^{1/2}=(1+u)C''|K_{h}|_{1,\infty} N^{-1/2},$$ we then obtain
%$\E\big[\big\{4\big(\widehat g_h^N(t,a)-g_{h}(t,a)\big)^2-\mathsf V_h^N\big\}_+\big]$ is bounded above by 
\begin{align*}
& \E\big[\big\{4\big(\widehat g_h^N(t,a)-g_{h}(t,a)\big)^2-\mathsf V_h^N\big\}_+\big] \\
\leq & \; 8C''|K_h|_{1,\infty}N^{-1/2}\int_{\tfrac{1}{2C''}(\mathsf V_h^N)^{1/2}|K_h|_{1,\infty}^{-1}N^{1/2}-1}^\infty 
 %\PP\big(\mathcal W_{K_h(t-a-\cdot)}^N \geq (1+u)C''|K_{h}|_{1,\infty}\max(N^{-1/2},r_N)\big)
 (1+u)\min\big((e^u-1)^{-1},1\big)du \\
 \lesssim & \;\exp\big(-\tfrac{1}{2C''}(\mathsf V_h^N)^{1/2}|K_h|_{1,\infty}^{-1}N^{1/2}\big) \leq  N^{-2}
\end{align*}
by definition of $\mathsf V_h^N$.\\

\noindent {\bf Step 4)} For the second stochastic term, we use the rough estimate
\begin{align*}
\E\big[\max_{h' \leq h}\big\{4\big(\widehat g_{h'}^N(t,a)-g_{h'}(t,a)\big)^2-\mathsf V_{h'}^N\big\}_+\big] & \leq \sum_{h'\leq h}\E\big[\big\{4\big(\widehat g_{h'}^N(t,a)-g_{h'}(t,a)\big)^2-\mathsf V_{h'}^N\big\}_+\big] \\
& \lesssim \mathrm{Card}(\mathcal G_1^N) N^{-2} \lesssim N^{-1} 
\end{align*}
where we used Step 3) to bound each term $\E\big[\big\{4\big(\widehat g_{h'}^N(t,a)-g_{h'}(t,a)\big)^2-\mathsf V_{h'}^N\big\}_+\big]$ independently of $h$ together with $\mathrm{Card}(\mathcal G_1^N) \lesssim N$. In conclusion, we have proved through Steps 2)-4) that $\E\big[\mathsf A_{h}^N(t,a)\big] \lesssim \delta_N$. Therefore, from Step 1), we conclude
$$\E\big[\big(\widehat g_\star^N(t,a)-g(t,a)\big)^2\big] \lesssim \mathcal B_h^N(g)(t,a)^2 + \mathsf V_h^N + \delta_N$$  
for any $h \in \mathcal G_1^N$. The proof of Theorem \ref{thm oracle g} is complete.

\subsection{Proof of Theorem \ref{thm oracle mu}}

\subsubsection*{Preliminaries} We first study the behaviour of the process $\Gamma^N(dt,da)$ of death occurences introduced in Section \ref{sec constrution estimators} and represented via \eqref{def death jumps}.

\begin{lem} \label{lem bracket}
With the notation of Section \ref{sec construction model}, we have
\begin{equation} \label{eq autre rep death process}
\Gamma^N(dt,da) = N^{-1}\int_{\mathbb N\setminus \{0\} \times \R_+} \delta_{a_i(Z_{s^-}^N)}(da){\bf 1}_{\{0 \leq \vartheta \leq \mu(s,a_i(Z_{s^-}^N)), i \leq \langle NZ_{s^-}^N,{\bf 1}\rangle\}}\mathcal Q_2(dt,di,d\vartheta),
\end{equation}
where $\mathcal Q_2$ is a Poisson random measure on $\R_+\times \mathbb N \setminus \{0\} \times \R_+$ with intensity $dt \big(\sum_{k \geq 1}\delta_k(di)\big) d\vartheta$.
Moreover, for nonnegative weights $w_1\in \mathcal L_{\mathcal D}^{\;\mathrm{time}}$ and $w_2 \in \mathcal L_{\mathcal D}^{\,\mathrm{age}}$, we have
\begin{equation} \label{eq ineq death mg}
\big|\int_{0}^T\int_{\R_+} w_1(s)w_2(s-u)\big(\Gamma^N(ds,du)-\mu(s,u)g(s,u)duds\big)\big| \leq \mathcal W_{w_1,w_2}^N(\mathcal F)_T+|(\Delta^N_{w_1,w_2})_T|,
\end{equation}
where $t \mapsto (\Delta_{w_1,w_2}^N)_t$ is a square integrable martingale with predictable compensator
\begin{equation} \label{def bracket mg saut}
\langle \Delta_{w_1,w_2}^N \rangle_t = N^{-1}\int_0^t \int_{\R_+} w_1(s)^2w_2(s-u)^2\mu(s,u)Z_{s^-}^N(du)ds.
\end{equation}
%Moreover we have
%\begin{equation} \label{eq maj bracket}
%\langle \Delta_{w_1,w_2}^N \rangle_T  \leq N^{-1}|\mu|_\infty\big(|w_1|^2_2|w_2|_2^2|g|_\infty+\mathcal W_{w_1^2,w_2^2}(\mathcal F)_T\big).
%% & \leq N^{-1}|\mu|_\infty\big(|w_1|^2_{1,\infty}|w_2|_{1,\infty}^2|g|_\infty+\mathcal W_{w_1^2,w_2^2}(\mathcal F)_T\big).
 %\end{equation}
\end{lem}
\begin{proof} The representation \eqref{eq autre rep death process} is straightforward. We add and substract in the left-hand side of  \eqref{eq ineq death mg} the term $\int_{0}^T\int_{0}^\infty w_1(s)w_2(s-u)\mu(s,u)Z_s^N(du)ds$
and obtain the desired inequality with
\begin{align*}
 (\Delta_{w_1,w_2}^N&)_t  
 =  \; \int_0^t\int_{\R_+} w_1(s)w_2(s-u)\big(\Gamma^N(ds,du)-\mu(s,u)Z_s^N(du)\big) \\
 = & \;N^{-1}\int_0^t\int_{\mathbb N\setminus \{0\} \times \R_+} w_1(s)w_2(s-u)\delta_{a_i(Z_{s^-}^N)}(da){\bf 1}_{\{0 \leq \vartheta \leq \mu(s,a_i(Z_{s^-}^N)), i \leq \langle NZ_{s^-}^N,{\bf 1}\rangle\}}\widetilde{\mathcal Q}_2(ds,di,d\vartheta),
\end{align*}
where $\widetilde{\mathcal Q}_2(ds,di,d\vartheta) = {\mathcal Q}_2(ds,di,d\vartheta)-ds \big(\sum_{k \geq 1}\delta_k(di)\big) d\vartheta$ is the associated compensated measure. Thus $(\Delta_{w_1,w_2}^N)_t  $ is a martingale and \eqref{def bracket mg saut} follows. 
%The proof of \eqref{eq maj bracket} is easily obtained by inserting the term $N^{-1}\int_0^t \int_{\R_+}w_1(s)^2w_2(s-u)^2g(s,u)duds$ in \eqref{def bracket mg saut}.
\end{proof}

We next study the deviation of $(\Delta_{w_1,w_2}^N)_T$. Define
\begin{equation}
V_{w_1,w_2}^N = \big(4C^\star (\log N) N^{-1/2}|w_1|_{1,\infty}|w_2|_{1,\infty}\big)^2.
\end{equation}
where $C^\star$ is the constant defined in \eqref{def upper bivariance} in Section \ref{sec oracle inequalities}. Let also
$$\chi_{w_1,w_2}^N=N^{-1}|w_1|_\infty |w_2|_\infty|\mu|_\infty$$ and
%$$\xi^N_{w_1,w_2} = N^{-1}|\mu|_\infty\big(|w_1|^2_2|w_2|_2^2|g|_\infty+(V_{w_1,w_2}^N)^{1/2}\big).$$
\begin{equation} \label{eq def xi}
\xi^N_{w_1,w_2} = 16 N^{-1}|\mu|_\infty|g|_\infty |w_1|^2_2|w_2|_2^2 (V_{w_1,w_2}^N)^{-1/2}(\log N) 
%{\color{red} (\log N)}.
\end{equation}

\begin{lem} \label{lem dev mg saut}
For 
%\begin{equation} \label{eq cond lemma}
$u > 2^{-6}V_{w_1,w_2}^N (\log N)^{-2}$,
%\end{equation}
we have
$$\PP\big(\big|(\Delta_{w_1,w_2}^N)_T \big| \geq u^{1/2}\big) \leq 2\exp\Big(-\frac{u^{1/2}}{2(\chi_{w_1,w_2}^N+\xi^N_{w_1,w_2})}\Big)+2\PP\big( N^{-1}|\mu|_\infty\mathcal W_{w_1^2,w_2^2}(\mathcal F)_T \geq \tfrac{1}{2}\xi_{w_1,w_2}^Nu^{1/2}\big).$$
\end{lem}
\begin{proof} We plan to apply a classical deviation inequality for martingales (see {\it e.g.} Lemma 2.1 in van de Geer \cite{VANDEGEER} or the classical textbook by Shorak and Wellner \cite{SHORACKWELLNER}), namely:
\begin{equation} \label{eq shorak wellner}
\PP\big((\Delta_{w_1,w_2}^N)_T \geq v, \langle \Delta_{w_1,w_2}^N \rangle_T \leq w\big) \leq \exp\Big(-\frac{v^2}{2(v\chi_{w_1,w_2}^N+w)}\Big)
\end{equation}
for every $v,w \geq 0$, 
where $\chi_{w_1,w_2}^N=N^{-1}|w_1|_\infty |w_2|_\infty|\mu|_\infty$ is an almost-sure bound of the size of the jumps of  $(\Delta_{w_1,w_2}^N)_T$. With $v=u^{1/2}$ and $w = \xi_{w_1,w_2}^Nu^{1/2}$, inequality \eqref{eq shorak wellner} gives
$$\PP\big(\big|(\Delta_{w_1,w_2}^N)_T \big| \geq u^{1/2}\big) \leq 2\exp\Big(-\frac{u^{1/2}}{2(\chi_{w_1,w_2}^N+\xi^N_{w_1,w_2})}\Big)+2\PP\big(\langle \Delta_{w_1,w_2}^N\rangle_T \geq \xi_{w_1,w_2}^Nu^{1/2}\big).$$ Inserting the term $N^{-1}\int_0^t \int_{\R_+}w_1(s)^2w_2(s-u)^2g(s,u)duds$ in \eqref{def bracket mg saut}, we obtain
%\begin{equation} \label{eq maj bracket}
$$
\langle \Delta_{w_1,w_2}^N \rangle_T  \leq N^{-1}|\mu|_\infty\big(|w_1|^2_2|w_2|_2^2|g|_\infty+\mathcal W_{w_1^2,w_2^2}(\mathcal F)_T\big),
% & \leq N^{-1}|\mu|_\infty\big(|w_1|^2_{1,\infty}|w_2|_{1,\infty}^2|g|_\infty+\mathcal W_{w_1^2,w_2^2}(\mathcal F)_T\big).
$$
% \end{equation}
 therefore
 \begin{align*}
 \PP\big(\langle \Delta_{w_1,w_2}^N\rangle_T \geq \xi_{w_1,w_2}^Nu^{1/2}\big) \leq
% \\
%\leq & {\bf 1}_{\{N^{-1}|\mu g|_\infty |w_1|^2_2|w_2|_2^2 \geq \tfrac{1}{2}\xi_{w_1,w_2}^Nu^{1/4}\}} +
 \PP\big( N^{-1}|\mu|_\infty\mathcal W_{w_1^2,w_2^2}(\mathcal F)_T \geq \tfrac{1}{2}\xi_{w_1,w_2}^Nu^{1/2}\big)
 \end{align*}
as soon as
\begin{equation} \label{eq cond determinisite}
N^{-1}|\mu|_\infty |g|_\infty |w_1|^2_2|w_2|_2^2 < \tfrac{1}{2}\xi_{w_1,w_2}^Nu^{1/2},
\end{equation} 
but by definition of $\xi_{w_1,w_2}^N$ in \eqref{eq def xi}, this condition is equivalent to $u > 2^{-6}V_{w_1,w_2}^N (\log N)^{-2}$.
 %Moreover, by \eqref{eq maj bracket} and the definition of $C^N(w_1,w_2)$, we have
%\begin{align*}
%& \; \PP\big(\langle \Delta_{w_1,w_2}^N \rangle_T> C^N(w_1,w_2)\big) \\
% \leq  & \; \PP\big(\mathcal W_{w_1^2,w_2^2}(\mathcal F)_T \geq (1+(\log N)^2)\max(N^{-1/2},r_N)C|w_1^2|_{1,\infty}|w_2^2|_{1,\infty}\big) 
%%\lesssim N^{-2}
%\end{align*}
%and this last quantity is of order $N^{-2}$ by Theorem \ref{thm concentration optimal}.  Lemma \ref{lem dev mg saut} is thus a consequence of \eqref{eq shorak wellner} with $v =C^N(w_1,w_2)$. 
%provided the probability of the event $\{\langle \Delta_{w_1,w_2}^N \rangle_T > C^N(w_1,w_2)\}$ is of order $\delta_N$. By \eqref{eq maj bracket} and the definition of $C^N(w_1,w_2)$, we have
%$$C^N(w_1,w_2) = N^{-1}|\mu|_\infty\big(|w_1|^2_2|w_2|_2^2|g|_\infty+(1+(\log N)^2)\max(N^{-1/2},r_N)C|w_1^2|_{1,\infty}|w_2^2|_{1,\infty}\big)$$
%where $C$ is the constant of Theorem \ref{thm concentration optimal}. By \eqref{eq shorak wellner}, we have
\end{proof}

% of $\widehat \gamma_{\boldsymbol h}^N(t,a)-\gamma_{\boldsymbol h}(t,a)$ with $\gamma_{\boldsymbol h}(t,a) = (\mathcal K_{\boldsymbol h} \circ \varphi) \star (\mu g)(t,a)$ for $\boldsymbol h \in \mathcal G_2^N$.

%\newpage

Under Assumption \ref{assumption minoration g}, we have a uniform lower bound on $g(t,a)$. 

\begin{lem} \label{lem min g}
Work under Work under Assumptions  \ref{H basic} and \ref{assumption minoration g}. Then, there exists $\epsilon >0$ depending on $\delta(t,a)$ defined in \eqref{eq condit mino g L} and \eqref{eq condit mino g U} and $|\mu|_\infty$ and $T$ such that  $g(t,a) \geq \epsilon$.
\end{lem}
The proof uses an explicit representation of $g(t,a)$ established in Proposition \ref{P first smoothness} and is delayed until Appendix \ref{proof lemma assumption minoration g}.
\subsubsection{Completion of proof of Theorem \ref{thm oracle mu}}
Let $(h,\boldsymbol h) \in \mathcal G_1^N \times \mathcal G_2^N$ and set $\pi(t,a) = \mu(t,a)g(t,a)$.\\

\noindent {\bf Step 1)} We plan to use the following decomposition
\begin{align*}
\widehat \mu_{h,\boldsymbol h}^N(t,a)_\varpi - \mu(t,a) = I+II,
\end{align*}
with
$$I = \frac{\pi(t,a)\big(g(t,a)-\widehat g_h^N(t,a) \vee \varpi\big)}{g(t,a)\widehat g_h^N(t,a)\vee \varpi}  
$$
and
$$II = \frac{\big(\widehat \pi_{\boldsymbol h}^N(t,a)-\pi(t,a)\big)g(t,a)}{g(t,a)\widehat g_h^N(t,a)\vee \varpi}.
$$
First, we have
\begin{align*}
|I| & \leq (\epsilon \varpi)^{-1}|\mu|_\infty|g|_\infty|g(t,a)-\widehat g_h^N(t,a) \vee \varpi | \leq (\epsilon \varpi)^{-1}|\mu|_\infty|g|_\infty|g(t,a)-\widehat g_h^N(t,a)|
\end{align*}
thanks to Lemma \ref{lem min g} as soon as $\varpi \leq \epsilon \leq g(t,a)$.
In the same way, 
$$|II| \leq  (\epsilon \varpi)^{-1}|g|_\infty |\widehat \pi_{\boldsymbol h}^N(t,a)-\pi(t,a)|$$
follows. Picking $h = \widehat h^N(t,a)$, $\boldsymbol h = \widehat {\boldsymbol h}^N(t,a)$ and taking square and expectation, we have thus established
\begin{equation} \label{eq first maj mu}
\E\big[\big(\widehat \mu_\star^N(t,a)_{\varpi}- \mu(t,a)\big)^2\big] \lesssim \E\big[\big(\widehat g_{\widehat h^N(t,a)}^N(t,a)-g(t,a)\big)^2\big]+ \E\big[\big(\widehat \pi_{\widehat {\boldsymbol h}^N(t,a)}^N(t,a)-\pi(t,a)\big)^2\big]
\end{equation}
as soon as $\varpi \leq \epsilon$. By Theorem \ref{thm oracle g}, we already have the desired bound for the first term in the right-hand side of \eqref{eq first maj mu}.\\

\noindent {\bf Step 2)} We study the second term in the right-hand side of \eqref{eq first maj mu}. For any $\boldsymbol h \in \mathcal G_2^N$, repeating Step 1) of the proof of Theorem \ref{thm oracle g}, we have
$$\E\big[\big(\widehat \pi_{\boldsymbol h}^N(t,a)-\pi(t,a)\big)^2\big] \lesssim \E\big[\mathsf A_{\boldsymbol h}^N(t,a)\big]+\mathsf V_{\boldsymbol h}^N + \mathcal B_{\boldsymbol h}^N(\pi)(t,a)^2.$$
In order to estimate $\E\big[\mathsf A_{\boldsymbol h}^N(t,a)\big]$, we repeat Step 2) of the proof of Theorem \ref{thm oracle g} and obtain
\begin{align}
& \max_{\boldsymbol {h}' \leq \boldsymbol h}\big\{\big(\widehat \pi_{\boldsymbol h}^N(t,a)-\widehat \pi_{{\boldsymbol h}'}(t,a)\big)^2-\mathsf V_{\boldsymbol h}^N-\mathsf V_{{\boldsymbol h}'}^N\big\}_+ \label{eq controle sup idem} \\
\lesssim & \;\mathcal B_{\boldsymbol h}^N(\pi)(t,a)^2+  \big\{4\big(\widehat \pi_{\boldsymbol h}^N(t,a)-\pi_{h}(t,a)\big)^2-\mathsf V_{\boldsymbol h}^N\big\}_++\max_{{\boldsymbol h}' \leq {\boldsymbol h}}\big\{4\big(\widehat \pi_{{\boldsymbol h}'}^N(t,a)-\pi_{{\boldsymbol h}'}(t,a)\big)^2-\mathsf V_{{\boldsymbol h}'}^N\big\}_+. \nonumber 
\end{align}

\noindent {\bf Step 3)} We estimate the expectation of the first stochastic term in the right-hand side of the last inequality. 
%Using the decomposition $\mathsf V_{\boldsymbol h}^N = I+II,$ with
%$$I = \big(4 (\log N)C^\star\max(N^{-1/2},r_N)|H_{h_1}|_{1,\infty} |K_{h_2}|_{1,\infty}\big)^2,$$
%and
%$$II = ,$$
Using the same trick as in \eqref{eq: sup omega}, we have by \eqref{eq ineq death mg} that
$$\big\{4\big(\widehat \pi_{\boldsymbol h}^N(t,a)-\pi_{h}(t,a)\big)^2-\mathsf V_{\boldsymbol h}^N\big\}_+ \lesssim I + II,$$
with
$$I = \big\{8\mathcal W_{H_{h_1}(\cdot-t), K_{h_2}(\cdot - (t-a))}^N(\mathcal F)_T^2-\tfrac{1}{2}\mathsf V_{\boldsymbol h}^N\big\}_+$$
and
$$II = \big\{ 8(\Delta^N_{H_{h_1}(\cdot-t), K_{h_2}(\cdot - (t-a))})_T^2-\tfrac{1}{2}\mathsf V_{\boldsymbol h}^N\big\}_+.$$
We bound each term separately. First, we have
\begin{align*}
 \E\big[I\big] &  = \int_0^\infty \PP\big(8\mathcal W_{H_{h_1}(\cdot-t), K_{h_2}(\cdot - (t-a))}^N(\mathcal F)_T^2-\tfrac{1}{2}\mathsf V_{\boldsymbol h}^N\geq \kappa\big)d\kappa \\  
 %& \leq \int_0^\infty \PP\big(|\widehat g_h^N(t,a)-g_h(t,a)| \geq \tfrac{1}{2}(\mathsf V_h^N+\kappa)^{1/2}\big)d\kappa \\
  & = \int_0^\infty \PP\big(\mathcal W_{H_{h_1}(\cdot-t), K_{h_2}(\cdot - (t-a))}^N(\mathcal F)_T \geq \tfrac{1}{2\sqrt{2}}(\tfrac{1}{2}\mathsf V_{\boldsymbol h}^N+\kappa)^{1/2}\big)d\kappa  \lesssim N^{-3}
\end{align*}
applying Theorem \ref{thm concentration optimal} with $w_1 = H_{h_1}(\cdot-t)$ and $w_2 =  K_{h_2}(\cdot - (t-a))$ in the same way as Step 3) in the proof of Theorem \ref{thm oracle g}. As for $II$, we have
\begin{equation} \label{eq dec II}
 \E\big[II\big]    = \int_{\tfrac{1}{2}\mathsf V_{\boldsymbol h}^N}^\infty \PP\big(\big|\Delta^N_{H_{h_1}(\cdot-t), K_{h_2}(\cdot - (t-a))})_T\big| \geq  \tfrac{1}{2\sqrt{2}}\kappa^{1/2} \big)d\kappa 
 \end{equation}
and we plan to apply Lemma \ref{lem dev mg saut} with  $w_1 = H_{h_1}(\cdot-t)$ and $w_2 =  K_{h_2}(\cdot - (t-a))$. Setting $u = \tfrac{1}{8}\kappa$, the condition of Lemma 
\ref{lem dev mg saut} is fulfilled as soon as $\kappa > 8 \cdot 2^{-6} V_{H_{h_1}, K_{h_2}}^N (\log N)^{-2} =  \tfrac{1}{8} V_{H_{h_1}, K_{h_2}}^N (\log N)^{-2}$ which is the case here since the integral in \eqref{eq dec II} above is taken for $\kappa \geq \tfrac{1}{2}\mathsf V_{\boldsymbol h}^N = \tfrac{1}{2} V_{H_{h_1}, K_{h_2}}^N$. It follows that 
$$\E[II] \leq III + IV,$$
with
$$III = 2\int_{\tfrac{1}{2}\mathsf V_{\boldsymbol h}^N}^\infty \exp\Big(-\frac{\kappa^{1/2}}{4\sqrt{2}(\chi_{H_{h_1}, K_{h_2}}^N+\xi^N_{H_{h_1}, K_{h_2}})}\Big) d\kappa$$
and
$$IV =  2\int_{\tfrac{1}{2}\mathsf V_{\boldsymbol h}^N}^\infty \PP\big( N^{-1}|\mu|_\infty\mathcal W_{(H_{h_1})^2,(K_{h_2})^2}(\mathcal F)_T \geq \tfrac{1}{4\sqrt{2}}\xi_{H_{h_1},K_{h_2}}^N\kappa^{1/2}\big)d\kappa.$$
First, we write
$$III = 4(\chi_{H_{h_1}, K_{h_2}}^N+\xi^N_{H_{h_1}, K_{h_2}})^2\int_{v_N}^\infty \kappa e^{-\tfrac{\kappa}{4\sqrt{2}}}d\kappa$$
with
$$v_N = \tfrac{\sqrt{2}}{2} (\mathsf V_{\boldsymbol h}^N)^{1/2}(\chi_{H_{h_1}, K_{h_2}}^N+\xi^N_{H_{h_1}, K_{h_2}})^{-1}.$$
Note that 
$$(V_{H_{h_1},K_{h_2}}^N)^{1/2} =  h_1^{-1/2}h_2^{-1/2}N^{-1/2} (\log N) 4C^\star |H|_{1,\infty}|K|_{1,\infty}.$$

It follows that
\begin{align*}
& \chi_{H_{h_1}, K_{h_2}}^N+\xi^N_{H_{h_1}, K_{h_2}} \\
 = & \; N^{-1}|H_{h_1}|_\infty |K_{h_2}|_\infty|\mu|_\infty + 16 N^{-1}|\mu|_\infty|g|_\infty |H_{h_1}|^2_2|K_{h_2}|_2^2 (V_{|H_{h_1}|,|K_{h_2}|}^N)^{-1/2}(\log N)\\
 = & \; N^{-1}h_1^{-1}h_2^{-1}|\mu|_\infty |H|_\infty |K|_\infty +N^{-1/2}h_1^{-1/2}h_2^{-1/2}
 %{\color{red}\xout{(\log N)^{-1}}} 
 4C^\star |\mu|_\infty|g|_\infty \tfrac{|H|^2_2|K|_2^2}{|H|_{1,\infty}|K|_{1,\infty}}.
 % \\
% \leq & \; |\mu|_\infty |H|_\infty |K|_\infty +(\log N)^{-1} 4|\mu|_\infty|g|_\infty |H|_{1,\infty}|K|_{1,\infty} \\
% \leq & \; |\mu|_\infty |H|_\infty |K|_\infty +(\log N)^{-1} 4|\mu|_\infty \frac{|H|^2_2|K|_2^2}{|H|_{1,\infty}|K|_{1,\infty}}
% \leq & \; \tfrac{1}{2}  |\mu|_\infty |H|_\infty |K|_\infty 
\end{align*}
%say, for large enough $N$ {\color{red} (word about uniformity in $|\mu|_\infty$ and $|g|_\infty$)}
By definition of $\mathcal G_2^N$ we have $h_i \geq N^{-1/2}$ hence 
$$(\chi_{H_{h_1}, K_{h_2}}^N+\xi^N_{H_{h_1}, K_{h_2}} )^2 \leq  \big(|\mu|_\infty |H|_\infty |K|_\infty +4C^\star |\mu|_\infty \tfrac{|H|^2_2|K|_2^2}{|H|_{1,\infty}|K|_{1,\infty}}\big)^2$$
follows and the term in front of the integral in $III$ is bounded. Moreover, 
\begin{align*}
v_N & = \tfrac{\sqrt{2}}{2} \frac{ (\log N) 4C^\star |H|_{1,\infty}|K|_{1,\infty}}{N^{-1/2}h_1^{-1/2}h_2^{-1/2}|\mu|_\infty |H|_\infty |K|_\infty +4C^\star|\mu|_\infty|g|_\infty \tfrac{|H|^2_2|K|_2^2}{|H|_{1,\infty}|K|_{1,\infty}}} \\
& \geq \tfrac{\sqrt{2}}{2} \frac{C^\star |H|_{1,\infty}|K|_{1,\infty}}{|\mu|_\infty ( |H|_\infty |K|_\infty+ 4C^\star |g|_\infty \tfrac{|H|^2_2|K|_2^2}{|H|_{1,\infty}|K|_{1,\infty}})} (\log N)  = C^{(3)}\log N
\end{align*}
say.
%, for large enough $N$ {\color{red} (word about uniformity in $|\mu|_\infty$ and $|g|_\infty$)}. 
Since 
$$\int_{v_N}^\infty \kappa e^{-\tfrac{\kappa}{4\sqrt{2}}}d\kappa \lesssim v_Ne^{-\tfrac{1}{4\sqrt{2}}v_N} \lesssim (\log N) N^{-C^{(3)}}$$
it suffices to check that $C^{(3)} > 2$ in order to have that $III$ is smaller in order than $N^{-2}$ and thus asymptotically negligible. 
%{\color{blue} (possibly change def of $\xi^N$ and incorporate a dilating constant)}.
We finally bound the term $IV$. Applying Theorem \ref{thm concentration optimal} with $(w_1,w_2) = (H_{h_1}^2, K_{h_2}^2)$, by the change of variable
$$\tfrac{N}{4\sqrt{2}|\mu|_\infty}\xi_{H_{h_1}, K_{h_2}}^N\kappa^{1/2}=(1+u)C''|H_{h_1}^2|_{1,\infty}|K_{h_2}^2|_{1,\infty}N^{-1/2}$$ we obtain that $IV$ is of order
$$y_N\int_{z_N}^\infty (1+u)\min\big((\exp(u)-1)^{-1}, 1\big)du,$$
with
\begin{align*}
y_N   = \big((\xi_{H_{h_1},K_{h_2}}^N)^{-1}|H_{h_1}^2|_{1,\infty}|K_{h_2}^2|_{1,\infty}N^{-3/2}\big)^2\;\;\text{and}\;\;
%and
z_N  = \tfrac{N^{3/2}\xi_{H_{h_1},K_{h_2}}^N (\mathsf V_{\boldsymbol h}^N)^{1/2}}{8C''|H_{h_1}^2|_{1,\infty}|K_{h_2}^2|_{1,\infty}}-1.
\end{align*}
Straightforward computations show that $y_N \lesssim h_1^{-2}h_2^{-2}N^{-2} \lesssim 1$ by construction of $\mathcal G_2^N$. Finally
$$z_N =  \frac{8|g|_\infty|H|_2|K|_2}{C''|H|_\infty|K|_\infty}\frac{\log N}{h_1^{1/2}h_2^{1/2}N^{-1/2}}-1 \geq C^{(4)}\log N-1.$$
say. One can check that $C^{(4)} > 2$ and we can therefore conclude that $IV$ also has a negligible order.\\ 

\noindent {\bf Step 4)} The control of the second term in the right-hand side of \eqref{eq controle sup idem} is done in the same way as in Step 4) 
of the proof of Theorem \ref{thm oracle g} and only inflates the previous bound by a factor or order $\mathrm{Card}(\mathcal G_2^N) \lesssim N^2$.
In turn $\E\big[\mathsf A_{\boldsymbol h}^N(t,a)\big] \lesssim N^{-1}$ and we have established by Step 2) that for any $\boldsymbol h \in \mathcal G_2^N,$
\begin{equation} \label{eq oracle gamma}
\E\big[\big(\widehat \pi_{\boldsymbol h}^N(t,a)-\pi(t,a)\big)^2\big] \lesssim \mathcal B_{\boldsymbol h}^N(\pi)(t,a)^2+\mathsf V_{\boldsymbol h}^N + \delta_N
\end{equation}
holds true with $\delta_N \lesssim N^{-1}$. Putting together Step 1) and Theorem \ref{thm oracle g} completes the proof.

\subsection{Proof of Theorem \ref{thm LB}}

\subsubsection*{Preliminaries}

We let $(Z_t)_{0 \leq t \leq T}$ denote the canonical process on\footnote{remember that ${\mathcal M_F}_+$ denotes the set of positive finite measures on $\R_+$} $\mathbb D([0,T], {\mathcal M_F}_+)$  endowed with the weak topology and equipped with its Borel sigma-field. If $\Upsilon$ is a probability measure on ${\mathcal M_F}_+$ and if $b,\mu \in \mathcal L_{\mathcal D}^\infty $, we write  $\PP_{b,\mu,\Upsilon}^N$ for the (necessarily unique) probability measure on $\mathbb D([0,T], {\mathcal M_F}_+)$ under which $(Z_t)_{0 \leq t \leq T}$ is a weak solution to \eqref{eq micro} with $\mathcal L(Z_0) = \Upsilon$.
 
 \begin{prop} \label{prop girsanov}
 For $i=1,2$, let $b_i,\mu_i \in \mathcal L_{\mathcal D}^\infty $ such that $\mathrm{supp}(b_2) \subset \mathrm{supp}(b_1)$ and $\mathrm{supp}(\mu_2) \subset \mathrm{supp}(\mu_1)$. 
 For any initial condition $\mathcal L(Z_0) = \Upsilon$, we have
$$
 \|\PP_{b_1,\mu_1,\Upsilon}^N-\PP_{b_2,\mu_2,\Upsilon}^N\|_{TV} \lesssim N^{1/2}\big(\big|b_1^{-1}b_2-1\big|_2+\big|\mu_1^{-1}\mu_2-1\big|_2\big),$$
where $\|\cdot\|_{TV}$ denotes total variation distance,  up to an explicitly computable constant that only depends on $\mu_1$ and $b_1$.
 \end{prop}
\begin{proof}
The proof is classical, and we only sketch it.  Thanks to the Dol\'eans-Dade exponential for semimartingales (see {\it e.g.} \cite{JACODSHIRYAEV}
%, Chapter {\color{blue}(complete here!!)
or L\" ocherbach \cite{LOCHERBACH1,LOCHERBACH2} in the context of birth and death processes) and abbreviating $f\big(s,a_i(Z_s^-)\big)$ by $f^{i}(s)$, we have
\begin{align*} 
&\frac{d\PP_{b_2,\mu_2,\Upsilon}^N}{d\PP_{b_1,\mu_1,\Upsilon}^N} =  
N^{-1}\int_0^T  \int_{\R_+} \Big( b_2(s,a)-b_1(s_a)+\mu_2(s,a)-\mu_1(s,a)\Big) Z_s^N(da)ds \nonumber \\
&+ \int_0^T  \int_{\mathbb N\setminus \{0\} \times \mathbb R_+}  {\bf 1}_{\{i \leq \langle NZ_s^N,{\bf 1}\rangle\}}\Big({\bf 1}_{\{0 \leq \vartheta \leq b_1^i(s)\}}\log \frac{b_2^i(s)}{b_1^i(s)}+{\bf 1}_{\{b_1^i(s) \leq \vartheta \leq \mu_1^i(s)\}}\log \frac{\mu_2^i(s)}{\mu_1^i(s)} \Big)\mathcal Q_1(ds,di,d\vartheta),
% \label{eq micro}
%\;\;\text{for every}\;\;t \in [0,T], 
\end{align*}
where $\mathcal Q_1$ is a Poisson random measures on $\R_+\times \mathbb N \setminus \{0\} \times \R_+$ with intensity $ds \big(\sum_{k \geq 1}\delta_k(di)\big) d\vartheta$ under $\PP_{b_1,\mu_1,\Upsilon}^N$. By Pinsker's inequality, it follows that 
\begin{align*}
& \|\PP_{b_1,\mu_1,\Upsilon}^N-\PP_{b_2,\mu_2,\Upsilon}^N\|_{TV}^2 \\
 \leq &\; \tfrac{1}{2}\E_{\PP_{b_1,\mu_1,\Upsilon}^N}\Big[\log \frac{d\PP_{b_1,\mu_1,\Upsilon}^N}{d\PP_{b_2,\mu_2,\Upsilon}^N}\Big] \\
 = &\; -\tfrac{1}{2}\E_{\PP_{b_1,\mu_1,\Upsilon}^N}\Big[\log \frac{d\PP_{b_2,\mu_2,\Upsilon}^N}{d\PP_{b_1,\mu_1,\Upsilon}^N}\Big] \\
 = &\;   \tfrac{N}{2} \E_{\PP_{b_1,\mu_1,\Upsilon}^N}\Big[ \int_0^T \int_{\R_+} \big(b_2-b_1+\mu_2-\mu_1-b_1\log \frac{b_2}{b_1}-\mu_1\log \frac{\mu_2}{\mu_1}\big)(s,a)Z_s^N(da)\Big] \\
  = & \;  \tfrac{N}{2} \E_{\PP_{b_1,\mu_1,\Upsilon}^N}\Big[ \int_0^T \int_{\R_+} \big(b_1\varphi(b_1^{-1}b_2-1)+\mu_1\varphi(\mu_1^{-1}\mu_2-1)(s,a)Z_s^N(da)\Big],
\end{align*}
with $\varphi(x) = x- \log(1+x) \leq x^2$ for $x \geq 0$. Therefore
\begin{align*}
\|\PP_{b_1,\mu_1,\Upsilon}^N-\PP_{b_2,\mu_2,\Upsilon}^N\|_{TV}^2 & \lesssim  \tfrac{N}{2}  \E_{\PP_{b_1,\mu_1,\Upsilon}^N}\Big[ \int_0^T \int_{\R_+} \big((b_1^{-1}b_2-1)^2+(\mu_1^{-1}\mu_2-1)^2\big)(s,a)Z_s^N(da)\Big] \\
& \lesssim N(|b_1^{-1}b_2-1|_2^2+|\mu_1^{-1}\mu_2-1|_2^2)
\end{align*}
and Proposition \ref{prop girsanov} is proved. 
%{\color{blue} (last estimation should be precised, and potentially improved, using $Z_s^N(da)ds \approx g_{b_2,\mu_2}(s,a)dsda$ in order to prove \eqref{eq LB g}!!)}
\end{proof}

\subsubsection*{Representation of $g$ in terms of $(g_0,b,\mu)$}
We need some notation. Let
$$L_{b,\mu}(t,a) = b(t,a)\exp\big(-\int_{t-a}^t \mu(s,s-t+a)ds\big)\;\;\text{for}\;\;(t,a) \in \mathcal D_L,$$
$$M_{b,\mu,g_0}(t) = \int_0^\infty b(t,t+u)g_0(u)\exp\big(-\int_0^t \mu(s,u+s)ds\big)du,\;\;\text{for}\;\;t \in [0,T],$$
and define $B_{b,\mu,g_0}:[0,T] \rightarrow \R_+$ as the solution to the integral equation
\begin{equation} \label{eq integral equation}
B_{b,\mu,g_0}(t)  = M_{b,\mu,g_0}(t)+ \int_0^t B_{b,\mu,g_0}(a)L_{b,\mu}(t,t-a)da\;\;\text{for every}\;\;t\in [0,T].
\end{equation}
Note that Assumptions \ref{H basic} and \ref{H smoothness basic} ensure the existence and uniqueness of \eqref{eq integral equation}.
Define next 
\begin{equation} 
\label{eq def weak}
g(t,a) = 
\left\{
\begin{array}{lll}
g_0(a-t)\exp\big(-\int_0^t \mu(s,a-t+s)ds\big) & \text{on} & \mathcal D_U \\ \\
B_{b,\mu,g_0}(t-a)\exp\big(-\int_{t-a}^t \mu(s,a+s-t)ds\big) & \text{on} & \mathcal D_L 
\end{array}
\right.
\end{equation}
and set for instance $g(t,a) = 0$ on $\{a=t\}$. One can check that $g$ defined in \eqref{eq def weak} is a weak solution to the McKendricks Von Voester equation \eqref{McKendrick}.

\subsubsection*{Completion of proof of Theorem \ref{thm LB}}

We follow a classical two-point lower bound argument using Le Cam's lemma: if $\mathbb P_i$, $i=1,2$ are two probability measures defined on the same probability space and $\Psi(\PP_i) \in \R$ is a functional of $\mathbb P_i$, we have
\begin{equation} \label{eq LeCam}
\inf_{F}\max_{i = 1,2}\E_{{\mathbb P}_i}\big[|F-\Psi(\PP_i)|\big] \geq \tfrac{1}{2}|\Psi(\PP_1)-\Psi(\PP_2)|(1-\|\mathbb P_1-\mathbb P_2\|_{TV}),
\end{equation}
where the infimum is taken over all estimators of $\Psi(\PP_i)$, see {\it e.g.} \cite{LECAM} among many other references.\\

\noindent {\bf Step 1)} To prove \eqref{eq LB g}, we pick
$$g_0 \in \mathcal H_L^\nu(a),\;\;b_0 \in \mathcal H^{\alpha,\beta}_L(t,a),\;\;\mu_1 \in \mathcal H^{\gamma,\delta}_L(t,a) \cap \mathcal L_{\mathcal D, \epsilon}^\infty$$ 
arbitrarily, together with a sequence $\Upsilon^N$ such that $N\langle Z_0^N, {\bf 1}\rangle \lesssim 1$ almost-surely under $\Upsilon^N$ and $\Upsilon^N(da) \rightarrow g_0(da)$ weakly as $N \rightarrow \infty$. Next, define 
$$\mu_2^N(s,u) = \mu_1(s,u)\big(1+\psi_{t-a}^N(s,u)\big),$$ where
$$\psi_{t-a}^N(s,u) = cN^{-1/2}\tau_N^{1/2}\psi\big(\tau_N(s-u-(t-a))\big),$$
with
$\tau_N = N^{1/(2s_{\mathrm{death}}^-+1)} = N^{1/(2\max(\gamma,\delta)+1)}$
and an infinitely many times differentiable nonnegative function $\psi$ with compact support that satisfies $\psi(0)=1$, $|\psi|_2^2=1$. Finally, pick $c>0$ small enough 
so that the property
$$\mu_2^N\in \mathcal H^{\gamma,\delta}_L(t,a) \cap \mathcal L_{\mathcal D, \epsilon}^\infty$$
holds, uniformly in $N$. This is possible since 
$$|\psi_{t-a}^N(\cdot,t-a)|_{{\mathcal H}^{\gamma}(t)} = cN^{-1/2}\tau_N^{1/2+\gamma}|\psi|_{\mathcal H^\gamma(t)} \leq c|\psi|_{\mathcal H^\gamma(t)} \leq c|\psi|_{\mathcal H^\gamma(t,a)}$$
and 
$$|\psi_{t-a}^N(t,\cdot)|_{{\mathcal H}^{\delta}(t-a)} = cN^{-1/2}\tau_N^{1/2+\delta}|\psi|_{\mathcal H^\delta(t-a)}\leq c|\psi|_{\mathcal H^\delta(t-a)} \leq c|\psi|_{\mathcal H^\delta(a)}.$$
By Proposition \ref{prop girsanov}, we have
%\begin{align*}
\begin{equation} \label{eq OK girsanov mu first}
 \|\PP_{b_0,\mu_1,\Upsilon^N}-\PP_{b_0,\mu_2^N,\Upsilon^N}\|_{TV} \lesssim N^{1/2}\big|\mu_1^{-1}\mu_2^N-1\big|_2 
  = N^{1/2}|\psi_{t-a}^N|_2 = c^{1/2} |\psi|_2^2 \leq \tfrac{1}{2}
\end{equation}
say, for large enough $N$ and sufficiently small $c$.\\
% \end{align*}
%say, by picking $c$ sufficiently small.

\noindent {\bf Step 2)} Let $(t,a) \in \mathcal D_U$. We let 
$$\Psi(\PP_{b,\mu,\Upsilon}^N) = g(t,a) = \widetilde g(t',a') = g_0(-a')\exp\big(-\int_0^{t'} \mu(s,s-a')ds\big)$$
by \eqref{eq def weak} above, with $(t',a')=(t,t-a)=\varphi(t,a)$.
% = 
%\left\{
%\begin{array}{lll}
%g_0(-a')\exp\big(-\int_0^{t'} \mu(s,s-a')ds\big) & \text{for} & (t',a') \in \widetilde{\mathcal D}_U \\ \\
%B_{b,\mu,g_0}(a')\exp\big(-\int_{a'}^{t'} \mu(s,s-a')ds\big) & \text{for} & (t',a') \in \widetilde{\mathcal D}_L, 
%\end{array}
%\right.
%$$
%as follows from \eqref{eq def weak bis} above,
%in the proof of Proposition  \ref{P first smoothness} (see Section \ref{proof of prop P first smoothness} below, with  $B_{b,\mu,g_0}$ defined in \eqref{eq integral equation}) and 
%with $(t',a')=(t,t-a)=\varphi(t,a)$, $\widetilde{\mathcal D}_U = \varphi(\mathcal D_U)$, $\widetilde{\mathcal D}_L = \varphi(\mathcal D_L)$. 
%In particular, for $(t,a) \in \mathcal D_U$, we have
It follows that
\begin{align}
& \big|\Psi(\PP_{b_0,\mu_2^N,\Upsilon}^N)-\Psi(\PP_{b_0,\mu_1,\Upsilon}^N)\big| \nonumber \\
%= &\; g_0(-a')\big|\exp\big(-\int_0^{t'}\mu_2^N(s,s-a')ds\big)-\exp\big(-\int_0^{t'}\mu_1(s,s-a')ds\big)\big| \\
= &\; g_0(-a')\exp\big(-\int_0^{t'}\mu_1(s,s-a')ds\big)\big|\exp\big(-\int_0^{t'}\psi_{a'}^N(s,s-a')ds\big)-1\big|  \nonumber\\
\geq & \; g_0(-a')\exp\big(-(|\mu_1|_\infty+|\psi_{a'}^N|_\infty) t'\big)\int_0^{t'}\psi_{a'}^N(s,s-a')ds  \nonumber \\
\geq & \; \tfrac{1}{2}g_0(-a')\exp(-|\mu_1|_\infty t')cN^{-1/2}\tau_N^{1/2}\psi(0)t' \nonumber\\
\gtrsim & \;  N^{-s_{\mathrm{dens}}^-/(2s_{\mathrm{dens}}^-+1)}  \label{eq minoration risque diff}
\end{align}
using $|e^{-x}-1| \geq xe^{-x}$ for $x \geq 0$ and the fact that $e^{-t'|\psi_{a'}^N|_\infty} \geq \tfrac{1}{2}$ say, for sufficiently large $N$.\\

\noindent {\bf Step 3)} Let $(t,a) \in \mathcal D_L$. We now have
$$\Psi(\PP_{b,\mu,\Upsilon}^N) = g(t,a) = \widetilde g(t',a') = B_{b,\mu,g_0}(a')\exp\big(-\int_{a'}^{t'} \mu(s,s-a')ds\big),$$
by \eqref{eq def weak} and where $B_{b,\mu,g_0}$ is defined in \eqref{eq integral equation}. It follows that
\begin{align*}
&\; \big|\Psi(\PP_{b_0,\mu_2^N,\Upsilon}^N)-\Psi(\PP_{b_0,\mu_1,\Upsilon}^N)\big| \nonumber \\
= &\; \big|B_{b_0,\mu_2^N,g_0}(a')\exp\big(-\int_{a'}^{t'} \mu_2^N(s,s-a')ds\big)-B_{b_0,\mu_1,g_0}(a')\exp\big(-\int_{a'}^{t'} \mu_1(s,s-a')ds\big)\big| \\
\geq &\; \big|I\big|-\big|II\big|,
\end{align*}
%%= &\; g_0(-a')\big|\exp\big(-\int_0^{t'}\mu_2^N(s,s-a')ds\big)-\exp\big(-\int_0^{t'}\mu_1(s,s-a')ds\big)\big| \\
%= &\; g_0(-a')\exp\big(-\int_0^{t'}\mu_1(s,s-a')ds\big)\big|\exp\big(-\int_0^{t'}\psi_{a'}^N(s,s-a')ds\big)-1\big|  \nonumber\\
%\geq & \; g_0(-a')\exp\big(-(|\mu_1|_\infty+|\psi_{a'}^N|_\infty) t'\big)\int_0^{t'}\psi_{a'}^N(s,s-a')ds  \nonumber \\
%\geq & \; g_0(-a')\exp\big(-(|\mu_1|_\infty+c|\psi|_\infty) t'\big)cN^{-1/2}\tau_N^{1/2}\psi(0)t' \nonumber\\
%\gtrsim & \;  N^{-s_{\mathrm{death}}^-/(2s_{\mathrm{death}}^-+1)}  \label{eq minoration risque diff bis}
%\end{align}
with
$$I = B_{b_0,\mu_1,g_0}(a')\big(\exp\big(-\int_{a'}^{t'} \mu_2^N(s,s-a')ds\big)-\exp\big(-\int_{a'}^{t'} \mu_1(s,s-a')ds\big)$$
and
$$II = \big(B_{b_0,\mu_2^N,g_0}(a')-B_{b_0,\mu_1,g_0}(a')\big)\exp\big(-\int_{a'}^{t'} \mu_2^N(s,s-a')ds\big).$$
To bound $I$ from below, we proceed as in Step 2). For simplicity, we assume moreover here that $b_0(t,a) = b_0$ is constant. We have $B_{b_0,\mu_1,g_0}(a') \geq M_{b_0,\mu_1,g_0}(t) \geq b_0|g_0|_1e^{-|\mu_1|_\infty a'}$ and in the same way as for \eqref{eq minoration risque diff} one can check that
\begin{align*}
& \; \big|\exp\big(-\int_{a'}^{t'} \mu_2^N(s,s-a')ds\big)-\exp\big(-\int_{a'}^{t'} \mu_1(s,s-a')ds\big| \\
 \geq & \;\tfrac{1}{2}e^{-|\mu_1|_\infty(t'-a')}(t'-a')N^{-s_{\mathrm{dens}}^-/(2s_{\mathrm{dens}}^-+1)} 
 \end{align*}
for large enough $N$ hence
\begin{equation} \label{eq mino I}
\big|I\big| \geq \tfrac{1}{2}b_0|g_0|_1e^{-|\mu_1|_\infty t'}(t'-a')cN^{-s_{\mathrm{dens}}^-/(2s_{\mathrm{dens}}^-+1)}. 
\end{equation}
In order to bound $II$ from above, we use the following technical facts that are checked in the same way as before: for every $(t,a) \in \mathcal D_L$, we have

\begin{align} 
%\big|M_{b_0,\mu_2^N,g_0}(t)-M_{b_0,\mu_1,g_0}(t)\big| \leq |\mu_1|_\infty c|\psi|_\infty t N^{-s_{\mathrm{dens}}^-/(2s_{\mathrm{dens}}^-+1)},
\big|M_{b_0,\mu_2^N,g_0}(a')-M_{b_0,\mu_1,g_0}(a')\big| & \leq b_0|\mu_1|_\infty T\int_0^\infty g_0(u)\psi_{(t,a)}^N(-u)du \nonumber\\
& \leq   b_0 |\mu_1|_\infty T|g_0|_\infty |\psi_{t-a}^N|_1\nonumber \\
& =   b_0|\mu_1|_\infty T|g_0|_\infty cN^{-1/2}\tau_N^{-1/2}|\psi|_1 \ll N^{-s_{\mathrm{dens}}^-/(2s_{\mathrm{dens}}^-+1)}
\label{eq controle M}
\end{align}
and
\begin{equation} \label{eq controle L}
\big|L_{b_0,\mu_2^N,g_0}(a',a'-a)-L_{b_0,\mu_1,g_0}(a',a'-a)\big| \leq b_0 |\mu_1|_\infty cN^{-1/2}\tau_N^{1/2}\psi\big(\tau_N(u-a')\big)(a'-u)
%2b_0 c|\psi|_\infty a N^{-s_{\mathrm{dens}}^-/(2s_{\mathrm{dens}}^-+1)}
%N^{-s_{\mathrm{death}}^-/(2s_{\mathrm{death}}^-+1)
\end{equation}
and since 
$B_{b_0,\mu,g_0}(t) \leq b_0 |g_0|_1+b_0 \int_0^t B_{b,\mu,g_0}(s)ds$
for every $t \in [0,T]$, we infer
\begin{equation} \label{eq control B basic}
B_{b_0\mu,g_0}(t) \leq b_0 |g_0|_1e^{b_0 T}
\end{equation}
by Gr\"onwall lemma. It follows that
\begin{align*}
& B_{b_0,\mu_2^N,g_0}(a')-B_{b_0,\mu_1,g_0}(a')  = M_{b_0,\mu_2^N,g_0}(a')-M_{b_0,\mu_1,g_0}(a') \\
&+\int_0^{a'}B_{b_0,\mu_1,g_0}(a)  \big(L_{b_0,\mu_2^N,g_0}(a', a'-a)-L_{b_0,\mu_1,g_0}(a', a'-a)\big)da\\
& + \int_0^{a'}L_{b_0,\mu_2^N,g_0}(a',a'-a) \big(B_{b_0,\mu_2^N,g_0}(a)-B_{b_0,\mu_1,g_0}(a)\big)da. 
\end{align*}
Taking absolute values and using \eqref{eq controle M}, \eqref{eq controle L} and \eqref{eq control B basic}, we derive
\begin{align*}
\big|B_{b_0,\mu_2^N,g_0}(a')-B_{b_0,\mu_1,g_0}(a') \big| 
& \leq   b_0 |\mu_1|_\infty T|g_0|_\infty cN^{-1/2}\tau_N^{-1/2}|\psi|_1 \\
&+b_0^2 |g_0|_1e^{b_0 T}|\mu_1|_\infty  cN^{-1/2}\tau_N^{1/2}\int_0^{a'}\psi\big(\tau_N(u-a')\big)(a'-u)du \\
%& \leq   \big(|\mu_1|_\infty a' +  2b_0  (a')^2 b_0|g_0|_1e^{b_0 T} \big)c|\psi|_\infty N^{-s_{\mathrm{dens}}^-/(2s_{\mathrm{dens}}^-+1)} \\
& + b_0\int_0^{a'}\big|B_{b_0,\mu_2^N,g_0}(s)-B_{b_0,\mu_1,g_0}(s)\big|ds. 
\end{align*}
Using $\tau_N^{1/2}\int_0^{a'}\psi\big(\tau_N(u-a')\big)(a'-u)du \leq \tau_N^{-1/2}T|\psi|_1$,  we derive
%\begin{align*}
$$
\big|B_{b_0,\mu_2^N,g_0}(a')-B_{b_0,\mu_1,g_0}(a') \big|  
\leq   b_0 |\mu_1|_\infty T|g_0|_\infty cN^{-1/2}\tau_N^{-1/2}|\psi|_1(1+b_0e^{b_0 T})e^{b_0a'} 
%\ll N^{-s_{\mathrm{dens}}^-/(2s_{\mathrm{dens}}^-+1)} 
%+ b_0^2 |g_0|_1e^{b_0 T}|\mu_1|_\infty  c\tau_N^{-1/2}T|\psi|_1
%\big(|\mu_1|_\infty c|\psi|_\infty a' +  2b_0 c|\psi|_\infty (a')^2 b_0|g_0|_1e^{b_0 T} \big)e^{b_0 a'}
%N^{-s_{\mathrm{dens}}^-/(2s_{\mathrm{dens}}^-+1)} 
%\end{equation}
$$
 by Gr\"onwall lemma again. We conclude 
\begin{equation} \label{eq controle mino II}
\big|II\big| 
%\leq 
%\big(|\mu_1|_\infty a' +  2b_0  (a')^2 b_0|g_0|_1e^{b_0 T} \big)c|\psi|_\infty e^{b_0 a'}
%N^{-s_{\mathrm{death}}^-/(2s_{\mathrm{death}}^-+1)}. 
%(|\mu_1|_\infty c|\psi|_\infty a' +  2b_0 c|\psi|_\infty (a')^2 b_0|g_0|_1e^{b_0 T})e^{b_0 a'}
%N^{-s_{\mathrm{dens}}^-/(2s_{\mathrm{dens}}^-+1)} 
\ll N^{-s_{\mathrm{dens}}^-/(2s_{\mathrm{dens}}^-+1)}. 
\end{equation}
Comparing \eqref{eq mino I} and \eqref{eq controle mino II}, we see that 
\begin{equation} \label{eq controle inf D lower}
\big|\Psi(\PP_{b_0,\mu_2^N,\Upsilon}^N)-\Psi(\PP_{b_0,\mu_1,\Upsilon}^N)\big|  \gtrsim N^{-s_{\mathrm{dens}}^-/(2s_{\mathrm{dens}}^-+1)}. 
\end{equation}
%provided
%\begin{equation} \label{eq mino nec}
%\tfrac{1}{2}cb_0|g_0|_1e^{-|\mu_1|_\infty t'}(t'-a') > (|\mu_1|_\infty c|\psi|_\infty a' +  2b_0 c|\psi|_\infty (a')^2 b_0|g_0|_1e^{b_0 T})e^{b_0 a'}.
%\end{equation}
%We formally see that in the limit $|\mu_1|_\infty \rightarrow 0$ and $b_0 \rightarrow 0$, the condition \eqref{eq mino nec} is satisfied if
%$$\tfrac{1}{2}t'-a' > 2b_0(a')^2|\psi |_\infty$$
%hence for small enough $b_0$. Therefore, for sufficiently small $\epsilon >0$, we can find $\mu_1 \in \mathcal H^{\gamma,\delta}_L(t,a) \cap \mathcal L_{\mathcal D, \epsilon}^\infty$ with a sufficiently small $|\mu|_\infty$ and  $b_0 >0$ so that \eqref{eq mino nec} holds for large enough $N$. We fix such a pair $(b_0,\mu_1)$ for the rest of the proof.\\

\noindent {\bf Step 4)}. Combining  \eqref{eq minoration risque diff} or \eqref{eq controle inf D lower} with \eqref{eq LeCam} and  \eqref{eq OK girsanov mu first}, we successively obtain
\begin{align*}
\sup_{b,\mu,g_0}\E_{\PP_{b,\mu,\Upsilon^N}^N}\big[|F-g(t,a)|\big] & \geq \tfrac{1}{2} \max_{i = 1,2}\E_{\PP_{b_0,\mu_i,\Upsilon^N}^N}\big[|F-\Psi(\PP_{b_0,\mu_i,\Upsilon^N}^N)| \\
& \geq \tfrac{1}{4}|\Psi(\PP_{b_0,\mu_1,\Upsilon^N}^N)-\Psi(\PP_{b_0,\mu_2^N,\Upsilon^N}^N)|(1-\|\PP_{b_0,\mu_1,\Upsilon^N}^N-\PP_{b_0,\mu_2^N,\Upsilon^N}^N\|_{TV}) \\
%& = \tfrac{1}{4}|\mu_2^N(t,a)-\mu_1(t,a)|(1-\|\PP_{b_0,\mu_1,\Upsilon^N}^N-\PP_{b_0,\mu_2\Upsilon^N}^N\|_{TV}) \\
&  \gtrsim N^{-s_{\mathrm{death}}^-/(2s_{\mathrm{death}}^-+1)}
\end{align*}
and \eqref{eq LB g} follows.\\

\noindent {\bf Step 5)} To prove \eqref{eq LB mu}, we proceed as in Step 1), considering now the perturbation
$$\mu_2^N(s,u) = \mu_1(s,u)\big(1+\psi_{t,a}^N(s,u)\big),$$
with
$$\psi_{t,a}^N(u) = cN^{-1/2}\tau_N^{1/2}\psi\big(\tau_N(s-t)\big)(\widetilde \tau_N)^{1/2}\psi\big(\widetilde \tau_N(u-a)\big)$$
and $\tau_N^\delta = (\widetilde \tau_N)^\gamma = N^{s(\gamma,\delta)/(2s(\gamma,\delta)+1)}$
and an infinitely many times differentiable function $\psi$ with compact support that satisfies $\psi(0)=1$, $|\psi|_2^2=1$. Finally, we pick $c>0$ small enough 
%constant that accomodate the property 
so that the property
$$\mu_2^N\in \mathcal H^{\gamma,\delta}_L(t,a) \cap \mathcal L_{\mathcal D, \epsilon}^\infty$$
holds, uniformly in $N$. This is possible since
$$|\psi_{(t,a)}^N|_{{\mathcal H}^{\gamma}(t)} \leq c N^{-1/2}\tau_N^{1/2+\gamma}(\widetilde \tau_N)^{1/2}|\psi|_{\mathcal H^\gamma(t)}| \lesssim c$$
and
$$|\psi_{(t,a)}^N|_{{\mathcal H}^{\delta}(a)} \leq cN^{-1/2}\tau_N^{1/2}(\widetilde \tau_N)^{1/2+\delta}|\psi|_{\mathcal H^\delta(a)}| \lesssim c$$
likewise. Finally, we note that 
\begin{equation} \label{eq calcul LB}
%\begin{align}
\big|\mu_2^N(t,a)-\mu_1(t,a)\big|  \geq  |\mu_1(t,a)\psi_{t,a}^N(t,a)| 
 \geq  \epsilon cN^{-1/2}\tau_N^{1/2}(\widetilde \tau_N)^{1/2}  
 \gtrsim N^{-s_{\mathrm{death}}^-/(2s_{\mathrm{death}}^-+1)}
%\end{align}
\end{equation}
and
\begin{equation} \label{eq OK girsanov mu}
 \|\PP_{b_0,\mu_1,\Upsilon^N}^N-\PP_{b_0,\mu_2^N,\Upsilon^N}^N\|_{TV} \lesssim N^{1/2}\big|\mu_1^{-1}\mu_2^N-1\big|_2 
  = N^{1/2}|\psi_{(t,a)}^N|_2 = c^{1/2} \leq \tfrac{1}{2}
\end{equation}
say, for sufficiently small $c>0$, by Proposition \ref{prop girsanov}, which conditions are satisfied since $\mu_1$ and $\mu_2^N$ are bounded below. The end of the proof is similar to that of Step 4) with  $\Psi(\PP_{b,\mu,\Upsilon^N}^N) = \mu(t,a)$ together with the bounds \eqref{eq calcul LB} and \eqref{eq OK girsanov mu}. Therefore \eqref{eq LB mu} is proved and Theorem \ref{thm LB} folllows.
 
%\noindent {\bf Step 6)}
%Finally, we let $\Psi(\PP_{b,\mu,\Upsilon^N}^N) = \mu(t,a)$ and $F$ be an arbitrary estimator of $\mu(t,a)$. We have
%\begin{align*}
%\sup_{b,\mu,g_0}\E_{\PP_{b,\mu,\Upsilon^N}^N}\big[|F-\mu(t,a)|\big] & \geq \tfrac{1}{2} \max_{i = 1,2}\E_{\PP_{b_0,\mu_i,\Upsilon^N}^N}\big[|F-\Psi(\PP_{b_0,\mu_i,\Upsilon^N}^N)| \\
%& \geq \tfrac{1}{4}|\Psi(\PP_{b_0,\mu_1,\Upsilon^N}^N)-\Psi(\PP_{b_0,\mu_i,\Upsilon^N}^N)|(1-\|\PP_{b_0,\mu_1,\Upsilon^N}^N-\PP_{b_0,\mu_2\Upsilon^N}^N\|_{TV}) \\
%& = \tfrac{1}{4}|\mu_2^N(t,a)-\mu_1(t,a)|(1-\|\PP_{b_0,\mu_1,\Upsilon^N}^N-\PP_{b_0,\mu_2\Upsilon^N}^N\|_{TV}) \\
%&  \gtrsim N^{-s_{\mathrm{death}}^-/(2s_{\mathrm{death}}^-+1)}
%\end{align*}
%by \eqref{eq calcul LB} and \eqref{eq OK girsanov mu} and \eqref{eq LB mu} follows likewise. The proof of Theorem \ref{thm LB} is complete.

\subsection{Proof of Theorem \ref{adapt esti g}} By (i) of Proposition \ref{P first smoothness} the smoothness assumptions on $(b,\mu,g)$ imply
\begin{equation} \label{eq loc smoothness g lower}
u \mapsto g(t,u) \in \mathcal H^{\min(\alpha,\beta,\gamma+1,\delta)}_{L'}(a)\;\;\text{for}\;\;(t,a)\in \mathcal D_L,
\end{equation}
and
\begin{equation} \label{eq loc smoothness g upper}
u \mapsto g(t,u) \in \mathcal H^{\max(\gamma\wedge(\delta+1),\delta)}_{L'}(a)\;\;\text{for}\;\;(t,a)\in \mathcal D_U
\end{equation}
for some $L'$ that depends on $L$ and the smoothness parameters only.
For any $h \in \mathcal G_1^N$, by standard kernel approximation, see {\it e.g.} \cite{TSYBAKOV}  the smoothness properties \eqref{eq loc smoothness g lower} and \eqref{eq loc smoothness g upper} together with the definition \eqref{eq def dens} of $s_{\mathrm{dens}}^+$ imply 
$$\big|\int_0^\infty K_h(u-a)g(t,u)du-g(t,a)\big| \lesssim h^{s_{\mathrm{dens}}^+\wedge \ell_0},$$
up to a constant that depends on $K$, $s_\mathrm{dens}^+$ and $L'$ only. It follows that
$$\mathcal B_h^N(g)(t,a)^2 \lesssim h^{2s_{\mathrm{dens}}^+\wedge \ell_0}.$$
We also have
$$\mathsf{V}_h^N \lesssim (\log N)^2N^{-1}h^{-1}$$
up to a constant that depends on $C''$ of Theorem \ref{thm concentration optimal} and $K$. By Theorem \ref{thm oracle g}, we conclude
\begin{align*}
\E\big[\big(\widehat g_\star^N(t,a)-g(t,a)\big)^2\big] &  \lesssim \min_{h \in \mathcal G_1^N}\big( h^{2s_{\mathrm{dens}}^+\wedge \ell_0}+ (\log N)^2 N^{-1} h^{-1}\big)+\delta_N \\
& \lesssim \Big(\frac{(\log N)^2}{N}\Big)^{2s_{\mathrm{dens}}^+\wedge \ell_0/(2s_{\mathrm{dens}}\wedge \ell_0+1)} 
\end{align*}
using the definition of $\mathcal G_1^N$. Moreover, this estimate is uniform in $(b,\mu,g_0)$. The proof of Theorem \ref{adapt esti g} is complete.

\subsection{Proof of Theorem \ref{th adapt minimax mu}} Define $\widetilde \mu$ via $\mu= \widetilde \mu \circ \varphi$ and set $\widetilde \pi = \widetilde \mu \,\widetilde g$.\\ 

\noindent {\bf Step 1)} Write $\mu(t,a) = \widetilde \mu(t',a') = \mu(t', t'-a')$ with $(t',a') = \varphi(t,a) = (t,t-a)$.
% or equivalently $(t,a) = \varphi^{-1}(t',a')=(t',t'-a')$. 
The property $\mu \in \mathcal H^{\gamma,\delta}_L(t,a)$ for every $(t,a) \in \mathcal D$ implies $\widetilde \mu \in \mathcal H^{\min(\gamma,\delta), \delta}_{L'}(t',a')$ for every $(t',a') \in \varphi(\mathcal D) = \mathcal D$, for some other constant $L'$ that depends on $L$.  By (ii) of Proposition \ref{P first smoothness} it follows that
$$\widetilde \pi \in \mathcal H^{\min(\gamma,\delta), \min(\alpha,\beta,\gamma+1,\delta)}_{L'}(t,a)\;\;\text{for}\;\;(t,a) \in \widetilde{\mathcal D}_L$$
and
$$
\widetilde \pi \in \mathcal H^{\min(\gamma,\delta), \delta}_{L'}(t,a)\;\;\text{for}\;\;(t,a) \in \widetilde{\mathcal D}_U.
$$
Let $(t,a) \in \mathcal D_L$ so that $\varphi(t,a) \in \widetilde{\mathcal D}_L$. By standard kernel approximation again, 
%together with
%\eqref{eq smoothness transfer} 
%£the property $\big((H \otimes K)_{\boldsymbol h} \circ \varphi\big) \star f (t,a) = (H \otimes K)_{\boldsymbol h} \star \widetilde f \big(\varphi(t,a)\big)$, 
we infer
\begin{align*}
& \; \big|\big((H \otimes K)_{\boldsymbol h} \circ \varphi\big)\star \pi(t,a)-\pi(t,a)\big| \\
 = &\; \big|(H \otimes K)_{\boldsymbol h} \big)\star \widetilde \pi\big(\varphi(t,a)\big)-\widetilde \pi\big(\varphi(t,a)\big)\big| \\
 = &\; \Big|\int_0^T \int_{0}^{\infty}H_{h_1}(\varphi_1(t,a)-s)K_{h_2}(\varphi_2(t,a)-u)\widetilde \pi(s,u)dsdu-\widetilde \pi\big(\varphi(t,a)\big)\Big| \\
 \lesssim & \; h_1^{\min(\gamma,\delta) \wedge \ell_0}+h_2^{\min(\alpha,\beta,\gamma+1,\delta) \wedge \ell_0}
\end{align*}
up to a constant that depends on $H,K$, $L'$ and the smoothness parameters only and where we have set $\varphi(t,a) = \big(\varphi_1(t,a),\varphi_2(t,a)\big)$. Similarly, if $(t,a) \in \mathcal D_U$, we have
$$
\big|\big((H \otimes K)_{\boldsymbol h} \circ \varphi\big)\star \pi(t,a)-\pi(t,a)\big| \lesssim h_1^{\min(\gamma,\delta) \wedge \ell_0}+h_2^{\delta \wedge \ell_0}.
$$
It follows that
\begin{equation} \label{eq biais gamma}
\mathcal B_{\boldsymbol h}^N(\pi)(t,a)^2 \lesssim 
\left\{
\begin{array}{lll}
h_1^{2\min(\gamma,\delta) \wedge \ell_0}+h_2^{2\min(\alpha,\beta,\gamma+1,\delta) \wedge \ell_0} & \text{if} & (t,a) \in {\mathcal D_L} \\
h_1^{2\min(\gamma,\delta) \wedge \ell_0}+h_2^{2\delta \wedge \ell_0} & \text{if} & (t,a) \in {\mathcal D_U}.
\end{array}
\right.
\end{equation}
We also have
\begin{equation} \label{eq variance gamma}
\mathsf{V}_{\boldsymbol h}^N \lesssim (\log N)^2N^{-1}h_1^{-1}h_2^{-1}
\end{equation}
up to a constant that depends on $C''$ of Theorem \ref{thm concentration optimal} and $H, K$.\\

\noindent {\bf Step 2)} By Theorems \ref{thm oracle mu} and \ref{adapt esti g}, we have
\begin{equation} \label{eq recap}
\E\big[\big(\mu_{\star}^N(t,a)_{\varpi}-\mu(t,a)\big)^2\big] \lesssim   \Big(\frac{(\log N)^2}{N}\Big)^{2s_{\mathrm{dens}}^+\wedge \ell_0/(2s_{\mathrm{dens}}^+\wedge \ell_0+1)} + \min_{\boldsymbol h \in \mathcal G_2^N}\big(\mathcal B_{\boldsymbol h}^N(\gamma)(t,a)^2+\mathsf V_{\boldsymbol h}^N\big) +\delta_N.
\end{equation}
Moreover, by definition of $s_L$ involved in \eqref{eq def sdeath}, we have 
$$\min_{\boldsymbol h \in \mathcal G_2^N}\big(h_1^{2\min(\gamma,\delta) \wedge \ell_0}+h_2^{2\min(\alpha,\beta,\gamma+1,\delta) \wedge \ell_0} + (\log N)^2N^{-1}h_1^{-1}h_2^{-1}\big) \lesssim \Big(\frac{(\log N)^2}{N}\Big)^{2s_L \wedge \ell_0/(2s_L \wedge \ell_0+1)}$$
and likewise, by definition of $s_U$ involved in \eqref{eq def sdeath}, we have
$$\min_{\boldsymbol h \in \mathcal G_2^N}\big( h_1^{2\min(\gamma,\delta) \wedge \ell_0}+h_2^{2\min(\gamma,\delta) \wedge \ell_0} + (\log N)^2N^{-1}h_1^{-1}h_2^{-1}\big) \lesssim  \Big(\frac{(\log N)^2}{N}\Big)^{2s_U \wedge \ell_0/(2s_U \wedge \ell_0+1)}.$$
Therefore, putting together \eqref{eq biais gamma} and \eqref{eq variance gamma} and using the definition of $s_{\mathrm{death}}$ in \eqref{eq def sdeath} we obtain
$$\min_{\boldsymbol h \in \mathcal G_2^N}\big(\mathcal B_{\boldsymbol h}^N(\gamma)(t,a)^2+\mathsf V_{\boldsymbol h}^N\big) \lesssim \Big(\frac{(\log N)^2}{N}\Big)^{s_{\mathrm{death}}^+(t,a)\wedge \ell_0/(2s_{\mathrm{death}}^+(t,a)\wedge \ell_0+1)}.$$
Since $s_{\mathrm{dens}}^+ \geq s_{\mathrm{death}}$, inequality \eqref{eq recap} becomes
%\begin{align*}
%$$
%\min_{\boldsymbol h \in \mathcal G_2^N}\big(\mathcal B_{\boldsymbol h}^N(\gamma)(t,a)^2+\mathsf V_{\boldsymbol h}^N\big) \lesssim 
%\min_{\boldsymbol h \in \mathcal G_2^N}\big(h_1^{2\min(\gamma,\delta + 1) \wedge \ell_0}+h_2^{2\min(\alpha,\beta,\gamma,\delta) \wedge \ell_0} + (\log N)^2N^{-1}h_1^{-1}h_2^{-1}\big)$$
%and if $(t,a) \in {\mathcal D}_L$,
%$$
%\min_{\boldsymbol h \in \mathcal G_2^N}\big(\mathcal B_{\boldsymbol h}^N(\gamma)(t,a)^2+\mathsf V_{\boldsymbol h}^N\big) \lesssim 
%\min_{\boldsymbol h \in \mathcal G_2^N}\big( h_1^{2\min(\gamma,\delta + 1) \wedge \ell_0}+h_2^{2\delta \wedge \ell_0} + (\log N)^2N^{-1}h_1^{-1}h_2^{-1}\big).$$
$$\E\big[\big(\mu_{\star}^N(t,a)_{\varpi}-\mu(t,a)\big)^2\big] \lesssim \Big(\frac{(\log N)^2}{N}\Big)^{s_{\mathrm{death}}(t,a)\wedge \ell_0/(2s_{\mathrm{death}}(t,a)\wedge \ell_0+1)} + \delta_N.$$
Since the estimate is uniform in $(b,\mu,g_0)$ and $\delta_N \lesssim N^{-1}$, this completes the proof of Theorem \ref{th adapt minimax mu}.

\section{Appendix} \label{sec appendix}
\subsection{Proof of Proposition \ref{prop concentration abstract}}
\subsubsection*{Preliminaries}
For $x\geq 0$ and $q \geq 1$, define $\psi_q(x)  =\exp(x^q)-1$. Let also
$$ \|\xi(f)\|_{\psi_q}  =\inf\big\{c>0,\;\E\big[\psi_q(c^{-1}\xi(f))\big] \leq 1\big\}$$
and 
$$D  =\text{diam}_d(\mathcal F) = \sup_{f,g\in \mathcal F}d(f,g).$$
%\begin{align*}
% & \psi_q(x)  =\exp(x^q)-1 \; \text{for} \; q \geq 1\;\text{and}\,x \geq 0,\\
% & \|\xi(f)\|_{\psi_q}  =\inf\big\{c>0,\;\E\big[\psi_q(c^{-1}\xi(f))\big] \leq 1\big\}, \\
%& D  =\text{diam}_d(\mathcal F) = \sup\{d(f,g), f,g\in \mathcal F\}.
%\end{align*}
\begin{prop}[Theorem 11.2, Eq. $(11.4)$ p. 302 in \cite{LEDOUXTALAGRAND}] \label{prop: Ledoux-Talagrand} \label{prop LT}
In the setting of Proposition \ref{prop concentration abstract}, 
%By Theorem 11.2 in \cite{LT},  if 
if $\|\xi(f)-\xi(g)\|_{\psi_q} \leq d(f,g)$
and
$E = \int_0^D\psi_q^{-1}\big(\mathcal N(\mathcal F,d,\epsilon)\big)d\epsilon < \infty,$
then
$$\PP\big(\sup_{f \in \mathcal F}|\xi(f)| \geq 8(E+u)\big) \leq \psi_q(u/D)^{-1},$$
provided $\xi(f_0)=0$ for some $f_0 \in \mathcal F$.
\end{prop}
% see $(11.4)$ p. 302 in \cite{LT}. 
We also recall the following bound based on a classical Chernoff bound argument, proof of which we omit. For $x \geq 0$, let $\widetilde \rho(x) = (1+x)\log (1+x)-x$.
\begin{lem} \label{lem: big Chernoff}
Let $X$ be a non-negative random variable on some probability space equipped with a probability measure $\mathbb Q$. If, for some $k_1,k_2, k_3>0$, we have
$$\E_{\mathbb Q}\big[e^{\lambda X}\big] \leq k_1\exp\big(k_2\rho(k_3 \lambda)\big)\;\;\text{for every}\;\;\lambda \geq 0,$$
then, for every $u\geq 0$,
$$\mathbb Q\big(X \geq u\big) \leq k_1\exp\big(-k_2\widetilde \rho(u/k_2k_3)\big).$$
\end{lem}
\subsubsection*{Proof of Proposition \ref{prop concentration abstract}} Thanks to Proposition \ref{prop LT}, all we need is an upper bound for $\|\xi(f)-\xi(g)\|_{\psi_1}$. Let $\kappa > 0$. We plan to apply Lemma \ref{lem: big Chernoff} with $\mathbb Q = \PP\big(\cdot\,|\,\mathcal A(\kappa)\big)$, $X=|\xi(f)-\xi(g)|$, $k_1=2\PP\big(\mathcal A(\kappa)\big)$, $k_2 = c_1(1+\kappa)$, $k_3=c_2d(f,g)$ and using \eqref{controle 2}. It follows that for every $u \geq 0$
\begin{equation} \label{eq: maj proba chern cond}
\PP\big(|\xi(f)-\xi(g)| \geq u\big) \leq 2\exp\big(-c_1(1+\kappa)\widetilde \rho\big(u/c_1(1+\kappa)c_2d(f,g)\big)\big)+\PP(\mathcal A(\kappa)^c).
\end{equation}
Now, let $c>0$. We have
\begin{align*}
\E\big[\psi_1(c^{-1}|\xi(f)-\xi(g)|)\big] & = \E\big[\exp (c^{-1}|\xi(f)-\xi(g)|)\big] -1 \\
& = \int_1^\infty \PP\big(\exp(c^{-1}|\xi(f)-\xi(g)|) \geq \kappa\big)d\kappa \\
& = \int_0^\infty \PP\big(|\xi(f)-\xi(g)| \geq c \kappa \big)e^\kappa d\kappa \\
& \leq 2\int_0^\infty \exp\Big(-c_1(1+\kappa)\widetilde \rho\big(c\kappa/c_1(1+\kappa)c_2d(f,g)\big)\Big)e^\kappa d\kappa + \tfrac{1}{2},
%\PP(\mathcal A(\kappa)^c)\big)e^\kappa d\kappa
\end{align*}
where we applied \eqref{eq: maj proba chern cond} with $u=c\kappa$ and used \eqref{controle 1} for bounding the second term. 
%By picking
%$c=\varpi d(f,g)$ for some $\varpi >0$, one may check in a straighforward yet tedious computation that  
It suffices then to pick $\varpi = \varpi(c_1,c_2) >0$ such that
\begin{equation} \label{eq: controle integrale}
2\int_0^\infty \exp\big(-c_1(1+\kappa)\widetilde \rho\big(\varpi(c_1,c_2) \kappa/c_1(1+\kappa)c_2 \big)\big)e^\kappa d\kappa \leq \tfrac{1}{2}.
\end{equation}
%a choice which is obviously possible.
Using \eqref{eq: controle integrale} in the previous estimate with $c= \varpi d(f,g)$, we obtain 
$$\E\big[\psi_1(\varpi d(f,g)^{-1}|\xi(f)-\xi(g)|)\big] \leq 1$$
and therefore 
$$\|\xi(f)-\xi(g)\|_{\psi_1} \leq \varpi d(f,g) = \widetilde d(f,g),$$
say. We may then apply Proposition \ref{prop: Ledoux-Talagrand} with $\widetilde d$ instead of $d$ and Proposition \ref{prop concentration abstract} follows. 

\begin{rk} \label{rk: le choix de varpi}
In \eqref{eq: controle integrale}, we may choose $\varpi(c_1,c_2) = k\sqrt{c_1}c_2$  for some $k>0$ that does not depend on $c_1$ nor $c_2$. Indeed, since $\widetilde \rho(x) \geq \tfrac{1}{4}x^2$ for $x \in [0,1]$, given the ansatz $\varpi(c_1,c_2) = k\sqrt{c_1}c_2$ in \eqref{eq: controle integrale}, it suffices to show the existence of $k$ satisfying $k \leq \sqrt{c_1}$
% \in (0,\sqrt{c_1}]$ 
and 
\begin{equation} \label{eq: controle integrale bis}
\int_0^\infty \exp\big(-\frac{k^2}{4}\frac{\kappa^2}{1+\kappa}+\kappa\big) d\kappa \leq \frac{1}{4}.
\end{equation}
One can check that \eqref{eq: controle integrale bis} holds for large enough $k$. A rough bound is $k = 2\sqrt{77}$, and therefore $c_1 \geq 308$ ensures the requirement $k \leq \sqrt{c_1}$.
%but since $c_1$ can be taken arbitrarily large, the condition $k \leq \sqrt{c_1}$ can always be fullfilled. 
\end{rk}
%in order to obtain \eqref{eq: controle integrale}.
%\newpage

\subsection{Proof of Proposition \ref{prop: KleinRio}} \label{proofKleinRio}
We have $Z_0^N = N^{-1}\sum_{i = 1}^N\delta_{A_i}$, where the $A_i$ are independent with common distribution $g_0(a)da$. Define 
$\mathcal F_{w_2} = \{f=w_2(-\cdot)g,g\in \mathcal F\}$.
We claim that
\begin{equation} \label{eq: cond Comte 1}
\sup_{f \in \mathcal F_{w_2}}|f|_\infty \lesssim |w_2|_{\infty}, 
\end{equation}
\begin{equation} \label{eq: cond Comte 2}
\E\big[\sup_{f \in \mathcal F_{w_2}}\big|\langle Z_0^N,f \rangle - \E[\langle Z_0^N,f\rangle]\big|\big] \lesssim N^{-1/2}|w_2|_{2}, 
\end{equation}
and
\begin{equation} \label{eq: cond Comte 3}
N^{-1}\sup_{f \in \mathcal F_{w_2}}\sum_{i = 1}^N \mathrm{Var}\big(f(A_i)\big) \lesssim |w_2|_2^2.
\end{equation}
The estimates \eqref{eq: cond Comte 1} and \eqref{eq: cond Comte 3} are straightforward. We turn to \eqref{eq: cond Comte 2}. Write $f=w_2(-\cdot)g$ for $f \in \mathcal F_{w_2}$, with $g \in \mathcal F$. Adding and substracting $\int_0^\infty w_2(-a)g(A_i)g_0(a)da$, we have
\begin{equation} \label{eq: decomp bias}
\langle Z_0^N,f \rangle - \E[\langle Z_0^N,f\rangle] 
  =  N^{-1}\sum_{i = 1}^N(w_2(-A_i)-\E[w_2(-A_i)])g(A_i)+\big(\int_0^\infty w_2(-a)g_0(a)da\big)\nu_{w_2}^N(g),
\end{equation}
with $\nu_{w_2}^N(g) = N^{-1}\sum_{i = 1}^N\big( g(A_i)-\E_{w_2}[g(A_i)]\big)$ and where $\E_{w_2}$ denotes expectation under a bias sampling proportional to $w_2(-\cdot)$. Since $\mathcal F$ is stable under $g \mapsto -g$ and uniformly bounded, we have
$$\sup_{g \in \mathcal F}\big|N^{-1}\sum_{i = 1}^N(w_2(-A_i)-\E[w_2(-A_i)])g(A_i)\big| \lesssim N^{-1}\sum_{i = 1}^N(w_2(-A_i)-\E[w_2(-A_i)]).$$
By Cauchy-Schwarz inequality, it follows that 
$$\E\big[\sup_{g \in \mathcal F}\big|N^{-1}\sum_{i = 1}^N(w_2(-A_i)-\E[w_2(-A_i)])g(A_i)\big|\big] \leq N^{-1/2}\mathrm{Var}\big(w_2(-A_i)\big)^{1/2} \lesssim N^{-1/2}|w_2|_2.$$
In the same way as in the proof of Proposition \ref{prop concentration abstract}, with $\psi_2(x)=e^{x^2}-1$, we show using the tools in \cite{LEDOUXTALAGRAND}], p.322, that for $g_1,g_2 \in \mathcal F$, we have
$\|\nu_{w_2}^N(g_1)-\nu_{w_2}^N(g_2)\|_{\psi_2} \lesssim N^{-1/2}|g_1-g_2|_\infty$ by Hoeffding inequality. It follows that $\E[\sup_{g \in \mathcal F}\nu_{w_2}^N(g)] \lesssim N^{-1/2}$. Noticing that the term $\int_0^\infty w_2(-a)g_0(a)da$ in front of $\nu_{w_2}^N(g)$ in \eqref{eq: decomp bias} is of order $|w_2|_2$ enables us to conclude the proof of \eqref{eq: cond Comte 2}. Noting that $\mathcal W_{w_2}(\mathcal F)_0 = \sup_{f \in \mathcal F_{w_2}}\langle Z_0^N,f \rangle$, the proof of Proposition \ref{prop: KleinRio} is now a consequence of Lemma 6.1. in Comte {\it et al.} \cite{COMTEDEDECKERTAUPIN} on the concentration properties of $\langle Z_0^N,f \rangle$, based on the bounds \eqref{eq: cond Comte 1}, \eqref{eq: cond Comte 2} and
\eqref{eq: cond Comte 3}. We omit the details.

\subsection{Proof of Proposition \ref{P first smoothness}}
\label{proof of prop P first smoothness}
The behaviour of the solution $\xi_t(da) = g(t,a)da$ of the McKendricks Von Voester transport equation is studied in numerous textbooks, see {\it e.g.} \cite{PERTHAME}. The proof goes along a classical representation of $g$ in terms of an auxiliary function solution to a certain renewal equation that enables one to study the pointwise smoothness of $(t,a) \mapsto g(t,a)$.

\subsubsection*{Preliminaries} 
We start with the following technical result, which is merely an observation: 
\begin{lem} \label{lem reg}
If for some $\sigma,\tau>0$ and for every $(t,a) \in {\mathcal D}$ we have $f \in \mathcal H^{\sigma,\tau}(t,a)$, then, for every $(t',a') \in \mathcal D$, 
\begin{itemize}
\item[(i)]
$u \mapsto \int_0^u f(s,u)ds \in \mathcal H^{\min(\sigma+1, \tau)}(a')$,
\item[(ii)] $u \mapsto \int_0^{t'} f(s, u+s)ds \in \mathcal H^{\max(\sigma\wedge(\tau+1), \tau)}(a')$.
\end{itemize}
\end{lem}
\begin{proof} Property (i) is straightforward. To obtain (ii),  we first write 
$$G_{t'}(u) = \int_0^{t'} f(s, u+s)ds = \int_u^{u+t'} f(s-u,s)ds= \int_u^{u+t'} \widetilde f(s,u)ds$$
%' = \int_0^{u+t}\widetilde f(s',u)ds'-\int_0^uf(s',u)ds'$$
with $\widetilde f(s,u) = f(s-u,s)$, so that $\widetilde f \in \mathcal H^{\min(\sigma,\tau),\sigma}(t'+a',a')$ for every $(t',a') \in \mathcal D$. Writing
$$\int_u^{u+t'} \widetilde f(s,u)ds = \int_0^{u+t'} \widetilde f(s,u)ds - \int_0^{u} \widetilde f(s,u)ds$$
an applying (i), we obtain $u \mapsto \int_0^{u} \widetilde f(s,u)ds \in \mathcal H^{\min(\min(\sigma, \tau) +1, \sigma)}(a') =  \mathcal H^{\min(\sigma, \tau+1)}(a')$ for every $a' \in \R_+$. Similarly,  $u \mapsto \int_u^{u+t} \widetilde f(s,u)ds \in \mathcal H^{\min(\sigma, \tau+1)}(a')$ therefore $G_{t'} \in \mathcal H^{\tau}(a') \in \mathcal H^{\min(\sigma, \tau+1)}$. But since $G_{t'} \in \mathcal H^{\tau}(a')$ trivially holds, we have in fact $G_{t'} \in \mathcal H^{\tau}(a')\cap \;\mathcal H^{\min(\sigma, \tau+1)}(a') =  \mathcal H^{\max(\sigma\wedge(\tau+1), \tau)}(a')$.
\end{proof}

%We need some notation. Let
%$$L_{b,\mu}(t,a) = b(t,a)\exp\big(-\int_{t-a}^t \mu(s,s-t+a)\big),\;\;\text{for}\;\;(t,a) \in \mathcal D_L,$$
%$$M_{b,\mu,g_0}(t) = \int_0^\infty b(t,t+u)g_0(u)\exp\big(-\int_0^t \mu(s,u+s)ds\big)du,\;\;\text{for}\;\;t \in [0,T]$$
%and define $B_{b,\mu,g_0}:[0,T] \rightarrow \R_+$ as the solution to the integral equation
%\begin{equation} \label{eq integral equation}
%B_{b,\mu,g_0}(t)  = M_{b,\mu,g_0}(t)+ \int_0^t B_{b,\mu,g_0}(a)L_{b,\mu}(t,t-a)da,\;\;\text{for every}\;\;t\in [0,T].
%\end{equation}
%Note that Assumptions \ref{H basic} and \ref{H smoothness basic} ensure the existence and uniqueness of \eqref{eq integral equation}.
%Define next 
%\begin{equation} 
%\label{eq def weak}
%g(t,a) = 
%\left\{
%\begin{array}{lll}
%g_0(a-t)\exp\big(-\int_0^t \mu(s,a-t+s)ds\big) & \text{on} & \mathcal D_U \\ \\
%B_{b,\mu,g_0}(t-a)\exp\big(-\int_{t-a}^t \mu(s,a+s-t)ds\big) & \text{on} & \mathcal D_L 
%\end{array}
%\right.
%\end{equation}
%and set for instance $g(t,a) = 0$ on $\{a=t\}$. It is not difficult to see that $g$ defined in \eqref{eq def weak} is a weak solution to the McKendricks Von Voester equation \eqref{McKendrick}.

\subsubsection*{Completion of proof of Proposition \ref{P first smoothness}} 

For $\sigma,\tau>0$, we write $f \in \mathcal H^{\sigma,\tau}$ if $f \in \mathcal H^{\sigma,\tau}(t,a)$ for every $(t,a) \in {\mathcal D}$.\\

\noindent {\bf Step 1)} For fixed $a$, we have $(s,t) \mapsto \mu(s,a-t+s) \in \mathcal H^{\gamma\wedge\delta,\delta}$ hence by (i) of Lemma \ref{lem reg} we have $t\mapsto \int_0^t \mu(s,a-t+s)ds \in \mathcal H^{\min((\gamma\wedge\delta)+1, \delta)} = \mathcal H^{\min(\gamma+1,\delta)}$. For fixed $t$, $(s,a) \mapsto \mu(s,a-t) \in \mathcal H^{\gamma,\delta}$ holds true, hence $a \mapsto \int_0^t \mu(s,a-t+s)ds \in \mathcal H^{\min(\gamma,\delta+1)\vee \delta}$ by (ii) of Lemma \ref{lem reg}. It follows that
$$(t,a) \mapsto \exp\big(-\int_0^t \mu(s,a-t+s)ds\big) \in \mathcal H^{\min(\gamma+1,\delta), \max(\gamma \wedge (\delta+1),\delta)}.$$
Also $(t,a) \mapsto g_0(a-t) \in \mathcal H^{\nu,\nu} \subset \mathcal H^{\min(\gamma+1,\delta), \max(\gamma \wedge (\delta+1),\delta)}$ since $\nu \geq \max(\gamma,\delta)+1$ hence the result on $\mathcal D_U$. In the same way, on $\mathcal D_L$, we have $t\mapsto \int_0^t \mu(s,a-t+s)ds \in \mathcal H^{\min(\gamma+1,\delta)}$ and $t\mapsto \int_0^{t-a} \mu(s,a-t+s)ds \in \mathcal H^{\min(\gamma+1,\delta)}$ by (i) of Lemma \ref{lem reg} hence 
\begin{equation} \label{eq both reg}
t\mapsto \int_{t-a}^t \mu(s,a-t+s)ds \in \mathcal H^{\min(\gamma+1,\delta)}.
\end{equation}
Moreover, $\int_{t-a}^t\mu(s,a+s-t)ds = -\int_0^a\mu(s+t,a+s)ds$ and $(s,a) \mapsto \mu(s,a+s-t) \in \mathcal H^{\gamma\wedge\delta, \delta}$ for fixed $t$, therefore
\begin{equation} \label{eq both reg bis}
a\mapsto \int_{t-a}^t \mu(s,a-t+s)ds \in \mathcal H^{\min(\gamma+1,\delta)}
\end{equation}
by (i) of Lemma \ref{lem reg} likewise. Putting together \eqref{eq both reg} and \eqref{eq both reg bis}, we conclude
\begin{equation} \label{eq difficult reg}
(t,a) \mapsto \exp\big(-\int_{t-a}^t \mu(s,a-t+s)ds\big) \in  \mathcal H^{\min(\gamma+1,\delta),\min(\gamma+1,\delta)}.
\end{equation}
%$(t,a) \mapsto \int_{t-a}^t \mu(s,a+s-t)ds \in \mathcal H^{(\gamma+1)\wedge \delta, (\gamma+1)\wedge \delta}$ and 
The property $b \in \mathcal H^{\alpha,\beta}$ together with \eqref{eq difficult reg}
entail $L_{b,\mu} \in \mathcal H^{\min(\alpha, \gamma+1,\delta), \min(\beta, \gamma+1,\delta)}$ hence
$$(t,a) \mapsto L_{b,\mu}(t,t-a) \in \mathcal H^{\min(\alpha, \gamma+1,\delta), \min(\alpha, \beta, \gamma+1,\delta)}$$
and
\begin{equation} \label{eq reg univ 1}
t \mapsto \int_0^t B_{b,\mu,g_0}(a)L_{b,\mu}(t,t-a)da \in \mathcal H^{\min(\alpha, \beta, \gamma+1,\delta)},
\end{equation}
follows by (i) of Lemma \ref{lem reg}. Plainly,
\begin{equation} \label{eq reg univ 2}
t \mapsto M_{b,\mu,g_0}(t) \in \mathcal H^{\min(\alpha,\beta, \gamma+1,\delta+1)}
\end{equation}
and putting together \eqref{eq reg univ 1} and \eqref{eq reg univ 2}, we conclude  
\begin{equation} \label{eq reg B}
t\mapsto B_{b,\mu,g_0}(t) \in \mathcal H^{\min(\alpha, \beta, \gamma+1,\delta)}.
\end{equation}
hence $(t,a) \mapsto B_{b,\mu,g_0}(t-a) \in \mathcal H^{\min(\alpha, \beta, \gamma+1,\delta), \min(\alpha, \beta, \gamma+1,\delta)}$.
The result of Proposition \ref{P first smoothness} (i) follows.\\

\noindent {\bf Step 2)}
Writing $(t',a') = \varphi(t,a) = (t,t-a)$, the representation \eqref{eq def weak} now becomes
\begin{equation} 
\label{eq def weak bis}
g(t,a) = 
%\widetilde g \circ \varphi(t,a) =
\widetilde g(t',a') = 
\left\{
\begin{array}{lll}
g_0(-a')\exp\big(-\int_0^{t'} \mu(s,s-a')ds\big) & \text{on} & \widetilde{\mathcal D}_U \\ \\
B_{b,\mu,g_0}(a')\exp\big(-\int_{a'}^{t'} \mu(s,s-a')ds\big) & \text{on} & \widetilde{\mathcal D}_L. 
\end{array}
\right.
\end{equation}
On $\widetilde{\mathcal D}_U$, we have 
$t' \mapsto \int_0^{t'} \mu(s,s-a')ds \in \mathcal H^{\min(\gamma,\delta) +1}$ and $a'  \mapsto \int_0^{t'} \mu(s,s-a')ds \in \mathcal H^{\max(\gamma \wedge (\delta+1),\delta)}$
by (ii) of Lemma \ref{lem reg} for the second case, hence
$$(t',a') \mapsto \exp\big(-\int_0^{t'} \mu(s,s-a')ds\big) \in \mathcal H^{\min(\gamma+1,\delta+1), \min(\gamma,\delta +1)}.$$
Since $(t',a') \mapsto g_0(-a') \in \mathcal H^{\infty, \nu}$ 
%and $(t',a') \mapsto \int_0^t \mu(s,s-a')ds \in \mathcal H^{(\gamma+1)\wedge (\delta+1),\delta}$ 
hence the result since $\nu \geq \delta$. Similarly, on $\widetilde{\mathcal D}_L$, by \eqref{eq reg B}, we have 
$(t',a') \mapsto B_{b,\mu,g_0}(a') \in \mathcal H^{\infty, \min(\alpha, \beta, \gamma+1,\delta)}$ and the same arguments as before yield
$$(t',a') \mapsto \int_{a'}^t \mu(s,s-a')ds \in \mathcal H^{\min(\gamma+1,\delta+1),\max(\gamma \wedge (\delta+1),\delta)}.$$
Combining these two properties gives the result on $\widetilde{\mathcal D}_L$ and completes (ii) of Proposition \ref{P first smoothness}.

\subsection{Further estimates on the McKendricks Von Foerster equation}  
\label{app analyse}

The following result is a classical estimate of the renewal equation, see for instance \cite{PERTHAME}. 
\begin{lem}[\cite{PERTHAME}, Theorem 2.2. in Chapter 2] \label{estimate g}
Work under Assumptions \ref{H basic}. We have
$$\sup_{0 \leq t \leq T}\int_0^\infty g(t,a)da \leq \int_0^\infty g_0(a)da\, e^{|b-\mu|_\infty T}$$
and
$$|g|_\infty \leq \max\big(|g_0|_\infty, |b|_\infty \sup_{0 \leq t \leq T}\int_0^\infty g(t,a)da\big)$$
%and
\end{lem}

\subsubsection*{Proof of Lemma \ref{lem min g}} \label{proof lemma assumption minoration g}

On $\mathcal D_U$, by \eqref{eq def weak} in the proof of Proposition \ref{P first smoothness}, we have 
$$g(t,a) = 
g_0(a-t)\exp\big(-\int_0^t \mu(s,a-t+s)ds\big) \geq \delta(t,a)e^{-|\mu|_\infty T}
$$
by \eqref{eq condit mino g U} of Assumption \ref{assumption minoration g}. On $\mathcal D_L$, \eqref{eq def weak} yields the representation
$$g(t,a) = B_{b,\mu,g_0}(t-a)\exp\big(-\int_{t-a}^a \mu(s,a+s-t)ds\big) \geq  B_{b,\mu,g_0}(t-a)e^{-|\mu|_\infty t}$$
and  by \eqref{eq integral equation}, we further have
\begin{align*}
B_{b,\mu,g_0}(t-a)   & \geq M_{b,\mu,g_0}(t-a) \\
& = \int_0^\infty b(t-a,t-a+u)g_0(u)\exp\big(-\int_0^{t-a} \mu(s,u+s)ds\big)du \\
& \geq \delta |\mathcal U_{(t,a)}| e^{-|\mu|_\infty(t-a)}
\end{align*}
by \eqref{eq condit mino g L} of Assumption \ref{assumption minoration g}. The proof of Lemma \ref{lem min g} is complete.\\

\noindent {\bf Acknowledgements} We gratefully acknowledge insightful comments and discussions with N. Champagnat, N. El Karoui, O. Lepski and V. C. Tran.

\bibliographystyle{plain}       % (uses file "plain.bst")
\bibliography{biblioBHJ}           % expects file "biblioBHJ.bib"

\begin{thebibliography}{10}

\bibitem{BERAN1981}
Rudolf Beran.
\newblock Nonparametric regression with randomly censored survival data.
\newblock Technical report, Technical Report, Univ. California, Berkeley, 1981.

\bibitem{BIERENS}
Herman~J. Bierens.
\newblock {\em Topics in advanced econometrics}.
\newblock Cambridge University Press, Cambridge, 1994.
\newblock Estimation, testing, and specification of cross-section and time
  series models.

\bibitem{BHO}
S.~Val\`ere Bitseki~Penda, Marc Hoffmann, and Ad\'ela\"\i~de Olivier.
\newblock Adaptive estimation for bifurcating {M}arkov chains.
\newblock {\em Bernoulli}, 23(4B):3598--3637, 2017.

\bibitem{BOUMEZOUED2016}
Alexandre Boumezoued.
\newblock {Improving {{HMD}} mortality estimates with {{HFD}} fertility data}.
\newblock {\em To appear in the North American Actuarial Journal}, 2016.

\bibitem{BHJ2018APP}
Alexandre Boumezoued, Marc Hoffmann, and Paulien Jeunesse.
\newblock A new inference strategy for general population mortality tables.
\newblock {\em Preprint hal-01773665}, 2018.

\bibitem{BRUNEL2008}
Elodie Brunel, Fabienne Comte, and Agathe Guilloux.
\newblock Estimation strategies for censored lifetimes with a lexis-diagram
  type model.
\newblock {\em Scandinavian Journal of Statistics}, 35(3):557--576, 2008.

\bibitem{CAIRNS1}
Andrew J.~G. Cairns, David Blake, Kevin Dowd, Guy~D. Coughlan, David Epstein,
  Alen Ong, and Igor Balevich.
\newblock A quantitative comparison of stochastic mortality models using data
  from {E}ngland and {W}ales and the {U}nited {S}tates.
\newblock {\em N. Am. Actuar. J.}, 13(1):1--35, 2009.

\bibitem{CAIRNS2016}
Andrew~JG Cairns, David Blake, Kevin Dowd, and Amy~R Kessler.
\newblock Phantoms never die: living with unreliable population data.
\newblock {\em Journal of the Royal Statistical Society: Series A (Statistics
  in Society)}, 179(4):975--1005, 2016.

\bibitem{TRAN2008}
St{\'e}phan Cl{\'e}men{\c{c}}on, Viet Chi~Tran, and Hector De~Arazoza.
\newblock A stochastic {SIR} model with contact-tracing: large population
  limits and statistical inference.
\newblock {\em Journal of Biological Dynamics}, 2(4):392--414, 2008.

\bibitem{COMTEDEDECKERTAUPIN}
Fabienne Comte, J\'er{\^o}me Dedecker, and Marie-Luce Taupin.
\newblock Adaptive density deconvolution with dependent inputs.
\newblock {\em Math. Meth. Statist.}, 17:87--112, 2008.

\bibitem{COMTE2011}
Fabienne Comte, St{\'e}phane Ga{\"\i}ffas, and Agathe Guilloux.
\newblock Adaptive estimation of the conditional intensity of marker-dependent
  counting processes.
\newblock In {\em Annales de l'Institut Henri Poincar{\'e}, Probabilit{\'e}s et
  Statistiques}, volume~47, pages 1171--1196. Institut Henri Poincar{\'e},
  2011.

\bibitem{DABROWSKA1987}
Dorota~M Dabrowska.
\newblock Non-parametric regression with censored survival time data.
\newblock {\em Scandinavian Journal of Statistics}, pages 181--197, 1987.

\bibitem{DHKR1}
Marie Doumic, Marc Hoffmann, Nathalie Krell, and Lydia Robert.
\newblock Statistical estimation of a growth-fragmentation model observed on a
  genealogical tree.
\newblock {\em Bernoulli}, 21(3):1760--1799, 2015.

\bibitem{DUDLEY}
R.~M. Dudley.
\newblock Universal {D}onsker classes and metric entropy.
\newblock {\em Ann. Probab.}, 15(4):1306--1326, 1987.

\bibitem{FOURNIERMELEARD}
Nicolas Fournier and Sylvie M\'el\'eard.
\newblock A microscopic probabilistic description of a locally regulated
  population and macroscopic approximations.
\newblock {\em Ann. Appl. Probab.}, 14(4):1880--1919, 2004.

\bibitem{GINENICKL}
Evarist Gin\'e and Richard Nickl.
\newblock {\em Mathematical foundations of infinite-dimensional statistical
  models}.
\newblock Cambridge Series in Statistical and Probabilistic Mathematics, [40].
  Cambridge University Press, New York, 2016.

\bibitem{GOLDENSHLUGERLEPSKI1}
Alexander Goldenshluger and Oleg Lepski.
\newblock Universal pointwise selection rule in multivariate function
  estimation.
\newblock {\em Bernoulli}, 14(4):1150--1190, 2008.

\bibitem{GOLDENSHLUGERLEPSKI2}
Alexander Goldenshluger and Oleg Lepski.
\newblock Bandwidth selection in kernel density estimation: oracle inequalities
  and adaptive minimax optimality.
\newblock {\em Ann. Statist.}, 39(3):1608--1632, 2011.

\bibitem{HKPT}
Wolfgang H\"ardle, Gerard Kerkyacharian, Dominique Picard, and Alexander
  Tsybakov.
\newblock {\em Wavelets, approximation, and statistical applications}, volume
  129 of {\em Lecture Notes in Statistics}.
\newblock Springer-Verlag, New York, 1998.

\bibitem{HFD}
HFD.
\newblock The human fertility database. max planck institute for demographic
  research (germany) and vienna institute of demography (austria).

\bibitem{HMD}
HMD.
\newblock The human mortality database.

\bibitem{HOFFMANN2016}
Marc Hoffmann and Ad{\'e}la{\"\i}de Olivier.
\newblock Nonparametric estimation of the division rate of an age dependent
  branching process.
\newblock {\em Stochastic Processes and their Applications}, 126(5):1433--1471,
  2016.

\bibitem{JACODSHIRYAEV}
Jean Jacod and Albert~N. Shiryaev.
\newblock {\em Limit theorems for stochastic processes}, volume 288 of {\em
  Grundlehren der Mathematischen Wissenschaften [Fundamental Principles of
  Mathematical Sciences]}.
\newblock Springer-Verlag, Berlin, second edition, 2003.

\bibitem{KEIDING1990}
Niels Keiding.
\newblock Statistical inference in the lexis diagram.
\newblock {\em Philosophical Transactions of the Royal Society of London A:
  Mathematical, Physical and Engineering Sciences}, 332(1627):487--509, 1990.

\bibitem{KLEINRIO}
Thierry Klein and Emmanuel Rio.
\newblock Concentration around the mean for maxima of empirical processes.
\newblock {\em Ann. Probab.}, 33:1060--1077, 2005.

\bibitem{LACOURMASSARTRIVOIRARD}
Claire Lacour, Pascal Massart, and Vincent Rivoirard.
\newblock Estimator selection: a new method with applications to kernel density
  estimation.
\newblock {\em Sankhya A}, 79(2):298--335, 2017.

\bibitem{LECAM}
Lucien Le~Cam.
\newblock {\em Asymptotic methods in statistical decision theory}.
\newblock Springer Series in Statistics. Springer-Verlag, New York, 1986.

\bibitem{LEDOUXTALAGRAND}
Michel Ledoux and Michel Talagrand.
\newblock {\em Probability in {B}anach spaces}, volume~23 of {\em Ergebnisse
  der Mathematik und ihrer Grenzgebiete (3) [Results in Mathematics and Related
  Areas (3)]}.
\newblock Springer-Verlag, Berlin, 1991.
\newblock Isoperimetry and processes.

\bibitem{LEPSKI}
O.~V. Lepski\u\i.
\newblock A problem of adaptive estimation in {G}aussian white noise.
\newblock {\em Teor. Veroyatnost. i Primenen.}, 35(3):459--470, 1990.

\bibitem{LEPSKIBASIC}
O.~V. Lepski\u\i.
\newblock Asymptotically minimax adaptive estimation. {I}. {U}pper bounds.
  {O}ptimally adaptive estimates.
\newblock {\em Teor. Veroyatnost. i Primenen.}, 36(4):645--659, 1991.

\bibitem{LOCHERBACH2}
E.~L\"ocherbach.
\newblock Likelihood ratio processes for {M}arkovian particle systems with
  killing and jumps.
\newblock {\em Stat. Inference Stoch. Process.}, 5(2):153--177, 2002.

\bibitem{LOCHERBACH1}
Eva L\"ocherbach.
\newblock L{AN} and {LAMN} for systems of interacting diffusions with branching
  and immigration.
\newblock {\em Ann. Inst. H. Poincar\'e Probab. Statist.}, 38(1):59--90, 2002.

\bibitem{LOW}
Mark~G. Low.
\newblock Nonexistence of an adaptive estimator for the value of an unknown
  probability density.
\newblock {\em Ann. Statist.}, 20(1):598--602, 1992.

\bibitem{MCKEAGUE1990}
Ian~W McKeague and Klaus~J Utikal.
\newblock Inference for a nonlinear counting process regression model.
\newblock {\em The Annals of Statistics}, pages 1172--1187, 1990.

\bibitem{MCKENDRICK}
A.G. McKendrick.
\newblock {Application of mathematics to medical problems}.
\newblock {\em Proc. Edin. Math. Soc.}, 54:98--130, 1926.

\bibitem{TRANMELEARD}
Sylvie M\'el\'eard and Viet~Chi Tran.
\newblock Slow and fast scales for superprocess limits of age-structured
  populations.
\newblock {\em Stochastic Process. Appl.}, 122(1):250--276, 2012.

\bibitem{NADARAYA}
\`E.~A. Nadaraja.
\newblock On a regression estimate.
\newblock {\em Teor. Verojatnost. i Primenen.}, 9:157--159, 1964.

\bibitem{NIELSEN1}
Jens~P. Nielsen and Oliver~B. Linton.
\newblock Kernel estimation in a nonparametric marker dependent hazard model.
\newblock {\em Ann. Statist.}, 23(5):1735--1748, 1995.

\bibitem{NIELSEN1995}
Jens~P Nielsen and Oliver~B Linton.
\newblock Kernel estimation in a nonparametric marker dependent hazard model.
\newblock {\em The Annals of Statistics}, pages 1735--1748, 1995.

\bibitem{PERTHAME}
Beno\^\i~t Perthame.
\newblock {\em Transport equations in biology}.
\newblock Frontiers in Mathematics. Birkh\"auser Verlag, Basel, 2007.

\bibitem{RICHARDS2008}
SJ~Richards.
\newblock Detecting year-of-birth mortality patterns with limited data.
\newblock {\em Journal of the Royal Statistical Society: Series A (Statistics
  in Society)}, 171(1):279--298, 2008.

\bibitem{SHORACKWELLNER}
Galen~R. Shorack and Jon~A. Wellner.
\newblock {\em Empirical processes with applications to statistics}, volume~59
  of {\em Classics in Applied Mathematics}.
\newblock Society for Industrial and Applied Mathematics (SIAM), Philadelphia,
  PA, 2009.
\newblock Reprint of the 1986 original [ MR0838963].

\bibitem{TRAN1}
Viet~Chi Tran.
\newblock Large population limit and time behaviour of a stochastic particle
  model describing an age-structured population.
\newblock {\em ESAIM Probab. Stat.}, 12:345--386, 2008.

\bibitem{TSYBAKOV}
Alexandre~B. Tsybakov.
\newblock {\em Introduction to nonparametric estimation}.
\newblock Springer Series in Statistics. Springer, New York, 2009.
\newblock Revised and extended from the 2004 French original, Translated by
  Vladimir Zaiats.

\bibitem{VANDEGEER}
Sara van~de Geer.
\newblock Exponential inequalities for martingales, with application to maximum
  likelihood estimation for counting processes.
\newblock {\em Ann. Statist.}, 23(5):1779--1801, 1995.

\bibitem{VONFOERSTER}
H.~Von~Foerster.
\newblock {\em {The Kinetics of Cellular Proliferation}}.
\newblock Grune \& Stratton, 1959.

\end{thebibliography}

\end{document}